\documentclass{amsart}
\usepackage[utf8]{inputenc}
\usepackage{color}
\usepackage{graphics}

\newcommand{\tx}{{\widetilde{x}}}
\newcommand{\tz}{{\widetilde{z}}}
\newcommand{\p}{\partial}
\newcommand{\dx}{\partial_x}
\newcommand{\dt}{\partial_t}
\newcommand{\dtx}{\partial_{\widetilde{x}}}
\newcommand{\dtz}{\partial_{\widetilde{z}}}
\newcommand{\tzeta}{\widetilde{\zeta}}

\newcommand{\tPhi}{\widetilde{\Phi}}

\newcommand{\tX}{\widetilde{X}}
\newcommand{\cT}{{\mathcal T}}
\newcommand{\cH}{{\mathcal H}}
\newcommand{\cL}{{\mathcal L}}
\newcommand{\cD}{{\mathcal D}}
\newcommand{\cE}{{\mathcal E}}
\newcommand{\cN}{{\mathcal N}}
\newcommand{\cK}{{\mathcal K}}

\newcommand{\wt}{\widetilde}
\newcommand{\lam}{{\lambda}}

\newcommand{\RR}{{\mathbb R}}
\newcommand{\cSt}{{\mathcal S}_{0}}
\newcommand{\cS}{{\mathcal S}}
\newcommand{\na}{{\nabla}}
\newcommand{\D}{{\Delta}}
\newcommand{\G}{{\Gamma}}
\newcommand{\Om}{{\Omega}}
\newcommand{\om}{{\omega}}
\newcommand{\al}{{\alpha}}
\newcommand{\ga}{{\gamma}}
\newcommand{\Lam}{{\Lambda}}
\newcommand{\ka}{{\kappa}}

\newcommand{\ma}{{\mathfrak a}}

\newcommand{\tN}{\widetilde{\cN}}
\newcommand{\ta}{\widetilde{\mathfrak a}}

\newcommand{\tv}{\widetilde{v}}
\newcommand{\tP}{\widetilde{P}}
\newcommand{\tn}{\widetilde{n}}

\newcommand{\ttze}{\widetilde{\dt\zeta}}

\newcommand{\f}{\frac}

\newtheorem{proposition}{Proposition}
\newtheorem{theorem}{Theorem}
\newtheorem{lemma}{Lemma}
\newtheorem{corollary}{Corollary}
\newtheorem{definition}{Definition}
\newtheorem{assumption}{Assumption}

\theoremstyle{remark}
\newtheorem{remark}{Remark}

\begin{document}

\title[ A priori energy estimate in weighted norms]{A priori energy estimate with decay in weighted norms for the water-waves problem with contact points}
\author{Mei Ming}
\address{School of Mathematics and Statistics,Yunnan University, Kunming 615000, P.R.China}
\email{mingmei@ynu.edu.cn}
\date{}

\maketitle

\begin{abstract}
We prove a weighted a priori energy estimate for the  two dimensional water-waves problem with contact points  in the absence of gravity and surface tension. When the surface graph function and its time derivative have some decay near the contact points, we show that  there is corresponding decay for the velocity, the pressure and other quantities in a short time interval. As a result, we have fixed contact points and contact angles. To prove the energy estimate, a conformal mapping is used to transform the equation for the mean curvature into an equivalent equation in a flat strip with some weights. Moreover, the weighted limits at contact points for the velocity, the pressure etc. are  tracked and discussed.   Our formulation can be adapted to deal with more general cases.
\end{abstract}

\tableofcontents

\section{Introduction}

We consider a  two-dimensional domain $\Om$ with two fixed contact points $p_l, p_r$ in the absence of  gravity and surface tension. In fact, the domain 
\[
\Om=\{(x,z)\,|\, b(x)\le z\le \zeta(t, x),\ x_l\le x\le x_r\}
\]
is changing with time, where
\[
\G_t=\{(x,z)\,|\, z=\zeta(t, x),\ x_l\le x\le x_r\}
\]
is the free surface curve denoted by a graph function $\zeta$ and $x_l, x_r$ are two constants.

The bottom is denoted by 
\[
\G_b=\{(x,z)\,|\, z=b(x), \ x_l\le x\le x_r\},
\]
where we assume that $b(x)$ can be smooth enough, and $\G_b$ becomes some line segment near the corners for the sake of simplicity.

The points of intersection $p_l, p_r$ between $\G_t$ and $\G_b$ take place only at $x=x_l, x_r$ (the left and the right), which means
\[
b(x)<\zeta(t, x)\qquad\hbox{when}\quad x_l<x<x_r,\ t\in[0, T]
\]
and $b(x_i)=\zeta(t,x_i)$ for $i=l, r$ and some $T>0$. 

\begin{center}
\resizebox{12cm}{!}{\includegraphics{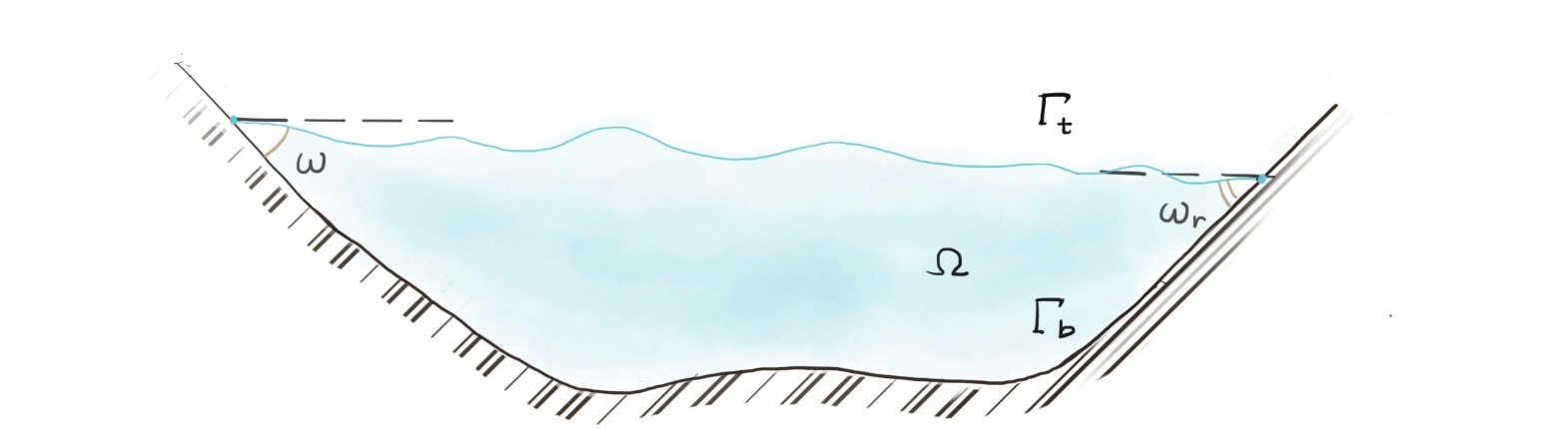}}\\
Figure 1. The physical domain
\end{center}

Meanwhile, we have some more assumptions on the free surface $\G_t$:

\begin{assumption}\label{config}
$\G_t$ are assumed to intersect with $\G_b$ at two fixed contact points when $t\in [0, T]$ for some $T>0$: There are the left contact point $p_l$ with fixed contact angle $\om$  and the right contact point $p_r$ with fixed angle $\om_r$, and both $\om, \om_r\in (0, \pi/2)$. Moreover, we take 
\[
\dx \zeta|_{x_i}=0,\quad i=l, r,\ t\in [0, T]
\]
which means the tangential lines at the contact points on $\G_t$ are both horizontal.
\end{assumption}

We plan to investigate  possible behaviors of ideal incompressible irrotational water waves in the case of fixed contact angles (some decay of the free surface indeed) and zero gravity. To present the water-waves system,  we start with Euler's equation:
\begin{equation}\label{Euler}
\dt v+v\cdot\na v=-\na P\qquad \hbox{in}\quad \Om,\quad t\ge 0
\end{equation}
where $v=(v_1, v_2)^T$ is the velocity field  and $P$ is the pressure to be determined.  We use the notation of material derivative frequently:  $D_t=\dt+v\cdot\na$.

The incompressibility and the irrotationality of the fluid are expressed by 
\begin{equation}\label{incomp-irro}
\na\cdot v=0,\quad \na\times v=0\qquad \hbox{in}\quad \Om,\quad t\ge 0.
\end{equation}

At the bottom, there is no fluid particles transported, which means 
\begin{equation}\label{bottom condi}
v\cdot n_b|_{\G_b}=0\qquad\hbox{for}\quad t\ge 0,
\end{equation}
where $n_b$ is the unit outward normal vector on $\G_b$.

The kinematic condition holds at the free surface, i.e.
\begin{equation}\label{kinematic condi}
\dt\zeta=(1+|\dx\zeta|^2)^{\f12}v\cdot n_t|_{\G_t: z=\zeta}\qquad \hbox{for} \quad t\ge 0,
\end{equation}
where $n_t$ is the unit outward normal vector on $\G_t$.

Moreover, ignoring the effect of surface tension, we have 
\begin{equation}\label{P condi}
P|_{\G_t}=0\qquad \hbox{for}\quad t\ge 0
\end{equation}
without loss of generality.

\bigskip

Before introducing our result, we recall briefly earlier works on the well-posedness of the water-waves problem with smooth boundaries, where one has smooth surface $\G_t$ satisfying 
\[
\G_t\cap \G_b=\emptyset.
\]
There is a very rich literature on various local well-posedess results for this problem, see \cite{Na, Yo1, Yo2, Craig, Wu1, Wu2, Lannes, Ig-Ta, Iguchi} and \cite{ABZ2009, CL, Lin, CS, SZ, SZ2, Sch, BG1, OT1, AM1, ZZ, AL, LannesBook, WZZZ} etc..
Starting around 2009, there are quite a few results focusing on the global well-posedness with small initial data and some long-time behaviors, see \cite{Wu5, GMS, AD, IP, DIP, HIT1, Wang1, Su} etc.. Moreover, there are also some works concerning  ``splash" or ``splat'' singularities on the free surfaces, see \cite{CCFG, CCFGG, CS2} etc..

\bigskip

Compared to the various results for smooth-boundary water waves, there are only very limited results on the mathematical theory of  well-posedness when the boundary  is Lipschitz, i.e. water waves with non-smooth boundaries (non-smooth-boundary water waves). In fact, the mathematical research for this problem just started these years, and there are in general two kinds of problems depending on the location of the corners. 

\medskip

The first kind of problem is the contact-line (contact-point in two dimensional case) problem, where there are intersections between the free surface and the fixed bottom, i.e. 
\[
\G_t\cap\G_b\neq\emptyset.
\]
As a pioneer work for this problem, Alazard, Burq and Zuily  mention in \cite{ABZ3} a special case with zero surface tension,  where the contact angle  with respect to a vertical wall is equal to $\pi/2$ (``the right angle" to keep the Taylor's sign condition, i.e. $\ma =-\na_{n_t}P|_{\G_t}>0$). They use a periodic extension to turn this problem into a classical smooth periodic case.  
de Poyferr\'e \cite{Poyferre} proves an a priori energy estimate in Sobolev spaces in general  n-dimensional bounded domains without surface tension. The contact angles there are assumed to be small to make sure that the solution has sufficient regularity near the corner. 
Under a similar assumption of small contact angles, Ming and Wang  obtain the a priori estimate as well as the local well-posedness  in two-dimensional corner domains  in the presence of  surface tension, see \cite{MW2, MW3}.  Meanwhile,  notice that both \cite{Poyferre} and \cite{MW2, MW3} use the geometric formulation  introduced in \cite{SZ}. 

There are  some works concerning properties of the corresponding Dirichlet-Neumann operator in Lipschitz domains, which is a very important ingredient in the research of water waves. Ming and Wang \cite{MW1} adapts some elliptic theory in corner domains (see for example \cite{PG}) to the desired form with boundary dependence.  Agrawal and Alazard  prove in \cite{AA} some refined Rellich inequalities and  some sharp inequality for the D-N operator.

Besides, there are also some more related works on models of water waves with contact lines. Lannes and M\'etivier \cite{LM} considers the Green-Naghdi system in a beach-type domain, which is a shallow-water model of the water-waves problem. Lannes \cite{Lannes1} works on the floating-body problem and proposes a new formulation  which can be easily  generalized in order to take into account the presence of a floating body.  Lannes and Iguchi \cite{LI} considers initial boundary value problem with a free boundary arising in wave-structure interaction, which contains the floating problem in the shallow water regime. In \cite{LI2023} Iguchi and Lannes prove the local well-posedness  for a shallow water model with a fixed  partially immersed object. For more related works, see \cite{Bocchi, BLW} etc..

For contact-line problems of Navier-Stokes system, Guo and Tice \cite{GT} shows  a priori estimates  in the case of  Stokes equations. Tice and Zheng \cite{TZ} proves the local well-posedness of the contact-line problem in two-dimensional Stokes flow. Moreover, Guo and Tice \cite{GT2} proves an a priori estimate for the contact-line problem of two-dimensional Navier-Stokes flow. For  Darcy's flow, one can see \cite{KM1, KM2} etc..

\medskip

 The second kind of problem focuses on crests or cusps on the free surface, i.e. the surface is Lipschitz. 
 Kinsey and Wu, Wu \cite{WuK, Wu3} prove the a priori estimate and the local well-posedness for two-dimensional gravity water waves in the half plane, where there is a crest on the free surface. The Taylor's sign  degenerates at the crest point,  and the authors flatten the domain with a Riemann mapping and then use some natural weighted norms to finish the energy estimate. The crest angle is less than $\pi/2$ here and doesn't change with respect to time.
 Based on these works, Agrawal \cite{Agra} shows that an initial interface with angled crests remains angled crests in a more general way.  
C\'ordoba, Enciso and Grubic \cite{CEG}  considers the case with cusps and crests on the surface in the absence of gravity, where the angles of these crests are less than $\pi/2$ and change in time. C\'ordoba, Enciso and Grubic   prove in \cite{CEG2023} the local well-posedness for gravity water waves in weighted Sobolev spaces which allow for interfaces with corners. In these two works, the Taylor's sign always degenerates at the crest point on the surface. Meanwhile, the weighted norms used in \cite{CEG2023} are introduced in \cite{KMR} for example, including both so-called homogeneous norm and inhomogeneous norm.
 
 We would like to mention the famous Stokes waves  which can  be dated back  to works by G. Stokes \cite{Stokes, Stokes2}, where  traveling-wave solutions with limit crest angle $2\pi/3$ are discussed. 
 
 \bigskip
 
It is straightforward to see that, the main difference between non-smooth-boundary water waves and smooth-boundary water waves  lies in the corners on the boundary. As a result, the analysis involving the corners (i.e. domain singularities) and  proper adaptions of the existing results of smooth-boundary water waves become  the key points in the non-smooth water waves.

\subsection{Main result of the paper} We focus on the case of two dimensional water waves with  contact points, where both the surface function $\zeta$ and the velocity $v$ have some decay near the corners. The Taylor's sign is also assumed to degenerate at the contact points but it keeps the positive sign with some weight, see Assumption \ref{assump on a}. Meanwhile, our problem is considered in the absence of gravity and surface tension. 

In fact,  it is not clear so far whether the Taylor's sign vanishes (degenerates) at the contact point, so it makes sense for us to consider the special case when the Taylor's sign degenerates at corner points. Moreover, similar weighted conditions for the Taylor's sign can be found in the surface-crest case, see  \cite{WuK, CEG2023}.

\medskip

Our work is inspired by two observations.
The first observation is about the singularity decomposition for solutions to elliptic systems (see \cite{PG} for example) and the homogeneous norm introduced in \cite{KMR} for elliptic estimates. On one hand, for an elliptic system in a corner domain, the singular part of the solution usually takes the form of $r^\lam$, where $r$ is the radius to the corner point, and $\lam$ is some related eigenvalue of this system. 

On the other hand,  according to the elliptic theory in corner domains (see for example \cite{KMR}), proper decay near the corners in boundary conditions can result in corresponding decay near the corners for the solution to an elliptic system. This kind of decay can be described by some homogeneous norms. 
In fact, for a sector $\mathcal K$ with the vertices at the origin,   the so-called homogeneous norm is defined as 
\[
\|u\|_{V^l_\beta(\mathcal K)}=\Big(\sum_{|\alpha|\le l}\int_\mathcal K  r^{2(\beta-l+|\al|)}\big|\na^\alpha_X\,u\big|^2dX\Big)^\frac12
\]
with $r$ the radius with respect to the vertices, $l$ the order of derivatives and $\beta$ the weight. A complete theory of elliptic estimates is developed in \cite{KMR} using this kind of weighted norms (which is similar as the elliptic theory in Sobolev spaces), and the authors  use a conformal mapping to transform the domain into a flat strip sometimes. According to this theory, these weighted norms can describe solutions to elliptic systems with boundary conditions given in the form of  $r^s$ with $s$ some real number (there are some more conditions about the power $s$ essentially avoiding the eigenvalues, see \cite{KMR}). Therefore, as we stated before, {\it proper decay on the boundary can lead to corresponding decay of the solution.}

The second observation is that one can use a conformal mapping $\cT$ to transform a fixed flat strip $\cSt$ (with complex coordinates $\wt Z=\tx+i \tz$, see \eqref{strip}) to the physical domain $\Om$ at time $t$ (see Figure 2 in Section \ref{preliminaries}), which behaves like $e^{\wt Z}$ if we focus on the left corner of $\Om$ and a natural weight (see \eqref{weight function}) appears:
\[
\al(\tx)\approx|\cT'|\approx e^\tx\qquad\hbox{near the left corner where}\quad \tx\rightarrow -\infty.
\] 
For more details about $\cT$, one can see Section \ref{preliminaries}.  Consequently, there will be some coefficients in form of $\cT$ appearing everywhere in the equation of water waves along with derivatives, but a very good point here is that one can rewrite the D-N operator in $\Om$ into the D-N operator in the flat strip with some coefficient (see \eqref{DN transform})! In fact, it is lucky to see that properties for the D-N operator defined in a flat strip are already well-studied, see for example \cite{Lannes, LannesBook}.

\medskip

As a result, it is interesting to see if this kind of homogeneous norms from \cite{KMR} can be adjusted to our problem, where we transform our equation to an equivalent equation in the  flat strip $\cSt$ with some natural weights. As in our previous works \cite{MW2, MW3, MW4}, we use the geometric formulation following Shatah and Zeng \cite{SZ}, so the system of water waves \eqref{Euler}-\eqref{P condi} is rewritten into an equation of the mean curvature $\kappa$ on the free surface, see \eqref{Nk eqn}.  Using a conformal mapping, this equation of $\kappa$ is transformed into a weighted equation on the surface the the flat strip.

Based on the homogeneous norms in \cite{KMR}, we introduce some weighted norms in the strip for the surface function $\zeta$, the velocity $v$ etc.. However, there are some difficulties arising even after we identify all the weighted norms and succeed to prove the corresponding weighted elliptic theory.

\medskip

The first difficulty lies in the estimate for the conformal mapping. The higher-order-derivative estimate is needed in the form of weighted estimate, and the time-dependent estimate is also involved as well. Combining our weighted elliptic theory with some computations in \cite{SWZ}, we succeed in finishing the desired estimates (see Section \ref{T estimate section}).

The second difficulty is about the higher-order terms of velocity $v$ and pressure $P$ on the bottom, while we notice that the surface $\G_t$ and the bottom $\G_b$ intersect at the contact points with different boundary conditions in the related elliptic systems for $v$ and $P$ etc.. Compared to estimates for smooth-boundary water waves or surface-crest water waves, this is a new part to deal with. Meanwhile, compared to the case of contact points with surface tension,  the regularity of the surface $\G_t$ (i.e. the graph function $\zeta$) is lower in our case without surface tension, and  there will be a loss of $1/2$-order derivative for the  velocity as long as one uses the kinematic condition \eqref{kinematic condi}. Therefore, the estimate for $v$ is very tricky. Luckily with our chosen weighted norms in the flat strip, we succeed to prove a delicate higher-order estimate for $v$ (see Section \ref{v estimate section}) combining some tricks  in \cite{SZ} (which considers the case of smooth free surface without any bottom).

Another difficulty lies in dealing with the corner behaviors of the surface, velocity and pressure etc.. In fact, it turns out that we need to trace all the weighted limits of these quantities at the corners. The non-zero weighted limits take place since we assume that the weighted Taylor's sign is bounded below by a positive number, which together with the boundary condition on the bottom of pressure $P$ imply that the weighted limits of pressure at corners are not vanishing. As a result, some more analysis for the corresponding weighted limits is needed. In our settings, we show that the weighted limits for $\dt\zeta$, $v$, $P$ and other quantities are constant with respect to time $t$.  

\medskip

As a result, one can see that our a priori estimate is different from all the known works mentioned before on well-posedness of water waves. So far there are only two works on the case of contact points without surface tension, see \cite{ABZ3} (for a special angle $\pi/2$) and \cite{Poyferre} (a priori estimate for general n dimensions), where the Taylor's sign is assumed to be bounded below and Sobolev spaces are used.  

Moreover, our result is also very different from the weighted formulations in \cite{WuK, CEG}. The first point is that there are contact points and the bottom $\G_b$ in our case, which is different from the case of crests on the free surface from the physical settings. The second point is that higher-order derivative terms appear on the bottom when we consider the energy estimate (as mentioned before), while there are no such terms for the case without fixed bottom. The last point is again about the weighted limits, which arise because of the weighted Taylor's sign and the boundary condition of pressure at the bottom.

\medskip

The main theorem of our paper is stated below in an informal way, while the exact main theorem is Theorem \ref{energy estimate} in Section \ref{energy}. To be more clear,  we  denote by $\tzeta=\zeta\circ \cT|_{\G_t}$, $\ttze=(\dt \zeta)\circ \cT|_{\G_t}$ for the surface graph function $z=\zeta$ with a slight abuse of notation (see Section \ref{equation}). Moreover, we recall the weight function $\al=\al(\tx)$ introduced by the conformal mapping $\cT$. The  energy $\cE(t)$ is defined in the beginning of Section \ref{energy}, which contains weighted norms of $(k+1/2)$-th order derivative of the mean curvature $\kappa\circ \cT|_{\G_t}$ and $k$-th order derivative of $(D_t\kappa)\circ \cT|_{\G_t}$ and some lower-order terms.

\begin{theorem} (Informal version of Theorem \ref{energy estimate}) Assume that the contact angles $\om, \om_r\in (0, \pi/2)$, the integer $k$ and real $\ga$ satisfy
\[
2\le k\le \min\{\f\pi\om, \f\pi{\om_r}\}, \quad 0<\ga+1\le \min\{\f{2\pi}{\om}, \f{2\pi}{\om_r}\}.
\]
Let  a solution to the water-waves system \eqref{Euler}-\eqref{P condi} be given by the free surface graph function $\zeta$ and its time derivative $\dt\zeta$ with $\al^{-1}\dtx\tzeta\in H^{k+3/2}(\RR)$ and $\al^{-\ga-2}\ttze-a_{\zeta_t}\in H^{k+2}(\RR)$ ($a_{\zeta_t}$ is the weighted limit of $\al^{-\ga-2}\ttze$ near corners, see \eqref{a zeta t def}) and the initial values are chosen from the set of initial values $\Lam_0$ (see Definition \ref{initial value bound} in Section \ref{ref set def}). Moreover, assume that  the weighted Taylor sign condition holds initially for $t=0$, i.e. 
there exists a positive constant $a_0$ such that
\[
\alpha^{-2\ga-3}(-\na_{n_t}P)\circ \cT|_{\G_t}\ge a_0>0 \qquad \hbox{when}\quad t=0.
\]
Then there exists $T^*>0$ depending only on the corresponding initial values such that the energy estimate for  the energy $\cE(t)$ is closed for $t\in [0, T]$ with $T\le T^*$. 
\end{theorem}

To understand better the energy and the a priori estimate  performed in  $\cSt$ instead of the physical domain $\Om$,  we use the following  remark to see what does the conclusion look like if we transform the weighted norms in $\cSt$ back into corresponding weighted norms in $\Om$. To see more details, see Remark \ref{homo norm}.

\begin{remark} According to our theorem,  we set that the mean curvature $\kappa$ of the surface $\G_t$ (instead of using the surface directly)  and its material derivative $D_t\kappa$ satisfy 
\[
r^{1/2}\kappa\in L^2,\ r^{-1/2}(r^{-\ga}D_t\kappa-a_{w_t})\in L^2\quad\hbox{near the corners},
\] 
where $\ga>-1$, $r$ is  the radius with respect to to the corner point and $a_{w_t}$ is some weighted limited defined in Proposition \ref{dt Ka estimate}. This makes Assumption \ref{config} hold.
As a result,   we have for the velocity $v$ and the pressure $P$ that
\[
r^{-1}(r^{-\ga-2}v-a_{v})\in L^2,\quad r^{-1}(r^{-2\ga-4}P-a_p)\in L^2 \quad\hbox{near the corners},
\]
where  $a_{v}$, $a_p$ are corresponding weighted limits of $v, P$ defined in \eqref{av def}, \eqref{a p def}. {\it So one can see that if we set some decay near the corners on the free surface, we will have corresponding decay  for the velocity and the pressure etc. in a short time interval. Moreover, there is some faster decay for time derivatives. Therefore, it makes sense that we deal with the special case when the contact points and angles are fixed during this time period.}
\end{remark}

\begin{remark} (1) The weighted Taylor's sign  holds for $t\in [0, T]$, see Remark \ref{rmk on Taylor} for more details;
(2) The surface is set to be tangentially horizontal at contact points because we want to simplify the estimates technically. It is possible to generalize this point. Of course it is also possible to consider more general cases with contact points and angles changing with respect to time.
\end{remark}

\subsection{Organization of the paper} In Section \ref{preliminaries} we introduce the conformal mapping $\cT$ from the strip domain $\cSt$ to the corner domain $\Om$, and lower-order-derivative estimates of $\cT$ are proved. In Section \ref{elliptic estimate} we  show some weighted elliptic theory in strip domains and deal with some weighted estimates for the D-N operator. The equation of $\kappa$ is introduced in Section \ref{equation}, which is then transformed again using the conformal mapping. Moreover, weighted estimates for $\cT$, $v$, $P$ etc. are proved and the corresponding weighted limits are discussed. In the end, our a priori energy estimate is proved in Section \ref{energy}.

\subsection{Notations}
- The surface $\G_t$ is parameterized by the graph function $z=\zeta(t, x)$, and $\G_b$ is parameterized by $z=b(x)$ with $b$ smooth enough.\\
- we denote by $n_t, \tau_t$ the unit outward normal and tangential vectors on the free surface $\G_t$. $n_b, \tau_b$ are defined in a similar way on the bottom $\G_b$.\\
- $\Pi$: the second fundamental form where $\Pi(w)=\na_w n_t\in T_X\G_t$ (space of tangential vectors)for  $w\in T_X\G_t$ .\\
- $\cSt$ is the strip domain with height $\om$ defined in \eqref{strip} with the same notations $\G_t$, $\G_b$ for boundaries.\\
- $u^\top=(u\cdot \tau_t)\tau_t$ for a vector-valued function $u$ defined on $\G_t$, and $u^\perp=u\cdot n_t$ on $\G_t$.\\
- Let $\cH(f)$ or $f_\cH$ be the harmonic extension in $\Om$ (or $\cSt$) for some function $f$ defined on $\Gamma_t$, which is defined by the elliptic system
\[
\left\{\begin{array}{ll}
\Delta \cH(f)=0\qquad\hbox{in}\quad \Omega\ (\hbox{or}\ \cSt),\\
\cH(f)|_{\Gamma_t}=f,\quad \na_{n_b}\cH(f)|_{\Gamma_b}=0.
\end{array}\right.
\]
- We denote by
\[
\cN f=\na_{n_t}f_{\cH}|_{\G_t}
\]
the Dirichlet-Neumann (D-N) operator without the coefficient $\sqrt{1+|\dx \zeta|^2}$.\\
- The inner product $(\cdot,\cdot)$ stands for the inner product $(\cdot, \cdot)_{L^2(\RR)}$.\\
- The norm $\|\cdot\|_{H^s}$ stands for Sobolev norm in $\Om$ or $\cSt$,  and   $|\cdot|_{H^s}$ denotes the norm in $\RR$ (or in $\G_t$, $\G_b$) when no confusion will be  made. \\
- $\wt f$ stands for $f\circ \cT$ and is defined in $\cSt$ or the surface $\G_t$ of $\cSt$, where $f$ is a function in $\Om$ or the surface $\G_t$ of $\Om$. \\
- $\chi_l, \chi_r$ are some smooth enough cut-off functions near $\tx=-\infty, +\infty$:
\[
\chi_l(\tx)=\begin{cases} 1,\quad \tx\le -c_0-1,\\
0,\quad \tx\ge -c_0,
\end{cases}
\chi_r(\tx)=\begin{cases} 0,\quad \tx\le C_0,\\
1,\quad \tx\ge C_0+1,
\end{cases}
\]
where the constants $c_0$, $C_0$ are defined later in the weight function.\\
- $[s]$ stands for the integer part of $s\in \RR$.\\
- We denote by $a_{f, l}(\tz)=\lim_{\tx\rightarrow -\infty}\al^{-\beta}f$ the weighted limit when $\tx\rightarrow -\infty$ for a function $f(\tx, \tz)$, and we denote by $a_{f, r}$ similar weighted limit when $\tx\rightarrow +\infty$. \\
- The index $i$ ($i=l, r$) in the estimates stands for the summation of $i=l, r$.

\section{Preliminaries}\label{preliminaries}
We first present some Sobolev embedding inequalities, product estimates and a trace theorem.
\begin{lemma}\label{embedding}(1) There exists a constant $C$ independent of $f\in H^s(\RR)$ with $s>1/2$ such that
\[
|f|_{L^\infty(\RR)}\le C|f|_{H^{s}(\RR)};
\]

(2) Let $\mathcal G$ be an open subset of $\RR^2$ with a Lipschitz boundary. Then there exists a constant $C$ independent of $f\in H^s(\mathcal G)$ with $s>1$ such that
\[
|f|_{L^\infty(\mathcal G)}\le C|f|_{H^{s}(\mathcal G)}.
\]
\end{lemma}
\begin{proof} One can check  for example Theorem 1.4.4.1 \cite{PG} for a  complete embedding statement and proof.

\end{proof}

\begin{lemma}\label{basic product} (1) Let $f\in H^1(\RR)$ and $g\in H^{1/2}(\RR)$. Then there exists a constant $C$ independent of $f, g$ such that
\[
|f g|_{H^{1/2}(\mathbb R)}\le C|f|_{H^1(\mathbb R)}|g|_{H^{1/2}(\mathbb R)};
\]

(2) Let $f\in W^{1, \infty}(\RR)$ and $g\in H^{1/2}(\RR)$. Then there exists a constant $C$ independent of $f, g$ such that
\[
|f g|_{H^{1/2}(\mathbb R)}\le C|f|_{W^{1,\infty}}|g|_{H^{1/2}(\mathbb R)};
\]

(3) Let ${\rm a}$ be a function in $W^{1,\infty}(\RR)$ such that ${\rm a}\ge a_0$ for some constant $a_0>0$ and $f\in H^{1/2}(\RR)$. Then there exists a constant $C$ depending on  $a_0^{-1}, |{\rm a} |_{W^{1,\infty}}$ such that
\[
|f|_{H^{1/2}(\RR)}\le C(a^{-1}_0, |{\rm a}|_{W^{1,\infty}})|{\rm a} f|_{H^{1/2}(\RR)}.
\]
\end{lemma}
\begin{proof} For the first estimate, it follows directly from Lemma 2.2 \cite{MW2} and Lemma \ref{embedding}. Moreover, the second estimate can be proved by a direct computation.

For the last estimate, one sets $\mathcal L=1/{\rm a}$ and finds that both $\mathcal L: L^2\rightarrow L^2$ and $\mathcal L: H^1\rightarrow H^1$ are bounded. Therefore, an interpolation with the bounds leads to the desired estimate immediately.

\end{proof}

\begin{lemma}\label{trace} (Traces in  a strip domain) The maps
\[
u\mapsto u|_{\G_j},\quad j=t,\,b
\]
are continuous  from $H^s(\cSt)$ onto $H^{s-1/2}(\G_j)$ for $s> 1/2$ and there hold
\[
|u|_{H^{s-1/2}(\G_j)}\le C \|u\|_{H^s(\cSt)}
\]
for some constant $C$ depending only on $\cSt$ and $j=t,\,b$.
\end{lemma}
\begin{proof} The proof follows from the trace theorem for the full space $\RR^2$ (Theorem 1.5.1.1 \cite{PG} for example) and standard extension and restriction theorems between $\RR^2$ and $\cSt$.

\end{proof}
\bigskip

We are going to introduce a conformal mapping $\cT$ defined from the fixed strip domain $\cSt$ to our physical domain $\Om$, which plays a key role in our paper.
 We show that this conformal mapping has very limited regularity near corners, although it can be smooth enough away from the corner.

\medskip

We define the fixed strip domain 
\begin{equation}\label{strip}
\cSt=\{(\tx,\tz)\,|\,-\om_*\le \tz\le 0\}
\end{equation}
so the depth of $\cSt$ is $ \om_*$. We choose here the depth 
\[
\om_*=\om,
\]
 i.e. the depth is equal to the left contact angle.
Moreover, we still denote by $\G_t$, $\G_b$ the upper and the lower boundaries of $\cSt$ when no confusion will be made.

\bigskip

We now define the conformal mapping
\[
\mathcal T:\tilde Z=\tx+i\,\tz\in  \cSt\mapsto Z=x+iz\in \Om,
\]
where we choose that {\it $\cT$ maps the left infinity of $\cSt$ to the left corner point $p_l$ of $\Om$, and  maps the right infinity of $\cSt$ to the right corner point $p_r$ of $\Om$.} 

\begin{remark}
Notice that we have one more freedom of choosing one point on $\G_t$  to fix the conformal mapping, which is also postponed to Section \ref{ref set def} and is not important.
We also notice that since $\Om$ depends on time $t$,  $\cT$ varies continuously with respect to $t$.
\end{remark}

\begin{center}
\resizebox{9cm}{!}{\includegraphics{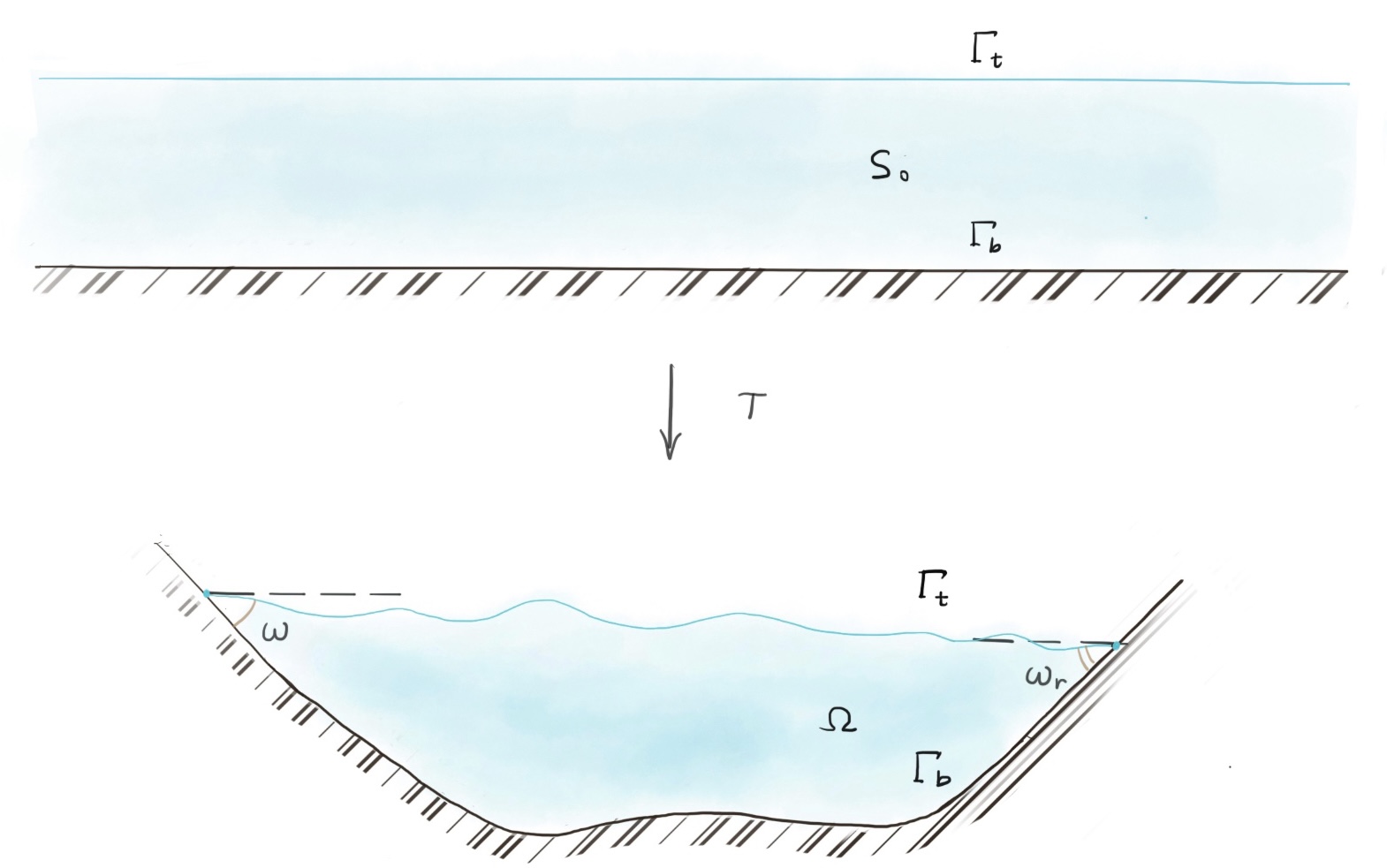}}\\
Figure 2. The conformal mapping
\end{center}

According to the properties of conformal mappings (see \cite{P} for example), we know directly that $\cT$ maps the upper (lower) boundary of $\cSt$ to the upper (lower) boundary of $\Om$, and $\cT$ can be smooth enough away from the corner.  When the surface $\G_t$ is  smooth enough, the limitation for the regularity of $\cT$ only comes from the part of the domain near the contact point.

\medskip

In fact,  we see in the following lines that {\it the conformal mapping only have limited boundary regularity due to the existence of the corner, and its regularity is measured by the angles of the corners.}

The following result focuses on the left contact point and is written in a more general case when the left contact angle $\om$ is different from $\om_*$ (recall that it is the height of the strip $\cSt$), and the boundary is assumed to be piecewise $C^{k,\sigma}$for the moment, which is the same as in \cite{P}. The detailed boundary dependence in weighted norms and  time dependence estimates are discussed later in Section \ref{ref set def}.

\begin{lemma}\label{conformal map corner} (Behavier near the left corner)
Let
\[
\beta=\f\om{ \om_*}
\]
and  the domain $\Om$ be given  with piecewise $C^{k,\sigma}$ boundary where $k$ is a nonegative integer and $\sigma\in (0,1)$.

Then  for  
\[
2\le k\le 1+\pi/\om,
\] there exist positive constants $c_0$ and  $c_1, C_1$ depending on $\Om$, $k$ such that
\[
c_1 e^{\beta\tx}\le |\cT^{(l)}(\tilde Z)|\le C_1 e^{\beta\tx},\qquad \forall \tx\le -c_0, \ l\le k.
\]
\end{lemma}
\begin{proof}

{\bf Step 1. } Mapping $F$ from  $\cK$ (a sector) to $\Om$. Recalling that $p_l$ represents the  left corner point of $\Om$, we define
\begin{equation}\label{F expression}
F(\bar Z)=\cT\circ \big[\beta^{-1}Ln(\bar Z-p_l)\big]\qquad\hbox{for}\quad \bar Z\in \cK
\end{equation}
where
\[
\cK=\big\{\bar Z\,\big|\,\bar Z=p_l+re^{i\theta},\,r\ge 0,\, -\om\le\theta\le0\big\}.
\]
Notice here that $\cK$ is the straightened sector of $\Om$ with the same angle $\om$ of the corner.
Therefore, we relate the estimate for $\cT$ to the estimate for $F$.
\medskip

{\bf Step 2.} Flattening of domains. Letting
\[
\al_0=\f\om \pi,
\]
we define
\[
F_1=Z^{\al_0}+p_l: \cK_{\al_0}\rightarrow \cK,
\]
which maps the half plane
\[
\cK_{\al_0}=\big\{Z\,\big|\,Z=re^{i\theta},\,r\ge 0,\, -\f{\om}{{\al_0}}=-\pi\le\theta\le 0\big\}
\]
to the sector $\cK$.

We also define
\[
F_2=(Z-p_l)^{1/{\al_0}}:\Om\rightarrow \Om_{\al_0}
\]
which flattens $\Om$ in a similar way and move the corner point to $(0,0)$.

We then set the mapping
\[
\wt F=F_2\circ F\circ F_1: \cK_{\al_0}\rightarrow \Om_{\al_0},
\]
which now maps between domains without corners.
Thanks to the assumption of  the boundary regularity, we know that boundaries of $\cK_{\al_0}$ and $\Om_{\al_0}$ are smooth enough near $(0,0)$ ($k\ge 2$ is needed in the following analysis in fact).

Consequently, we know that $\wt F$ is smooth enough near $(0,0)$ and its derivatives are bounded by Theorem 3.6 \cite{P} (Kellogg-Warschawski Theorem).
Besides, direct computations using conformal mappings to the unit disk and applying Theorem 3.7 \cite{P}, we have $Arg \wt F'(0,0)=0$,
which implies
\begin{equation}\label{F prime}
 \wt F'(0,0)=(c,0)\quad\hbox{for some}\quad c>0.
\end{equation}

\medskip

{\bf Step 3.} Estimates of $F$ and $\cT$.
First, we expand $F$ near  $p_l$:
\[
\begin{split}
    F(\bar Z)&
    =\big[\wt F\big((\bar Z-p_l)^{1/{\al_0}}\big)\big]^{\al_0}+p_l\\
    &=p_l+\wt F'(0,0)^{\al_0}(\bar Z-p_l)+\frac{\wt F''(0,0)}{2\wt F'(0,0)^{1-{\al_0}}}(\bar Z-p_l)^{1+1/{\al_0}}+o\big((\bar Z-p_l)^{1+1/{\al_0}}\big).
\end{split}
\]
Notice here that the second term doesn't vanish and the third term makes sense thanks to \eqref{F prime}.

Therefore, there exist two constants $c_1, C_1$ depending on a neighborhood of the corner such that
\[
c_1\le |F'(\bar Z)|\le C_1,\quad\hbox{when}\quad \bar Z\rightarrow p_l.
\]

Moreover, for $k$th-order derivative of $F$ satisfying $ k< 1+1/{\al_0}=1+\pi/\om$, we always have  $F^{(k)}(0,0)=0$ and
\[
 |F^{(k)}(\bar Z)|\le C,\quad\hbox{when}\quad \bar Z\rightarrow p_l
\]
for some constant $C$ depending on $k$ and a  neighborhood of the corner of $\Om$.

\bigskip

Now we are ready to deal with derivatives of $\cT$. In fact,   we obtain thanks to \eqref{F expression} that
\[
\cT'(\tilde Z)=\frac{F'(\bar Z)}{\big[\beta^{-1}Ln(\bar Z-p_l)\big]'}
=F'(\bar Z)\big[e^{\beta\tilde Z}+p_l\big]',
\]
where
\[
\bar Z=\bar Z(\tilde Z)=e^{\beta\tilde Z}+p_l.
\]
As a result, using the estimates of $ F$, one has the estimate for $\cT'$ when $\tx\le -c_0$ for some $c_0$.

Moreover, for higher-order derivatives of $\cT$, direct computations show that the main part in $\cT^{(l)}$ is always
\[
F'(\bar Z)\big[e^{\beta\tilde Z}+p_l\big]^{(l)},
\]
so the desired estimates also follow from the estimate for $F'$.

\end{proof}

A similar result also holds near the right corner and the proof is omitted.
\begin{lemma}\label{conformal map corner right} (Behavier near the right corner)
Let
\[
\beta_r=\f{\om_r}{ \om_*}
\]
and  the domain $\Om$ satisfy the same assumption as in the previous lemma.

Then  for  
\[
2\le k\le 1+\pi/\om_r,
\]
 there exist positive constants $C_0$ and  $c_2, C_2$ depending on $\Om$, $l$ such that
\[
c_2 e^{\beta_r\tx}\le |\cT^{(l)}(\tilde Z)|\le C_2 e^{\beta_r\tx},\qquad \forall\tx\ge C_0, \ l\le k.
\]
\end{lemma}

\bigskip

Based on the estimates near the corner, we now deal with the global estimates for $\cT$ in $\cSt$ and take the time parameter into considerations as well. To begin with, we first introduce the following $C^\infty(\RR)$ weight function
\[
\alpha_\beta(\tx)=\begin{cases}
e^{\beta \tx},\quad \tx\le -c_0,\\
e^{-\beta_r \tx},\quad \tx\ge C_0
\end{cases}
\]
with  $\beta, \beta_r$ defined as above.
Since  the contact angles $\om, \om_r$ are fixed in our settings, one can see that the weight function $\al_\beta(\tx)$ doesn't depend on time $t$.

\begin{proposition}\label{CM estimate}(Global estimates with time parameter)
Assume that the domain $\Om$ and $\cT$ are given by Definition \ref{ref set} in Section \ref{ref set def}. Then there exist two positive constants $c, C$ depending on the bound $L$  from Definition \ref{initial value bound} such that
\[
c\,\al(\tx)\le |\cT'(t, \tilde Z)|\le C \al(\tx),\qquad \forall \tilde Z=(\tx, \tz)\in \cSt,\ \forall t\in [0, T].
\]

Moreover, one also has  for  all $t\in [0, T]$ that
\[
|\cT^{(l)}(t, \tilde Z)|\le C \al(\tx),\qquad\forall\tilde Z\in \cSt,\   l\le k+1\le  1+\min\{\f\pi\om, \f\pi{\om_r}\}
\]
with some constant $C>0$ depending on the initial values  and $k$.
\end{proposition}

\begin{proof}
In fact,   we only need to sum up based on Lemma \ref{conformal map corner} and Lemma \ref{conformal map corner right} and use the bound Definition \ref{initial value bound} to control the estimate with respect to time $t$. 

First of all, according to the settings in Definition \ref{ref set} and Definition \ref{initial value bound}, the surface $\G_t$ lies in a weighted $H^{k+5/2}$ Sobolev space and has uniform bound with respect to the time, so it is locally in $H^{k+5/2}$, which means that the boundary of $\Om$ is piecewise $C^{k+1,\sigma}$ for some $\sigma\in (0, 1)$ by the Sobolev embedding theorem. Therefore,  applying Lemma \ref{conformal map corner} and Lemma \ref{conformal map corner right}, we know immediately that the constants in the estimates there are all controlled by the bound from  Definition \ref{initial value bound}.

Second, since we already have $\cT^{(l)}$ controlled by $\al(\tx)$ near $\infty$, what is left for us to deal with is the behavior of $\cT^{(l)}$ on the bounded domain $\{(\tx, \tz)\in \cSt | -c_0\le \tx\le C_0\}$. Recalling that the conformal mapping can be a smooth-enough diffeomorphism away from corners due to the regularity of boundary,  we know immediately that $\cT'$ satisfies
\[
c_3\le |\cT'|\le C_3
\]
for some constants $c_3, C_3$ depending on the initial values. Besides, we also have
\[
\big|\cT^{(l)}\big|\le C,\qquad l\le k+1
\]
for some positive $C$ depending on the initial values and $k$.

 As a result, summing up all these estimates above, we  derive the desired estimates.

\end{proof}

\bigskip

In the end, we specify  the following weight function used throughout this paper, which is a special case of $\al_\beta$. Recall that we choose the height $\om_*$ of the strip $\cSt$ to be the left contact angle $\om$, i.e.  $\beta=1$ for the left corner:
\begin{equation}\label{weight function}
\alpha(\tx)=\begin{cases}
e^{\tx},\quad\quad  \tx\le -c_0\\
e^{-\beta_r\tx},\quad \tx\ge C_0
\end{cases}
\end{equation}
with $\beta_r=\om_r/\om$.

\medskip

 We now present a natural and useful  corollary, which holds since we set the surface and the bottom to be graphs of functions.
\begin{corollary}\label{dx x estimate}  Let $\cT$ satisfy Definition \ref{ref set} and Definition \ref{initial value bound}. Then there exist constants $c, C$ depending on the bound in Definition \ref{initial value bound} such that
\[
c\al(\tx)\le \big|\dtx x\big|\le C \al(\tx)\quad \hbox{on}\quad \G_t, \ \G_b
\]
holds uniformly with respect to $t\in [0, T]$.
\end{corollary}
\begin{proof}  We only prove the case of $\G_t$ here, and the case of $\G_b$ follows in a similar way. Since
\[
\cT'(Z)=(\dtx x, \dtx z)\qquad\hbox{  in}\quad  \cSt,
\]
  we have
\[
\cT'|_{\G_t: \tz=0}=(\dtx x|_{\tz=0}, \dtx \tzeta)
\]
where $\tzeta=\zeta\circ \cT|_{\G_t}=\zeta(t, x(t, \tx, 0))$ and
\[
\dtx\tzeta=\dtx x\wt{\dx \zeta}.
\]

As a result,  we obtain
\[
\cT'|_{\G_t}=\dtx x|_{\tz=0}(1, \wt{\dx \zeta}),
\]
which implies immediately that
\[
\big|\dtx x|_{\tz=0}\big|\le \big|\cT'|_{\G_t}\big|\le \big(1+|\wt{\dx \zeta}|^2_{L^\infty}\big)^{\f12}\big|\dtx x|_{\tz=0}\big|,
\]
where $|\wt{\dx \zeta}|_{L^\infty}=\big|(\dtx x)^{-1}\dtx\tzeta\big|_{L^\infty}$ is controlled by the upper bound in Definition \ref{initial value bound}.

Consequently, applying the previous proposition leads to the desired estimate.

\end{proof}

\begin{remark}\label{lower order T} The estimates above will be used later only for the lower-order $W^{k+1,\infty}$  estimates of $\cT$  due to the Sobolev  embedding theorem.  There will be weighted higher-order estimates for $\cT$ in the Section \ref{ref set def}, which show the dependence on the boundary, especially for the dependence on the surface function $\zeta$.
\end{remark}

\section{Weighted elliptic estimates related to D-N operator }\label{elliptic estimate}

We present some weighted space in this section, which will be used later in the energy estimate. The weighted elliptic estimates are proved here for the mixed-boundary system as well as for the Dirichlet system. Moreover, the corresponding estimates for the D-N operator $\tN_0$ in the strip domain $\cSt$ are  proved.
\subsection{Weighted elliptic estimates on the strip domain}
We  consider  the following mixed-boundary problem 
\begin{equation}\label{MBVP on S}
\begin{cases}
\Delta u=h\qquad\hbox{in}\quad \cSt,\\
u|_{\Gamma_t}=f,\quad \na_{n_b}u|_{\Gamma_b}=g.
\end{cases}
\end{equation}

Before defining the weighted space, we explain the related eigenvalues of \eqref{MBVP on S}, needed in following lines. To begin with,   we introduce the Laplace transform  with respect to $\tx$
\[
\breve u(\lambda,\cdot)=(\cL u)(\lambda,\cdot)=\displaystyle\int_{\RR}e^{-\lam \tx}u(\tx,\cdot)d\tx, \quad \forall \lam\in \mathbb C.
\]
Performing Laplace Transform, system  \eqref{MBVP on S} formally becomes
\[
\begin{cases}
\lam^2 \breve u+\dtz^2\breve u=\cL(h),\qquad \tz\in [-\om, 0]\\
\breve u|_{\tz=0}=\cL(f),\quad -\dtz\breve u|_{\tz=-\om}=\cL( g).
\end{cases}
\]
As a result, the corresponding eigenvalue problem reads
\[
\begin{cases}
-\phi''(\tz)-\lam^2\phi(\tz)=0,\qquad \tz\in [-\om, 0]\\
\phi(0)=0,\quad -\phi'(-\om)=0.
\end{cases}
\]
Moreover, direct computation shows that the eigenvalues are
\[
\lam_m=\f{(m+1/2)\pi}{ \om},\quad m\in\mathbb Z.
\]
Notice that the first positive eigenvalue is $\lam_0=\pi/(2\om)$ when $m=0$, which will be used later.

\bigskip

Next, we introduce a simple version of weighted norms in $\cSt$. For any $\gamma\in\RR$ and an integer $l\ge 0$, recall (5.2.1) \cite{KMR} the following norms
\[
\|u\|_{\mathcal W^l_{2,\ga}(\cSt)}=\|e^{\ga \tx}u\|_{H^l(\cSt)}
\]
and
\[
|u|_{\mathcal W^{l-1/2}_{2,\ga}(\Gamma_t)}=|e^{\ga \tx}u|_{H^{l-1/2}(\Gamma_t)},
\]
where $\Gamma_t$ can be replaced by $\Gamma_b$.

We present  a specialized version of  Theorem 5.2.2  from page 160 of \cite{KMR}.
\begin{proposition}\label{KMR elliptic estimate}
{\it Suppose that  no eigenvalues $\lambda$ of the problem \eqref{MBVP on S} lie on the line $Re \lambda=-\ga$. Then this problem is uniquely solvable in $\mathcal W^l_{2,\ga}(\cSt)$ for every given $(h, f,  g)$ from the space $\mathcal W^{l-2}_{2,\ga}(\cSt)\times\mathcal W^{l-1/2}_{2,\ga}(\Gamma_t)\times\mathcal W^{l-3/2}_{2,\ga}(\Gamma_b)$ with $l\ge 2$. Furthermore, the solution satisfies the estimate
\[
\|u\|_{\mathcal W^l_{2,\ga}(\cSt)}\le C\big(\|h\|_{\mathcal W^{l-2}_{2,\ga}(\cSt)}+| f|_{\mathcal W^{l-1/2}_{2,\ga}(\Gamma_t)}+| g|_{\mathcal W^{l-3/2}_{2,\ga}(\Gamma_b)}\big)
\]
with a constant $C$ depending on $l, \cSt$.
}
\end{proposition}

\begin{remark}\label{H1 weighted estimate} Checking the proofs of Theorem 3.6.1 and Theorem 5.2.2 in \cite{KMR} line by line, one knows that the estimate in Proposition \ref{KMR elliptic estimate} can be extended to the case $l=1$:
\[
\|u\|_{\mathcal W^1_{2,\ga}(\cSt)}\le C\big(\|h\|_{\mathcal W^{1}_{2,\ga}(\cSt)^*}+| f|_{\mathcal W^{1/2}_{2,\ga}(\Gamma_t)}+| g|_{\mathcal W^{0}_{2,\ga}(\Gamma_b)}\big),
\]
where $\mathcal W^{1}_{2,\ga}(\cSt)^*$ is the dual space of $\mathcal W^{1}_{2,\ga}(\cSt)$ and the norm for $g$ is not a sharp one but easier to handle.

\end{remark}

\medskip

In the end, we introduce the modified weighted norm used later in our paper. Let
\[
\|u\|_{H^l_\ga}:=\|\al^\ga u\|_{H^l(\cS_0)}
\]
and the norm on the boundary is defined in a similar way.  We will use both of the two notations above for the sake of convenience.

\begin{proposition}\label{Weighted elliptic estimate on S}
 Let the integer $l\ge 2$ and real  $\ga$ satisfy
\[
|\ga |< \min\{\f\pi{2 \om}, \f\pi{2\om_r}\}.
\]
Suppose that  system \eqref{MBVP on S} admits a solution $u$ such that $\alpha^{\pm\ga}u\in H^l(\cSt)$ with the right-hand side satisfies
\[
(\alpha^{\ga}h,\alpha^{\ga} f,\alpha^{\ga} g)\in H^{l-2}(\cSt)\times H^{l-1/2}(\Gamma_t)\times H^{l-3/2}(\Gamma_b).
\]

Then  the estimate holds
\[
\|\alpha^{\ga}u\|_{H^l(\cSt)}\le C\big(\|\alpha^{\ga}h\|_{H^{l-2}(\cSt)}+| \alpha^{\ga}f|_{H^{l-1/2}(\Gamma_t)}+| \alpha^{\ga}g|_{H^{l-3/2}(\Gamma_b)}\big)
\]
with a constant $C$ depending on $l, \cSt, c_0, \ga$.
\end{proposition}

\begin{remark} In fact, the unique existence of the solution can also be obtained. We refer to Theorem 2.6.4 \cite{KMR2001} where  this weighted norm in $\cSt$ is transformed back to the so-called homogeneous weighted norm in the physical domain $\Om$.  The restriction of the power $\ga$ comes directly from the existence part, which takes the symmetric interval with respect to the smallest positive eigenvalue $\lam_m$ with $m=0$ defined above. Therefore, we only focus on the weighted estimate here.
\end{remark}

The following  lemma is about the unique solvability and estimate of system \eqref{MBVP on S} in $H^1$ variational formulation.
\begin{lemma}\label{H1 estimate on S}
{\it Let the right-hand side of system \eqref{MBVP on S} satisfy
\[
(h, f, g)\in H^{-1}(\cSt)\times H^{1/2}(\Gamma_t)\times H^{-1/2}(\Gamma_b).
\]
Then the system   admits a unique solution $u\in H^1(\cSt)$ with the estimate
\[
\|u\|_{H^1(\cSt)}\le C\big(\|h\|_{H^{-1}(\cSt)}+|f|_{ H^{1/2}(\Gamma_t)}+|g|_ {H^{-1/2}(\Gamma_b)}\big)
\]
where the constant $C$ depends on $\cSt$.
}
\end{lemma}
\begin{proof} This can be proved in a standard way by applying Lax-Milgram theorem, see \cite{Lannes} for example.

\end{proof}

\begin{proof}  (Proof of Proposition \ref{Weighted elliptic estimate on S}). In fact, the proof is similar to the analysis in Section 2.6.1 \cite{KMR2001}, which is proved in $\Om$ with an equivalent weighted norm in $\Om$. The idea is to start with the existence and estimate in standard variational formula $H^1$ norm in $\Om$, and then apply Corollary 6.3.3 \cite{KMR} to show that the elliptic operator is an isomorphism for $\vert\ga\vert\le \min\{\pi/(2 \om), \pi/(2\om_r)\}$, i.e. to stay in the interval  bounded by the first positive eigenvalue.  In order to be self contained, we provide a direct proof adjust to our case in the following lines.

\medskip

We first prove the case where the weight is $\al^{-\ga}$ ($\ga>0$), which is an easier part. Based on the estimates with weight $\al^{-\ga}$, we  deal with the case with weight $\al^\ga$ by a dual argument.
To begin with, we first introduce some localizations to reduce the estimates to the case in Proposition \ref{KMR elliptic estimate}. 

{\bf Step 1.} Localization and estimates near $\infty$. We only focus on the left infinity, and the proof for the right infinity follows in a similar way.

Recall that $\chi_l$ is  a $\mathcal C^\infty(\RR)$ cut-off function, and let $\bar\chi_l(\tx)$ be a similar cut-off function satisfying
\[
\chi_l\bar\chi_l=\bar\chi_l.
\]
Defining
\[
u_l=\chi_l u,\quad u_r=\chi_r u,\quad u_R=(1-\chi_l-\chi_r)u,
\]
we have immediately 
\[
u=u_l+u_r+u_R
\] and
\[
\alpha^{-\ga}u_l=\chi_l\alpha^{-\ga}u=\chi_l e^{-\ga  \tx}u.
\]
As a result, we know that
\[
\|\alpha^{-\ga}u_l\|_{H^l(\cSt)}=\|\chi_l e^{- \ga \tx}u\|_{H^l(\cSt)}
=\|u_l\|_{\mathcal W^l_{2,-\ga }(\cSt)},
\]
and $u_l$ satisfies the system
\[
\begin{cases}
\Delta u_l=\chi_lh+[\Delta, \chi_l]u:= h_l\qquad\hbox{in}\quad \cSt,\\
u_l|_{\Gamma_t}=\chi_lf:=f_l,
\quad \na_{n_b}u_l|_{\Gamma_b}=\chi_lg+(\na_{n_b}\chi_l)u|_{\Gamma_b}:= g_l.
\end{cases}
\]
Applying Proposition \ref{KMR elliptic estimate}, we get
\[
\|u_l\|_{\mathcal W^l_{2,-\ga }(\cSt)}\le C\big(\|h_l\|_{\mathcal W^{l-2}_{2,-\ga}(\cSt)}
+| f_l|_{\mathcal W^{l-1/2}_{2,-\ga }(\Gamma_t)}+| g_l|_{\mathcal W^{l-3/2}_{2,-\ga }(\Gamma_b)}\big)
\]
where $l\ge 2$ and $C=C(\cSt)$. The right-side terms will be handled one by one. In fact, for $h_l$ term, one has
\[
\begin{split}
\|h_l\|_{\mathcal W^{l-2}_{2,-\ga }(\cSt)}&
\le \|\chi_lh\|_{\mathcal W^{l-2}_{2,-\ga }(\cSt)}+\|[\Delta, \chi_l]u\|_{\mathcal W^{l-2}_{2,-\ga }(\cSt)}\\
&\le \|\chi_lh\|_{\mathcal W^{l-2}_{2,-\ga }(\cSt)}+\|e^{- \ga \tx}[\Delta, \chi_l]u\|_{H^{l-2}(\cSt)}\\
&\le \|\chi_lh\|_{\mathcal W^{l-2}_{2,-\ga }(\cSt)}+C(c_0,\ga, l)\|(1-\bar\chi_l-\bar\chi_r)u\|_{H^{l-1}(\cSt)}
\end{split}
\]
where $\bar \chi_r$ is defined in a similar way as $\bar \chi_l$.

For $g_l$ term, one has similarly
\[
\begin{split}
|g_l|_{\mathcal W^{l-3/2}_{2,- \ga }(\Gamma_b)}&
\le |\chi_lg|_{\mathcal W^{l-3/2}_{2,- \ga }(\Gamma_b)}+|(\na_{n_b}\chi_l)u\|_{\mathcal W^{l-2}_{2,- \ga }(\Gamma_b)}\\
&\le |\chi_lg|_{\mathcal W^{l-3/2}_{2,- \ga }(\Gamma_b)}+C(c_0,\ga, l)\|(1-\bar\chi_l-\bar\chi_r)u\|_{H^{l-3/2}(\Gamma_b)}.
\end{split}
\]
Summing up these estimates above, we obtain
\begin{equation}\label{uc estimate}
\begin{split}
\|u_l\|_{\mathcal W^l_{2,- \ga }(\cSt)}\le &C\Big(\|\chi_lh\|_{\mathcal W^{l-2}_{2,- \ga }(\cSt)}
+| \chi_lf|_{\mathcal W^{l-1/2}_{2,- \ga }(\Gamma_t)}+| \chi_lg|_{\mathcal W^{l-3/2}_{2,- \ga }(\Gamma_b)}
+\|(1-\bar\chi_l-\bar\chi_r)u\|_{H^{l-1}(\cSt)}\Big),
\end{split}
\end{equation}
where the constant $C=C(c_0,\ga, l,\cSt)$. Moreover, similar estimate can be proved for $u_r$ where $-\ga$ is replaced by $\ga \beta_r$.

\medskip

{\bf Step 2. }Estimates for $u_R$. Recalling that $u_R=(1-\chi_l-\chi_r)u$, we have directly
\[
\|\alpha^{-\ga}u_R\|_{H^l(\cSt)}\le C\|u_R\|_{H^l(\cSt)}
\]
with $C=C(c_0,\ga, l)$ and $u_R$ satisfies the system
\[
\begin{cases}
\Delta u_R=(1-\chi_l-\chi_r)h-[\Delta, \chi_l+\chi_r]u\qquad\hbox{in}\quad \cSt,\\
u_R|_{\Gamma_t}=(1-\chi_l-\chi_r)f,
\quad \na_{n_b}u_R|_{\Gamma_b}=(1-\chi_l-\chi_r)g-(\na_{n_b}\chi_l+\na_{n_b}\chi_r)u|_{\Gamma_b}.
\end{cases}
\]
Standard elliptic estimates in strip domains (see \cite{Lannes}) and similar arguments as for $u_l$ case  lead to
\begin{equation}\label{uR estimate}
\begin{split}
\|u_R\|_{H^l(\cSt)}\le
&C\Big(\|(1-\chi_l-\chi_r)h\|_{H^{l-2}(\cSt)}+|(1-\chi_l-\chi_r)f|_{H^{l-1/2}(\Gamma_t)}
\\
&\quad +|(1-\chi_l-\chi_r)g|_{H^{l-3/2}(\Gamma_b)}+\|(1-\bar\chi_l-\bar\chi_r)u\|_{H^{l-1}(\cSt)}\Big),
\end{split}
\end{equation}
where the constant $C=C(c_0,\ga, l,\cSt)$.

Summing up \eqref{uc estimate}  and \eqref{uR estimate}, we obtain
\[
\begin{split}
\|\alpha^{-\ga}u\|_{H^l(\cSt)}\le& C\Big(\|\alpha^{-\ga}h\|_{H^{l-2}(\cSt)}+|\alpha^{-\ga} f|_{H^{l-1/2}(\Gamma_t)}
+|\alpha^{-\ga}g|_{H^{l-3/2}(\Gamma_b)}\\
&\quad +\|(1-\bar \chi_l-\bar\chi_r) u\|_{H^{l-1}(\cSt)}\Big).
\end{split}
\]
To finish the proof, we  need to deal with the lower-order term on the right side.

\medskip

{\bf Step 3.} Lower-order estimate.
In fact, a standard interpolation in Sobolev spaces leads to
\[
\|(1-\bar \chi_l -\bar\chi_r)u\|_{H^{l-1}(\cSt)}\le \epsilon \|(1-\bar \chi_l -\bar\chi_r)u\|_{H^{l}(\cSt)}+C_\epsilon \|(1-\bar \chi_l -\bar\chi_r)u\|_{L^2(\cSt)}
\]
where $C_\epsilon=C(1/\epsilon, \cSt)$.

For the first term on the right side, thanks to the definition of $\bar\chi_l$, one has directly that
\[
\|(1-\bar \chi_l-\bar\chi_r) u\|_{H^{l}(\cSt)}\le C(c_0,\ga, l)\|\alpha^{-\ga}u\|_{H^l(\cSt)}.
\]
On the other hand, applying Lemma \ref{H1 estimate on S} and turning Sobolev norms into the weighted norms lead to
\[
\begin{split}
\|(1-\bar \chi_l-\bar\chi_r) u\|_{L^2(\cSt)}&\le C \|u\|_{L^2(\cSt)}\\
&\le C\big(\|h\|_{H^{-1}(\cSt)}+|f|_{ H^{1/2}(\Gamma_t)}+|g|_ {H^{-1/2}(\Gamma_b)}\big)\\
&\le C\big(\|\alpha^{-\ga}h\|_{H^{l-2}(\cSt)}+|\alpha^{-\ga}f|_{ H^{l-1/2}(\Gamma_t)}
+|\alpha^{-\ga}g|_ {H^{l-3/2}(\Gamma_b)}\big)
\end{split}
\]
since $\ga> 0$.

As a result, we conclude that
\[
\begin{split}
\|(1-\bar \chi_l-\bar\chi_r) u\|_{H^{l-1}(\cSt)}\le & \epsilon C \|\alpha^{-\ga}u\|_{H^l(\cSt)}+C\big(\|\alpha^{-\ga}h\|_{H^{l-2}(\cSt)}+|\alpha^{-\ga}f|_{ H^{l-1/2}(\Gamma_t)}\\
&\quad
+|\alpha^{-\ga}g|_ {H^{l-3/2}(\Gamma_b)}\big),
\end{split}
\]
which together with Step 2 finishes the proof for the case with  weight $\al^{-\ga}$.

{\bf Step 4.} Estimate for the case $\al^\ga$ ($\ga>0$). To prove this case, we first need to investigate again the variational formula for  system \eqref{MBVP on S} with homogeneous boundary conditions and modify the $H^1$ estimate to obtain the weighted $H^1_{-\ga}$ estimate. Second, a dual argument is applied here to obtain the estimate for $H^1_\ga$ case. Third, the inhomogeneous-boundary case follows in a natural way by a trace theorem.

To begin with, we consider system \eqref{MBVP on S} with homogeneous boundary conditions:
\begin{equation}\label{MBVP H on S}
\begin{cases}
\Delta u=h\qquad\hbox{in}\quad \cSt,\\
u|_{\Gamma_t}=0,\quad \na_{n_b}u|_{\Gamma_b}=0
\end{cases}
\end{equation}
and we modify the variation formula for $u\in H^1_0(\cSt)=\{ v\in H^1(\cSt)\,|\, v|_{\G_t}=0,\,\na_{n_b} v|_{\G_b}=0\}$ (See Section 2.6.1 \cite{KMR2001}) now with the right side $h\in (H^1_\ga)^*$ (the dual space of $H^1_\ga$):
\[
\int_{\cSt}\na u\cdot\na v d\tx d\tz=\langle h, \,v\rangle,
\]
where $v\in H^1_0(\cSt)$ and the right side is the dual product on $H^1_\ga$.

One knows immediately that
\[
\langle h, \,v\rangle\le \|h\|_{(H^1_\ga)^*}\|v\|_{H^1_\ga}\le C\|h\|_{(H^1_\ga)^*}\|v\|_{H^1}
\]
where $C$ depends on the weight function $\al$.

Consequently, we obtain
\[
\|u\|_{H^1}\le C\|h\|_{(H^1_\ga)^*}.
\]

On the other hand, repeating the proof in Step 1  and Step 2 with $l=1$, $f, g=0$  and using Remark \ref{H1 weighted estimate}, we find
\[
\|\al^{-\ga} u\|_{H^1}\le C\Big(\|\chi_l h\|_{(\mathcal W^{1}_{2,\ga})^*}+\|\chi_r h\|_{(\mathcal W^{1}_{2,\ga\beta_r})^*}+\|(1-\chi_l-\chi_r) h\|_{H^{-1}}+\|(1-\bar\chi_l-\bar\chi_r)u\|_{H^1}\Big).
\]
Moreover, direct proof shows that
\[
\|\chi_l h\|_{(\mathcal W^{1}_{2,\ga})^*}+\|\chi_r h\|_{(\mathcal W^{1}_{2,\ga\beta_r})^*}+\|(1-\chi_l-\chi_r) h\|_{H^{-1}}\le C \|h\|_{(H^1_\ga)^*}
\]
with the constant $C$ depending on $\chi_i, \al$.

Summing up these estimates above, we conclude that
\begin{equation}\label{variational estimate H1ga}
\|u\|_{H^1_{-\ga}}\le C \|h\|_{(H^1_\ga)^*},
\end{equation}
which shows that the elliptic operator  $\mathcal A$ for  system \eqref{MBVP H on S} is an isomorphism from $H^1_{-\ga}$ to $(H^1_\ga)^*$ with homogeneous boundary conditions  and is consistent with Section 2.6.1 \cite{KMR2001}.

Next, we consider the adjoint operator $\mathcal A^*$ and the elliptic operator $\mathcal A^+$ associated to the formally adjoint problem of \eqref{MBVP H on S}:
\[
\begin{cases}
\Delta v=h\qquad\hbox{in}\quad \cSt,\\
v|_{\Gamma_t}=0,\quad -\na_{n_b}v|_{\Gamma_b}=0,
\end{cases}
\]
where we can see that $\mathcal A^+$ is in fact the same as $\mathcal A$.

We know from Lemma 6.3.4 and (6.3.23) in \cite{KMR}   that  $v\in H^1_{-\ga}$ is a solution of $\mathcal A^+v=(h, 0, 0)$ for a given  $h\in (H^1_\ga)^*$ if and only if $\mathcal A^* v=h$. As a result, this implies for all  $u\in H^1_\ga$ that
\[
\langle h,\,u\rangle =\langle \mathcal A^*v,\, u\rangle =\langle v,\, \mathcal Au\rangle\le \|v\|_{H^1_{-\ga}}\|\mathcal A u\|_{(H^1_{-\ga})^*}.
\]
Besides, we have  by \eqref{variational estimate H1ga} that
\[
\|v\|_{H^1_{-\ga}}\le  C\|\mathcal A^+v\|_{(H^1_\ga)^*}=C\|h\|_{(H^1_\ga)^*},
\]
which leads to
\[
\langle h,\,u\rangle\le C\|h\|_{(H^1_\ga)^*}\|\mathcal A u\|_{(H^1_{-\ga})^*}.
\]
Therefore, we finally obtain the following variational weighted estimate with weight $\al^\ga$:
\[
\|u\|_{H^1_\ga}\le  C\|\mathcal A u\|_{(H^1_{-\ga})^*}
\]
i.e. $\mathcal A$ is also an isomorphism from $H^1_\ga$ to $(H^1_{-\ga})^*$. Moreover, $\|\mathcal A u\|_{(H^1_{-\ga})^*}$ can be controlled directly by $\|\al^\ga h\|_{L^2(\cSt)}$.

To close the proof, repeating the proof from Step 1 to Step 3 and using a trace theorem, the weighted estimate for higher-order case can be easily finished.

\end{proof}




In the end, we also mention the following Dirichlet system:
\begin{equation}\label{DVP on S}
\begin{cases}
\Delta u=h\qquad\hbox{in}\quad \cSt,\\
u|_{\Gamma_t}=f,\quad u|_{\Gamma_b}=g.
\end{cases}
\end{equation}

A similar weighted estimate holds for this system and the proof is omitted here. Moreover, we also refer to Theorem 3.11 \cite{BK} for the corresponding weighted estimate in $\Om$.
\begin{proposition}\label{Dirichlet system estimate}  Let the integer $l\ge 1$ and the real  $\ga $ satisfy
\[
|\ga| < \min\{\f\pi{\om}, \f\pi{\om_r}\}.
\]
Suppose that  the problem \eqref{DVP on S} admits a solution $u$ such that $\alpha^{\ga}u\in H^l(\cSt)$ with the right-hand side given:
\[
(\alpha^{\ga}h,\alpha^{\ga} f,\alpha^{\ga} g)\in H^{l-2}(\cSt)\times H^{l-\f12}(\Gamma_t)\times H^{l-\f12}(\Gamma_b).
\]

Then  the estimate holds
\[
\|\alpha^{\ga}u\|_{H^l(\cSt)}\le C\big(\|\alpha^{\ga}h\|_{H^{l-2}(\cSt)}+| \alpha^{\ga}f|_{H^{l-\f12}(\Gamma_t)}+| \alpha^{\ga}g|_{H^{l-\f12}(\Gamma_b)}\big)
\]
with a constant $C$ depending on $l, \cSt, c_0, \ga$.

\end{proposition}

\begin{remark} The range of the weight is given in Theorem 3.11 \cite{BK} and Theorem 2.6.1 \cite{KMR2001}, which is also decided by  the first eigenvalue $\lam=\pi/\om$ for the Dirichlet system.

\end{remark}

\medskip

\subsection{Weighted estimates for D-N operator}
We first recall some basic properties for the D-N operator $\tN_0$ in the strip domain $\cSt$ from Proposition 3.8 \cite{Lannes}. The Sobolev norm $|\cdot|_{H^s}$ always represents the norm on $\G_t$ from now on when no confusion will be made.

\begin{lemma}\label{basic DN estimate} Let $f, g$ be any two functions in Schwartz class $\cS(\RR)$.  Then\\
(1) The operator $\tN_0$ is self-adjoint:
\[
\big(\tN_0 f,\, g\big)=\big(f, \tN_0 g\big)
\]
where $(\cdot, \cdot)$ is the $L^2(\RR)$ inner product;

(2) The operator $\tN_0$ is positive:
\[
\big(\tN_0f,\, f\big)\ge 0;
\]
(3) There exists a constant $C$ independent of $f, g$ such that the following estimate holds:
\[
\big|\big(\tN_0 f,\, g\big)\big|\le C|f|_{H^{1/2}}|g|_{H^{1/2}}.
\]
Moreover, there exist some positive $\mu$ and $c$ independent of $f$  such that
\[
\big(\tN_0f,\, f\big)+\mu|f|^2_{L^2}\ge c|f|^2_{H^{1/2}}.
\]
\end{lemma}

\bigskip

 We now consider weighted estimates involving  the D-N operator.
\begin{proposition}\label{DN weighted estimate}
{\it Let  the integer $l\ge 1$,  $\ga$ satisfy
\[
|\ga|<\min\{\f\pi{2\om}, \f\pi{2\om_r}\},
\]
 and the weight $\alpha=\alpha(\tx)$ be defined by \eqref{weight function}. Then the following statements hold true:

(1) For any $f$ such that $\alpha^{\ga}f\in H^{3/2}(\Gamma_t)$, one has
\[
\big|\alpha^{\ga}\tN_0f\big|_{H^{1/2}}\le C|\alpha^{\ga}f|_{H^{3/2}}
\]
with a constant $C=C(c_0,\ga,  \cSt)$.

(2) For any $f$ such that $\alpha^{\ga}f\in H^{l+1}(\Gamma_t)$, one has
\[
\big|\alpha^{\ga}\tN_0f\big|_{H^l}\le C|\alpha^{\ga}f|_{H^{l+1}}
\]
with a constant $C=C(l, c_0,\ga,  \cSt)$.
}
\end{proposition}
\begin{proof} (1) The proof is straightforward. In fact, one writes
\[
\big|\alpha^{\ga}\tN_0f\big|_{H^{1/2}}=\big|\dtz(\alpha^{\ga}f_\cH)\big|_{H^{1/2}},
\]
which together with Proposition \ref{Weighted elliptic estimate on S} and Lemma \ref{trace} imply immediately that
\[
\begin{split}
\big|\alpha^{\ga}\tN_0f\big|_{H^{1/2}}
\le C\big\|\alpha^{\ga}f_\cH\big\|_{H^2}
\le C |\alpha^{\ga}f|_{H^{3/2}}.
\end{split}
\]

(2) Higher-order estimate. Similar arguments lead to the desired conclusion.

\end{proof}


Moreover, we also need  some commutator estimates related to the D-N operator.
 \begin{proposition}\label{DN commutator weighted estimate}
Let real $\ga, \ga_1, \ga_2$ satisfy
\[
\ga=\ga_1+\ga_2,\qquad |\ga|, |\ga_1|, |\ga_2|< \min\{\f\pi{2\om}, \f\pi{2\om_r}\}
\] and the weight $\alpha=\alpha(\tx)$ be defined by \eqref{weight function}. Then the following statements hold true:

(1) For any two functions $f, g$ such that $\alpha^{\ga_1}f\in H^{2}(\Gamma_t)$, $\alpha^{\ga_2}g\in H^{1/2}(\Gamma_t)$, one has
\[
\big|\alpha^{\ga}[f, \tN_0]g\big|_{L^2}\le C|\alpha^{\ga_1}f|_{H^{2}}|\alpha^{\ga_2}g|_{H^{1/2}}
\]
with some constant $C=C(c_0,\ga, \cSt)$. 

Moreover, when $\al^{\ga_1}f\in W^{2,\infty}$, one has
\[
\big|\alpha^{\ga}[f, \tN_0]g\big|_{L^2}\le C|\alpha^{\ga_1}f|_{W^{2,\infty}}|\alpha^{\ga_2}g|_{H^{1/2}}
\]

(2) (Higher-order estimate) Let $l\in \mathbb N$. For any two functions $f, g$ such that $\alpha^{\ga_1}f\in H^{l+1}(\Gamma_t)$, $\alpha^{\ga_2}g\in H^{l}(\Gamma_t)$, one has
\[
\big|\alpha^{\ga}[f, \tN_0]g\big|_{H^l}\le C|\alpha^{\ga_1}f|_{H^{l+1}}|\alpha^{\ga_2}g|_{H^{l}}
\]
with some constant $C=C(l, c_0,\ga, \cSt)$.

(3) There exists a constant $C=C(c_0,\ga, \cSt)$ such that
\[
\big|\al^{\ga_1}[\tN_0, \al^{\ga_2}]f\big|_{H^{1/2}}\le C \big|\al^{\ga} f\big|_{H^{1/2}}.
\]
\end{proposition}

\begin{proof} (1) Recalling  that $f_\cH$ denotes the harmonic extension of $f$ in $\cSt$ , we start with rewriting the commutator on $\G_t$:
\[
[f, \tN_0]g=f\na_{n_t}g_\cH-\na_{n_t}(fg)_\cH=-(\na_{n_t}f_\cH)g
-\na_{n_t}\big(f_\cH g_\cH-(fg)_\cH\big).
\]

As a result, denoting by $w=f_\cH g_\cH-(fg)_\cH$, we obtain
\begin{equation}\label{commu estimate}
\big|\alpha^{\ga}[f, \tN_0]g\big|_{L^2}\le
\big|\alpha^{\ga}(\na_{n_t}f_\cH)g\big|_{L^2}+\big|\alpha^{\ga}\na_{n_t}w\big|_{L^2},
\end{equation}
where $w$ satisfies the system
\[
\begin{cases}
\Delta w=2\na f_\cH\cdot \na g_\cH\quad \hbox{in}\quad \cSt,\\
w|_{\Gamma_t}=0,\quad \na_{n_b}w|_{\Gamma_b}=0.
\end{cases}
\]

For the first term on the right side of \eqref{commu estimate},  we have directly  that
\[
\big|\alpha^{\ga}(\na_{n_t}f_\cH)g\big|_{L^2}\le |\alpha^{\ga_1}\na_{n_t}f_\cH|_{L^\infty}
|\alpha^{\ga_2}g|_{L^2}.
\]
Noticing as before that
\[
\alpha^{\ga}\na_{n_t}f_\cH|_{\G_t}=\dtz(\alpha^{\ga}f_\cH)|_{\G_t}
\]
and applying Lemma \ref{trace} and Proposition \ref{Weighted elliptic estimate on S} on it, we find
\[
\begin{split}
\big|\alpha^{\ga}(\na_{n_t}f_\cH)g\big|_{L^2}&\le C\|\alpha^{\ga_1}f_\cH\|_{H^{5/2}}|\alpha^{\ga_2}g|_{L^2}\\
&\le C|\alpha^{\ga_1}f|_{H^2}|\alpha^{\ga_2}g|_{L^2}.
\end{split}
\]

For the second term $\big|\alpha^{\ga}\na_{n_t}w\big|_{L^2}$ in \eqref{commu estimate},  we have immediately by Proposition \ref{Weighted elliptic estimate on S} that
\[
\begin{split}
\big|\alpha^{\ga}\na_{n_t}w\big|_{L^2(\Gamma_t)}&\le \big|\alpha^{\ga}\na_{n_t}w\big|_{H^{1/2}(\Gamma_t)}
\le C\|\alpha^{\ga}w\|_{H^2(\cSt)}\\
&\le C\|\alpha^{\ga}\na f_\cH\cdot \na g_\cH\|_{L^2(\cSt)}\\
&\le C\|\alpha^{\ga_1}\na f_\cH\|_{L^\infty(\cSt)}\|\alpha^{\ga_2}\na g_\cH\|_{L^2(\cSt)}.
\end{split}
\]
Using  Lemma \ref{embedding}, Lemma \ref{trace} and Proposition \ref{Weighted elliptic estimate on S} again, one has
\[
\|\alpha^{\ga_1}\na f_\cH\|_{L^\infty(\cSt)}\le C|\alpha^{\ga_1} f|_{H^2(\Gamma_t)}.
\]
Moreover, applying Remark \ref{H1 weighted estimate} to $g$ leads to
\[
\|\alpha^{\ga_2}\na g_\cH\|_{L^2(\cSt)}\le C\|\alpha^{\ga_2}g_\cH\|_{H^1(\cSt)}
\le C|\alpha^{\ga_2} g|_{H^{1/2}(\Gamma_t)}.
\]

Summing up these estimates above, the proof for the first estimate in (1) is finished. For the second one, repeating the proof above while replace $f_\cH$ by $f$ (constant with respect to $\tz$), one can derive the desired estimate, where the second-order derivative of $f$ is reached in the estimate of $w$.

\medskip

(2) Higher-order derivative estimates. One has immediate the following estimate in a similar way as above:
\[
\big|\alpha^{\ga}[f, \tN_0]g\big|_{H^1}\le
\big|\alpha^{\ga}(\na_{n_t}f_\cH)g\big|_{H^1}+\big|\alpha^{\ga}\na_{n_t}w\big|_{H^1},
\]
and the desired estimate follows naturally.

\medskip

(3) Direct computations lead to
\[
[\tN_0,\al^{\ga}]f=\dtz\cH\big(\al^{\ga}f\big)|_{\G_t}-\al^{\ga}\dtz f_{\cH}|_{\G_t}=\dtz w_1|_{\G_t},
\]
where
\[
w_1=\cH\big(\al^{\ga}f\big)-\al^{\ga}f_\cH
\]
satisfies
\[
\begin{cases}
\Delta w_1=-2\dtx\al^{\ga} \dtx f_\cH-\dtx^2\al^{\ga} f_\cH\quad\hbox{in}\quad \cSt,\\
w_1|_{\G_t}=0,\quad \na_{n_b}w_1|_{\G_b}=0.
\end{cases}
\]

As a result, one has by Lemma \ref{trace}, Proposition \ref{Weighted elliptic estimate on S} and Remark \ref{H1 weighted estimate} that
\[
\begin{split}
\big|\al^{\ga_1}[\tN_0,\al^{\ga_2}]f\big|_{L^2}&=|\al^{\ga_1}\dtz w_1|_{L^2} \le C\|\al^{\ga_1}w_1\|_{H^2}\\
&\le C\big(\big\|\al^{\ga_1}\dtx\al^{\ga_2} \dtx f_\cH\big\|_{L^2}+\big\|\al^{\ga_1}\dtx^2\al^{\ga_2} f_\cH\big\|_{L^2}\big)\le C\big|\al^{\ga}f\big|_{H^{1/2}}.
\end{split}
\]
Notice that here the derivatives of $\al$ are controlled by $C \al$ thanks to the definition of $\al$. Therefore, the proof is finished.

\end{proof}

\section{Reformulation of the problem and preparations}\label{equation}

Similarly as in previous works \cite{MW1, MW2, MW3}, we use again  the geometric formulation in \cite{SZ, SZ2}.  We start with recalling and rewriting the equation of  the mean curvature $\kappa$ of the surface in the physical domain $\Om$, and then transform it into an equivalent equation  in the strip domain $\cSt$.


\subsection{Geometric formulation in the strip domain}

To begin with,  the mean curvature of $\G_t$ is expressed by
\[
\kappa=\na_{\tau_t}n_t\cdot \tau_t.
\] 
Remembering that the free surface $\G_t$ is expressed as $z=\zeta(t, x)$ explicitly, we know immediately that   $\kappa$ is parameterized as 
\[
\kappa=\dx\Big(\f{\dx\zeta}{\sqrt{1+|\dx\zeta|^2}}\Big).
\]
In the following we use both the geometric expression and  the parameterization of $\kappa $ for the sake of convenience, and there is a slight abuse of the notation $\G_t=\RR\times \{0\}$ of $\cSt$ when we use the parameterization $z=\zeta$ and consider $\tzeta$ in $\RR$ indeed.

We also recall some computations from \cite{SZ, MW1}:
\begin{equation}\label{Dtn expression}
D_t n_t=-\big((\na v)^Tn_t\big)^\top=-(\na_{\tau_t}v\cdot  n_t)\tau_t,\quad D_t \tau_t=(\na_{\tau_t}v \cdot n_t)n_t\qquad\hbox{on}\quad \G_t.
\end{equation}

\bigskip

Next, taking the surface tension coefficient  $\sigma=0$ in \cite{SZ},  we find that $\ka$ satisfies
\begin{equation}\label{Dtk eqn}
D_t\ka=-\Delta_{\G_t}v\cdot n_t-2\na_{\tau_t}n_t\cdot \na_{\tau_t}v \qquad\hbox{on}\quad \G_t.
\end{equation}
Taking $D_t$ on this equation again, one obtains immediately as in \cite{SZ} the following equation:
\begin{equation}\label{J eqn 1}
\begin{split}
    D^2_t \ka&=-n_t\cdot\Delta_{\G_t}D_t v+R_1=n_t\cdot\Delta_{\G_t}(\na P)+R_1,
\end{split}
\end{equation}
where
\[
\begin{split}
R_1=& 2\Big[\tau_t\na_{\tau_t}\big((\na v)^Tn_t\big)^\top\cdot \tau_t+\na_{(\na_{\tau_t} v)^\top}n_t\Big]\cdot \na_{\tau_t} v+\D_{\G_t}v\cdot\big((\na v)^Tn_t\big)^\top+2\na_{\tau_t}n_t\cdot (\na v)^2\tau_t\\
&
-2(\na_{\tau_t} v\cdot n_t)\Big(\na_{\tau_t}n_t\cdot\na_{n_t}v
+\na_{\tau_t}n_t\cdot(\na v)^Tn_t\Big)
-2n_t\cdot D^2v\big(\tau_t,\,(\na_{\tau_t}v)^\top\big)\\
&
-2(\na_{n_t}v\cdot n_t)(\na_{\tau_t}n_t\cdot \na_{\tau_t} v)+n_t\cdot \na v\big((\D_{\G_t}v)^\top\big) -\kappa\big|\big((\na v)^Tn_t\big)^\top\big|^2 +2\na_{\tau_t}n_t\cdot \na_{\tau_t}\na P.
\end{split}
\]
Moreover, direct computations show that the pressure $P$ satisfies the following system
\begin{equation}\label{Pvv system}
\begin{cases}
\Delta P=-tr(\na v)^2\qquad \hbox{in}\quad \Om,\\
P|_{\G_t}=0,\quad \na_{n_b}P|_{\G_b}=\na_vn_b\cdot v.
\end{cases}
\end{equation}

Similarly as in \cite{SZ}, we rewrite the term $n_t\cdot \D_{\G_t}\na P$ in \eqref{J eqn 1}:
\[
n_t\cdot \D_{\G_t}\na P=-\ma \,\cN\ka+R_2,
\]
where
\[
\ma=-\na_{n_t} P
\]
is the Taylor's sign and 
$\cN$ is the Dirichlet-Neumann operator. Here $R_2$ is derived by direct computations related to $\Delta_{\G_t}$:
\begin{equation}\label{R2}
\begin{split}
R_2=&-n_t\cdot \Delta\na P+\kappa n_t\cdot\na_{n_t}\na P+n_t\cdot D^2(\na P)(n_t, n_t)\\
=&-n_t\cdot\na\big(tr(Dv)^2\big)+\na_{n_t}\big(\kappa_\cH \na_{n_t} P\big)-\na_{n_t}\kappa_\cH\na_{\bar n_t}P-\kappa\na_{n_t}\bar n_t\cdot \na P+\na_{n_t}\big(D^2P(\bar n_t, \bar n_t)\big)\\
&-2D^2P(n_t, \na_{n_t}\bar n_t)\\
=&-n_t\cdot\na\big(tr(Dv)^2\big)-\na_{n_t}u_p
+\kappa\na P\cdot \na_{n_t}\bar n_t+2D^2P\big(n_t, \na_{n_t}\bar n_t\big)
\end{split}
\end{equation}
where  
\[
u_p:=\kappa_\cH\na_{\bar n_t}P+D^2P(\bar n_t,\bar n_t),
\]
and $\bar n_t$ is an extension of $n_t$ in $\Om$. For simplicity, we just take 
\[
\bar n_t=n_t \qquad\hbox{in}\quad \Om
\]
since $n_t$ relies on the parametrization $z=\zeta(t, x)$ and depends only on $t, x$.

Moreover,  we need to emphasize that $\na_{n_t}u_p$ is in fact a higher-order term about the same order as $\na^2 v$.  In fact, $u_p$ satisfies
\begin{equation}\label{u1 system}
\begin{cases}
   \Delta u_p=2\na\kappa_{\cH}\cdot \na(\na_{n_t}P)+\kappa_\cH\Delta (\na_{n_t}P)+\Delta \big(D^2 P(n_t, n_t)\big)\qquad  \hbox{in}\quad \Om\\
u_p|_{\G_t}=\kappa \na_{n_t}P+D^2P(n_t, n_t)=\Delta P-\Delta_{\G_t}P=-tr(Dv)^2\big|_{\G_t},\\
\na_{n_b}u_p|_{\G_b}=\kappa_{\cH} \na_{n_b}(\na_{n_t} P)+\na_{n_b}D^2P(n_t, n_t)+2D^2 P(\na_{n_b}n_t, n_t)\big|_{\G_b}
\end{cases}
\end{equation}
where one notes the term $\p^3 P$ on $\G_b$. 

The boundary dependence of the corresponding elliptic system of $P$ usually leads to a loss of half-order derivative of the boundary, but  we will  show that in Corollary \ref{higher order local}  that this still works since the boundary dependence in this case only lies in the bottom, which can be as smooth as we want.

\bigskip

Substituting these computations into  equation \eqref{J eqn 1}, we  find
\begin{equation}\label{Nk eqn}
   D^2_t \ka+\ma \cN \ka=R_1+R_2\qquad\hbox{on}\quad \G_t.
\end{equation}



\bigskip
Now we plan to transform \eqref{Nk eqn} into an equivalent equation in $\cSt$ by the conformal mapping $\cT$ defined before. We always use the notation 
\[
\wt u=u\circ \cT
\]
 in the following lines, where  $u$ is some function defined in $\Om$.

To begin with,  for space derivatives in $\Om$, we have the following transform
\begin{equation}\label{na u transform}
\wt{\na u}=(\na u)\circ \cT=(\na_X\wt X)\circ \cT \na \wt u
\end{equation}
where 
\[
X=(x, z)\in \Om,\quad \wt X=(\tx, \tz)\in \cSt
\] 
and 
\begin{equation}\label{na X matrix}
(\na_X\wt X)\circ \cT=|\cT'|^{-2}\left(\begin{matrix}
\dtx x &\dtx z\\
\dtx z & -\dtx x
\end{matrix}\right),
\quad |\cT'|=|(\dtx x, \dtx z)|=\sqrt{(\dtx x)^2+(\dtx z)^2}.
\end{equation}

Meanwhile, taking $\tz=0$ in $z=z(t, \tx, \tz)$, we obtain directly
\[
z=z(t, \tx, 0)=\zeta(t, x(t, \tx, 0)):= \tzeta(t, \tx),
\]
which is usually denoted by 
\[
\tzeta=\zeta\circ \cT|_{\G_t}
\]
 with a slight abuse of notation.
Consequently, we have
\begin{equation}\label{T prime -1}
|\cT'|\big|_{\G_t}=|\dtx x|\sqrt{1+|(\dtx x)^{-1}\dtx\tzeta|^2}
\end{equation}
and
\begin{equation}\label{na X matrix Gt}
\begin{split}
 (\na_X\wt X)\circ \cT\big|_{\G_t}&=|\cT'|^{-2}\left(\begin{matrix}
\dtx x &\dtx\tzeta\\
\dtx\tzeta & -\dtx x
\end{matrix}\right)\\
&=(\dtx x)^{-1}\f 1{1+\big|(\dtx x)^{-1}\dtx \tzeta\big|^2}
\left(\begin{matrix}
1 &(\dtx x)^{-1}\dtx\tzeta\\
(\dtx x)^{-1}\dtx\tzeta & -1
\end{matrix}\right).
\end{split}
\end{equation}

\bigskip

Second, we transform  the D-N operator $\cN$ in $\Om$ into the D-N operator $\wt\cN_0$ in $\cSt$, which is well-understood with estimates already. In fact, direct computations show that for a function $f$ defined on $\G_t$, the harmonic extension $f_\cH$ satisfies
\[
\begin{cases}
\Delta \wt{f_\cH}=0\qquad \hbox{in}\quad \cSt,\\
\wt{f_\cH}|_{\G_t}=\wt f,\quad \na_{n_b}\wt{f_\cH}|_{\G_b}=0.
\end{cases}
\]
As a result, we obtain
\begin{equation}\label{DN transform}
\big(\cN f\big)\circ \cT=(\na_{n_t} f_\cH|_{\G_t})\circ \cT=|\cT'|^{-1}|_{\G_t}\widetilde{\cN}_0\widetilde{f}
\end{equation}
where   $\tN_0$ is the D-N operator in $\cSt$.

\bigskip

On the other hand, we also need to consider the transformation of $D_t u$ for a smooth function $u$ defined in $\Om$.   Taking $\wt u=u\circ \cT$ in $\cSt$ as above,   direct computations show that
\[
(D_t u)\circ \cT=(\dt u)\circ \cT+(v\cdot \na u)\circ \cT=(\dt u)\circ \cT+\tv\cdot (\na_X\tX)\circ\cT\na \wt u,
\]
where
\[
(\dt u)\circ \cT=\dt\wt u-\tv_{\cT0}\cdot \wt{\na u}=\dt\wt u-\tv_{\cT0}\cdot (\na_X\tX)\circ\cT\na \wt u
\]
with
\[
\tv_{\cT0}=(\dt x,\,\dt z)^T.
\]
Consequently, denoting by
\begin{equation}\label{vT}
\tv_{\cT}=(\na_X\tX)^T\circ\cT(\tv-\tv_{\cT0}),
\end{equation}
we conclude the following formula
\[
(D_t u)\circ \cT=\dt\wt u+\tv_{\cT}\cdot\na\wt u:=\wt D_t \wt u\qquad\hbox{in}\quad \cSt.
\]

\bigskip

Moreover, when we look at a function $f=f(t,x)$ defined by the parametrization of the surface,  i.e. 
\[
f(t,x)=F(t,x,z)|_{\G_t:z=\zeta(t,x)}
\]
for some function $F$, we have
\[
\underline D_t f=(D_t F)|_{z=\zeta}\qquad \hbox{with} \quad \underline D_t=\dt+\underline v_1\dx
\]
where
\[
 v=(v_1, v_2)^T,\quad \underline v_1=v_1|_{z=\zeta}.
\]
{\it To simplify the notations, we also use $D_t$ instead of $\underline D_t$ when no confusion will be made.}

Similar computations as above show that for
\[
\wt f=f\circ \cT|_{\tz=0}=f(t, x(t,\tx, 0))
\]
we have
\begin{equation}\label{Dt transform}
(D_t f)\circ \cT|_{\tz=0}=\dt \wt f+u_1\dtx \wt f:=\wt D_t\wt f
\end{equation}
with
\[
 u_1=(\dtx x)^{-1}(\wt {\underline v}_1-\dt x),\quad \wt {\underline v}_1=\underline v_1\circ \cT|_{\tz=0}.
\]
{\it We will use the same notation $\wt D_t$ as in the strip domain $\cSt$ when no confusion will be made in the following text.}

\bigskip

Based on the preparations above,  we are ready to rewrite \eqref{Nk eqn} in the strip domain $\cSt$. First of all, we denote by
\[
w:=\wt \ka.
\]
Second, using the parameterizing $z=\zeta$ on $\G_t$, compositing  $\cT|_{\G_t: \tz=0}$ on both sides of \eqref{Nk eqn} and applying \eqref{Dt transform}, we  derive  
\begin{equation}\label{w eqn}
\wt D_t^2 w+\wt\ma |\cT'|^{-1}\widetilde{\cN}_0w=\wt R_1+\wt R_2
\end{equation}
where
\[
\wt R_i=R_i\circ \cT,\quad  i=1,2.
\]
Notice that  here {\it we omit the notation $\cdot|_{\G_t}$ for $|\cT'|^{-1}$ when no confusion will be made.}

\subsection{The proper weights} 
Since a natural weight $|\cT'|^{-1}$ appears in \eqref{w eqn}, we plan to identify the weights for all the quantities. To begin with, recalling the weight function $\al(\tx)$ from \eqref{weight function} and  setting
\[
\ga\in\mathbb R \quad\hbox{to be fixed later,}
\]
$\wt D_t w$ is assumed to be equipped with the weight $\al^{-\ga}$ and we try to use  the weighted norm $\big|\al^{-\ga}\wt D_t w\big|_{H^k}$.


Remembering that  $|\cT'|^{-1}\approx \al^{-1}$ thanks to Proposition \ref{CM estimate},
 it seems that we should put the weight $\al^{-\ga-\f12}$ before $w$ while notice that there is the Taylor's sign $\ta$ with $w$.

Besides,  recalling \eqref{na u transform} shows that for every space derivative in $\Om$, there is  approximately one $\al^{-1}$ together with one space derivative in $\cSt$.  Meanwhile, looking at \eqref{T prime -1} and \eqref{na X matrix Gt},  we find a natural weight $\al^{-1}$ in front of $\dtx\tzeta=\dtx x\wt{\dx\zeta}$ , which suggests that the weight for $\dtx\tzeta$ seems to be  $\al^{-1}$.


To have some more clues about weights, we look at the system for $\wt P=P\circ \cT$. In fact, $\tP$ satisfies
\begin{equation}\label{tPvv eqn}
\begin{cases}
\Delta \tP=-|\cT'|^2tr\big((\na_X \tX)\circ \cT\na \tv\big)^2\qquad \hbox{in}\quad \cSt,\\
\tP|_{\G_t}=0,\quad \dtz\tP|_{\G_b}=|\cT'|\big((\na_X \tX)^T\circ \cT\big)\tv\cdot \na\tn_b\cdot \tv.
\end{cases}
\end{equation}
 Proposition \ref{Weighted elliptic estimate on S} suggests that the terms on the right side of the above system decide the weight of $\tP$. Besides, we have
\[
|\cT'|\approx \al,\quad (\na_X \tX)\circ \cT|_{\G_t}\approx \al^{-1}
\]
 thanks to  Proposition \ref{CM estimate}, and notice that  $n_b$ is assumed to be constant near the corners,
 which implies immediately that the boundary term $\dtz\tP|_{\G_b}$ vanishes near $\tx=\infty$.
Consequently,  the weight for $\tP$ should be determined by $\tv$-terms  on the right side.

\bigskip



Now we are ready to identify all the weights.  According to the assumption that $\wt D_t w$ carries the weight $\al^{-\ga}$ and the fact $w\approx \al^{-2}\p^2\tzeta$, we find the weight $\al^{-\ga-2}$ for $\wt D_t\tzeta$.

On the other hand, direct computations and applying the kinematic condition \eqref{kinematic condi} show that
\begin{equation}\label{dt tzeta expression}
\begin{split}
\dt\tzeta&=\wt{\dt\zeta}+\dt x(\dtx x)^{-1}\dtx\tzeta\\
&=\tv\cdot \big(-(\dtx x)^{-1}\dtx\tzeta,\,1)^T+\dt x(\dtx x)^{-1}\dtx\tzeta.
\end{split}
\end{equation}

As a result, checking \eqref{Dt transform}, \eqref{dt tzeta expression}, \eqref{tPvv eqn} and counting the orders of derivatives one by one suggest the following weights:
\[
 \cT':\al^{-1},\quad \ttze:  \alpha ^{-\ga-2},\quad
\tv:  \alpha^{-\ga-2}, \quad \tP: \alpha^{-2\ga-4}.
\]
The weight for $\dt \cT$ will be determined later. 
Moreover,  the weight for
\[
\wt\ma=-\widetilde{\na_{n_t}P}=-|\cT'|^{-1}\dtz\tP|_{\tz=0}
\] should be $\alpha^{-2\ga-3}$,
which implies that
\[
\ta^{1/2}\ \hbox{carries the weight} \ \alpha^{-\ga-3/2}.
\]
We will assume later that $\ta$ is positive with the weight (the weighted Taylor's sign condition),  which together with the system of $\tP$ suggests that $\al^{-2\ga-4}\tP$ doesn't go to zero when $\tx$ goes to $\infty$. Consequently, {\it we will have to investigate the weighted limits for $\tP$ and other quantities as well.}

Meanwhile, we  assume that
\[
w=\wt\ka \quad \hbox{ carries the weight }\ \alpha^{q},
\]
so the weight for $\tzeta$ should be $\al^{q+2}$ (we use $\dtx\tzeta$ indeed).

In the end,  the left side of equation \eqref{w eqn} suggests that $
\ta^{1/2} w$ should carry the weight $\alpha^{-\ga-1/2}$, which means
\[
-\ga-3/2+q=-\ga-1/2,\quad \hbox{i.e.}\quad q=1.
\]
Consequently,  we have the following weights:
\[
w:  \alpha, \quad \tzeta \quad\hbox{or}\ \dtx\tzeta: \alpha^{-1},
\]
which coincide with our observation before.

\subsection{Preparations for the energy estimate}\label{ref set def}
 Before introducing the energy, we  define the set of solutions which involves  the surface function $\zeta$, the conformal mapping $\cT$, the velocity $v$ and the pressure $P$. 
 
 To begin with, we consider the weighted limits of $\ttze$, $\cT'$, $\tP$, $\dtz\tPhi$ which will be used later. We define 
\begin{equation}\label{a zeta t def}
a_{\zeta_t}(t, \tx)=\chi_l(\tx)a_{\zeta_t, l}(t)+\chi_r(\tx)a_{\zeta_t, r}(t)
\end{equation}
where $\chi_l, \chi_r$ are the cut-off functions near $\mp \infty$, and $a_{\zeta_t, i}$ ($i= l, r$) are the weighted limits of $\ttze$:
\[
a_{\zeta_t, l}=\lim_{x\rightarrow -\infty}\al^{-\ga-2}\wt{\dt \zeta},\quad a_{\zeta_t, r}=\lim_{x\rightarrow +\infty}\al^{-\ga-2}\wt{\dt \zeta}.
\]
Therefore, $a_{\zeta_t}$ is well-defined as long as $\al^{-\ga-2}\ttze\in L^\infty(\cSt)$ for each $t$.

Moreover, we set $T_{1,c}$, $a_p$ and $a_{\Phi_z}$ as the corresponding weighted limits for $\cT'$, $\tP$ and $\dtz \tPhi$ defined in \eqref{T1c def}, \eqref{a p def} and Proposition \ref{v estimate}  respectively, which can be  well-defined in a similar way and turn out to depend on $a_{\zeta_t}$.

We now are ready to define the set of solutions to \eqref{w eqn}.
\begin{definition} \label{ref set} Let $T>0$, $k\in \mathbb N$ and $Q_\Lam$ be some positive-coefficient polynomial.
Let $\Lam=\Lam(T, k)$ be the collection of $(\zeta,  \cT, \Phi, P)$ satisfying the water-waves system \eqref{Euler}-\eqref{P condi} (where $\cT$ is defined by $\zeta$) and the following conditions:

(i) $(\zeta,  \cT, \Phi, P)$ lies in  the following spaces
\[
\begin{split}
&\al^{-1}\dtx\tzeta\in L^\infty_T\big(H^{k+3/2}(\RR)\big),\quad \al^{-\ga-2}\ttze-a_{\zeta_t}\in L^\infty_T \big(H^{k+2}(\RR)\big), \al^{-1}\cT'-T_{1,c}\in L^\infty_T\big(H^{k+2}(\cSt)\big),\\
& \al^{-\ga-3}\na\tPhi-(a_{\Phi_x},  a_{\Phi_z})^T\in L^\infty_T\big(H^{k+5/2}(\cSt)\big),\quad \al^{-2\ga-4}\tP-a_p\in L^\infty_T\big(H^{k+7/2}(\cSt)\big),
\end{split}
\]
where $a_{\zeta_t}$,  $T_{1,c}$, $a_p$ and $a_{\Phi_x}, a_{\Phi_z}$ are well defined such that
\[
a_{\zeta_t}\in L^\infty([0, T]),  \quad T_{1,c}, a_p, a_{\Phi_x}, a_{\Phi_z}\in  L^\infty_T\big(L^\infty([-\om, 0])\big);
\]

(ii) $(\zeta,  \cT, \Phi, P)$ is bounded uniformly for $t\in [0, T]$:
\[
\begin{split}
&|\al w|^2_{H^{k+1/2}}+|\al^{-\ga}\wt D_t w-a_{w_t}|^2_{H^{k}}+\|\al^{-1}\cT'-T_{1,c}\|^2_{L^2}+|\al^{-\ga-2}\ttze-a_{\zeta_t}|^2_{L^2}
\\
&+\big\|\al^{-\ga-3}\na\tPhi-(a_{\Phi_x}, a_{\Phi_z})^T\big\|^2_{L^2}+\big\|\al^{-2\ga-4}\tP-a_p\big\|^2_{L^2}+|a_{\zeta_t, i}|^2+|b|^2+|b_r|^2+|b_{r2}|^2\le   Q_0
\end{split}
\]
where $a_{w_t}$ is well defined in Proposition \ref{dt Ka estimate} and 
\[
\begin{split}
Q_0=&Q_\Lam\big(|\al w(0)|_{H^{k+1/2}}, |\al^{-\ga}\wt D_t w(0)-a_{w_t}(0)|_{H^k}, \|\al^{-1}\cT(0)-T_{1,c}(0)\|_{L^2}, |\al^{-\ga-2}\ttze(0)-a_{\zeta_t}(0)|_{L^2},\\
& \big\|\al^{-\ga-3}\na\tPhi(0)-(a_{\Phi_x}, a_{\Phi_z})^T(0)\big\|_{L^2},  \big\|\al^{-2\ga-4}\tP(0)-a_p(0)\big\|_{L^2}, |a_{\zeta_t, i}(0)|, |b(0)|, |b_r(0)|, |b_{r2}(0)|\big).
\end{split}
\]
\end{definition}

Moreover, we define the set for the initial values.
\begin{definition}\label{initial value bound} Let $L>0$ be some fixed constant. $\Lam_0$ is the collection of the initial values of $(\zeta,  \cT, \Phi, P)$ satisfying the water-waves system \eqref{Euler}-\eqref{P condi} such that
\[
\begin{split}
&|\al w(0)|_{H^{k+1/2}}, |\al^{-\ga}\wt D_t w(0)-a_{w_t}(0)|_{H^k}\le L,\\
& \|\al^{-1}\cT(0)-T_{1,c}(0)\|_{L^2}, |\al^{-\ga-2}\ttze(0)-a_{\zeta_t}(0)|_{L^2}, \big\|\al^{-\ga-3}\na\tPhi(0)-(a_{\Phi_x}, a_{\Phi_z})^T(0)\big\|_{L^2}\le L,\\
&\big\|\al^{-2\ga-4}\tP(0)-a_p(0)\big\|_{L^2}, |a_{\zeta_t, i}(0)|, |b(0)|, |b_r(0)|, |b_{r2}(0)|\le L.
\end{split}
\]
\end{definition}

\begin{remark}\label{Assump 1} According to the weighted norms for $\tzeta$ and $\dt\tzeta$ in these definitions above, we see directly  that Assumption \ref{config} holds for $t\in [0, T]$.
\end{remark}
\begin{remark}\label{fix T}
For each $t$, the third point in $\G_t$ of $\cSt$ to fix $\cT$ is always chosen such that
\[
\lim_{\tx\rightarrow-\infty}\al^{-1}\dtx x|_{\G_t: \tz=0}=1.
\]
In fact, this choice  is not important of course and simplifies a little bit the first condition in \eqref{bdy condition ax az}. 
\end{remark}





\bigskip

Since the higher-order energy will be based on  the curvature $w$ (related to the  surface function $\tzeta$), we need to investigate estimates for all the quantities like $\tv,  \tP$ in  \eqref{w eqn} and express them by  $w$. Meanwhile, we notice that the higher-order space derivatives and the time derivative of $\cT$ are  involved a lot. As a result, the  estimates in our paper will be proved under the choice  
\[
(\zeta,  \cT, \Phi, P)\in \Lam(T, k)
\]
with their initial values chosen from $\Lam_0$ and 
\[
k\le \min\{\f\pi\om, \f\pi{\om_r}\}
\]
comes from Proposition\ref{CM estimate}.

Moreover, we need the following weighted Taylor's sign assumption in our energy estimate.

\begin{assumption}\label{assump on a}
There exists a positive constant $a_0$ such that
\[
\alpha^{-2\ga-3}\ta\ge a_0>0 \qquad\hbox{on}\quad \G_t: \tz=0,\quad\hbox{for}\ t\in [0, T].
\]
\end{assumption}

\noindent{\bf Notations}:  \\
(1) We denote by $Q(\cdot)$ some polynomial with positive constant coefficients depending on $k, \ga, \om, \om_r$ and the bound $L$ from Definition \ref{initial value bound};\\
(2) $Q(\cT',\na\tPhi, \tP, w, s)$ stands for a polynomial $Q(\cdot)$ depending on $|\al w|_{H^{s}}$ and the corresponding weighted $L^2$ norms of $\cT'$, $\na\tPhi$, $\tP$ as well as the $L^\infty$ norms of their left and right weighted limits (In fact, all these limits turn out to depend on $T_{1,c}$  and $a_{\zeta_t}$ defined in \eqref{a zeta t def} and \eqref{T1c def}).  The other notations $Q(\cdot, w, s)$ are defined in a similar way;\\
(3) $Q_b$ is a polynomial $Q(|b|_{L^\infty}, |b_r|_{L^\infty}, |b_{r2}|_{L^\infty})$;\\
(4) $C$ is a positive constant depending on $k$, $\om$, $\om_r$, $\ga$ and the bound $L$ from Definition \ref{initial value bound};\\
(5) We will denote by  $|\cdot|$  the norm with respect to $\RR$ (for $\G_t$ or $\G_b$) and by $\|\cdot\|$ the norm with respect to $\cSt$ when no confusion will be made.
\medskip

\subsection{Weighted estimates for the conformal mapping $\cT$}\label{T estimate section}
As the first step,  we  show detailed  boundary-dependence estimates for higher-order space derivatives of $\cT$. 
To begin with, we investigate the behavior of $\cT'$ at  infinity, which is already showed to some extent  by Proposition \ref{CM estimate}.
Corresponding to the left corner, we denote by
\begin{equation}\label{limit T'}
a_{x, l}(t, \tz)=\lim_{\tx\rightarrow -\infty}\al^{-1}\dtx x,\qquad
a_{z, l}(t,\tz)=\lim_{\tx\rightarrow -\infty}\al^{-1}\dtx z,\quad \forall \tz\in  [-\om, 0]
\end{equation}
where we know from Proposition \ref{CM estimate}, Remark \ref{lower order T}, Definition \ref{ref set}, Definition \ref{initial value bound} and Remark \ref{fix T}  that
$a_{x, l}, a_{z, l}$ vary continuously   with respect to $z$ with the following end-point values:
\begin{equation}\label{bdy condition ax az}
\begin{cases}
a_{x, l}(t, 0)=1, \quad a_{x, l}(t, -\om)=b(t)\\
a_{z, l}(t, 0)=\lim_{\tx\rightarrow-\infty}\wt{\dx \zeta}=0,\quad a_{z, l}(t, -\om)=\lim_{\tx\rightarrow-\infty}\wt{\dx b}=-k_bb(t)
\end{cases}
\end{equation}
where $b(t)$ is some function unknown and to be estimated later.

Similarly, we define for the right corner the following limits:
\begin{equation}\label{limit T' right}
a_{x,r}(t, \tz)=\lim_{\tx\rightarrow +\infty}\al^{-1}\dtx x,\qquad
a_{z,r}(t,\tz)=\lim_{\tx\rightarrow +\infty}\al^{-1}\dtx z,\qquad \forall \tz\in [-\om, 0]
\end{equation}
 with the  values
\begin{equation}\label{bdy condition ax az right}
\begin{cases}
a_{x, r}(t, 0)=b_r(t), \quad a_{x,r}(t, -\om)=b_{r2}(t)\\
a_{z, r}(t, 0)=\lim_{\tx\rightarrow+\infty}\wt{\dx \zeta}=0,\quad a_{z, r}(t, -\om)=\lim_{\tx\rightarrow+\infty}\wt{\dx b}=-k_{b, r}b_{r2}(t).
\end{cases}
\end{equation}

Moreover, we denote by 
\[
a_x=\chi_{l}(\tx) a_{x, l}(t, \tz)+\chi_{r}(\tx)a_{x, r}(t, \tz), \quad a_z=\chi_{l}(\tx) a_{z, l}(t, \tz)+\chi_{ r}(\tx)a_{z, r}(t, \tz)
\]
where the cut-off functions $\chi_{i}$ are already defined before.

To consider higher-order space derivatives for $\cT$, we  denote by 
\begin{equation}\label{T1c def}
T_{1, c}= \chi_l T_{1,l}+\chi_r T_{1,r}
\end{equation}
where 
\[
T_{1, l}=(a_{x, l}, a_{z, l}),\quad T_{1, r}=(a_{x, r}, a_{z, r}).
\]
We plan to use the  norm $\|\al^{-1}\cT'-T_{1, c}\|_{H^k}$, or equivalently the norms $\|\al^{-1}\dtx x- a_x\|_{H^k}$, $\|\al^{-1}\dtx z-a_z\|_{H^k}$.

To understand this norm, we need some more computations.  In fact,  when $\al^{-1}\dtx z-  a_z\in H^k(\cSt)$ for some integer $k\ge 0$, we have
\[
\begin{split}
\dtx (\al^{-1}\dtx z-a_z)=&-\al^{-2}\al' \dtx z+\al^{-1}\dtx^2z+\chi'_{ l}a_{z, l}+\chi'_{r} a_{z,r}\\
=&-(\al^{-2}\al' \dtx z-\chi_{l,1}a_{z, l}+\beta_r \chi_{r,1}a_{z, r})+(\al^{-1}\dtx^2 z- \chi_{l, 1}a_{z, l}+\beta_r \chi_{r, 1}a_{z, r})\\
&+\chi'_{ l}a_{z, l}+\chi'_{r} a_{z,r}
\end{split}
\]
where $\chi_{i, 1}$ are  cut-off functions similar as $\chi_i$ with $supp\chi_{i, 1}$ slightly smaller than $supp \chi_i$.

As a result, this together with $\al^{-1}\dtx z-  a_z\in L^2(\cSt)$ implies
\[
\al^{-1}\dtx^2 z- a_{z, 1}\in L^2(\cSt)
\]
with 
\[
a_{z, 1}=\chi_{ l, 1}a_{z, l}-\beta_r \chi_{ r, 1}a_{z, r},
\]
and the following  estimate holds
\[
\begin{split}
\|\al^{-1}\dtx^2 z-a_{z, 1}\|_{L^2}&\le \|\dtx (\al^{-1}\dtx z-  a_z)\|_{L^2}+\|\al^{-2}\al' \dtx z-a_{z, 1}\|_{L^2}+\|\chi'_{c, i}a_{z, i}\|_{L^2}\\
&\le C\big(\|\al^{-1}\dtx z- a_z\|_{H^1}+|a_{z,i}|_{L^\infty}\big).
\end{split}
\]
Here we denote by $|a_{z,i}|_{L^\infty}$ the summation of $|a_{z,i}|_{L^\infty}$ on $i=l ,r$.

Therefore, we  also know by an induction that
\[
\al^{-1}\dtx^mz-a_{z, m}\in L^2(\cSt),\quad 1\le m\le k+1\qquad\hbox{when}\quad \al^{-1}\dtx z- a_z\in H^k(\cSt)
\]
with the notation
\[
a_{z, m}=\chi_{l, m}a_{z, l}+(-\beta_r )^m\chi_{r, m}a_{z, r}
\]
and $\chi_{i, m}$ are cut-off functions similar as $\chi_i$ with smaller supports than $\chi_{i, m-1}$.
Moreover,  $\|\al^{-1}\dtx^mz-a_{z, m}\|_{L^2}$ can be controlled by  $\|\al^{-1}\dtx z-  a_z\|_{H^k}$ and $|a_{z, i}|_{L^\infty}$.

Consequently, we summerize these computations in the following lemma with a slightly more general form.
\begin{lemma}\label{minus norm dx}
Let $u$ be a smooth enough function in $\cSt$ and
\[
a_u(\tx, \tz)=\chi_la_{u,l}+\chi_ra_{u,r}=\chi_{l}\lim_{\tx\rightarrow-\infty} \al^{-\ga}u+\chi_{r}\lim_{\tx\rightarrow+\infty} \al^{-\ga}u.
\] Assume that $\al^{-\ga}u-a_u\in H^k(\cSt)$ with $\ga$ real.

Then we have for any integer $m$ satisfying $1\le  m\le k$ that
\[
\al^{-\ga}\dtx^lu-a_{u, m}\in L^2(\cSt)
\]
with 
\[
a_{u, m}=\ga^m\big(\chi_{l, m}a_{u, l}+(-\beta_r )^m\chi_{r, m}a_{u, r}\big)
\]
and the estimate holds
\[
\|\al^{-\ga}\dtx^lu-a_{u, m}\|_{L^2}\le C\big(\|\al^{-\ga}u-a_u\|_{H^k}+|a_u|_{L^\infty}\big)
\]
with the constant $C$ depending on $k, \cSt, \chi_i, \al, \beta_r$.
\end{lemma}

\medskip

Focusing on $\cT$ again and applying C-R equations and Lemma \ref{minus norm dx}, we know immediately that $a_{x, i}$ and $a_{z, i}$ are related and we only write the case $i=l$:
\[
\begin{split}
\dtz a_{z, l}&=\lim_{\tx\rightarrow -\infty}\al^{-1}\dtx \dtz z=\lim_{\tx\rightarrow -\infty}\al^{-1}\dtx^2 x=a_{x, l},\\
\dtz a_{x, l}&=\lim_{\tx\rightarrow -\infty}\al^{-1}\dtx \dtz x=-\lim_{\tx\rightarrow -\infty}\al^{-1}\dtx^2 z=-a_{z, l},
\end{split}
\]
where we need $x, z\in C^2(\cSt)$, which is satisfied in Definition \ref{ref set}  already.

Solving these equations above with boundary conditions \eqref{bdy condition ax az}, we obtain
\begin{equation}\label{ax az}
a_{x, l}=\cos \tz +\f{\cos \om-b(t)}{\sin\om}\sin\tz,\quad a_{z, l}=\f{k_bb(t)}{\sin\om}\sin\tz
\end{equation}
where $k_b$ is the slope of $\G_b$ in $\Om$ near the left corner.
The expressions for $a_{x, r}$ and $a_{z, r}$ are similar. In fact, these expressions are very  natural.  When we focus on the left corner of $\Om$  which approximates a   sector, the conformal mapping $\cT$ is simply like $e^{\wt Z} $.

Now we are ready to deal with higher-order space derivatives of $\cT$.
\begin{proposition}\label{dxT estimate}
Let $T_{1,c}$ be given by \eqref{T1c def}. Then for any integer $k\ge 1$, there holds
\[
\|\al^{-1}\cT'-T_{1,c}\|_{H^{k+2}}\le Q\big(\cT', w, k+1/2\big)
\]
Moreover,  we have 
\[
|T_{1,c}|\le C\big(|b(t)|+|b_r(t)|+|b_{r2}(t)|\big).
\]
\end{proposition}

\begin{proof} {\bf Step 1}. We first explore some relationship between $\cT$ and the surface curvature $\kappa$, which is inspired by Shatah-Walsh-Zeng\cite{SWZ}.
According to the definition of $\cT$, the curvature  can be expressed as
\[
\wt\kappa=\kappa\circ \cT=\f{\dtx x\dtx^2 z-\dtx^2 x\dtx z}{|\cT'|^3}\qquad \hbox{on}\quad \G_t,
\]
which leads directly to the following expression
\[
Re \f{\cT''}{i \cT'}=\wt\kappa|\cT'|\qquad \hbox{on}\quad \G_t.
\]
Moreover, the same expression can be derived on $\G_b$ with the curvature $\kappa_b$.

Therefore, we express the holomorphic function as 
\begin{equation}\label{holomorphic T}
 \f{\cT''}{i \cT'}=u_r+i\,u_i,
\end{equation}
where the real component $u_r$ satisfies
\[
\begin{cases}
\Delta u_r=0\qquad \hbox{in}\quad \cSt,\\
 u_r|_{\G_t}=\wt\kappa|\cT'|\,|_{\G_t}=w|\cT'|\,|_{\G_t},\quad  u_r|_{\G_b}=\wt\kappa_b|\cT'|\,|_{\G_b}.
\end{cases}
\]

Performing standard elliptic estimate in Sobolev spaces, we have
\[
\|u_r\|_{H^{k+1}}\le C \big(\big|w|\cT'|\big|_{H^{k+1/2}}+\big|\wt\kappa_b|\cT'|\big|_{H^{k+1/2}(\G_b)}\big)
\]
where $C$ is a constant depending on $k, \cSt$.

We firstly take care of the term $\big|w|\cT'|\big|_{H^{k+1/2}}$. In fact, one derivative on $|\cT'|$ in this norm results in terms like $g\dtx^2 x w$ and $g\dtx^2 z w$, where $\al w\in H^{k+1/2}$,  and  $g$ is some function depending on $\cT^{(l)}$ ($l\le k+1$) and can be bounded by $L^\infty$ norm of $\cT^{(l)}$ by Proposition \ref{CM estimate}.
Consequently, we write
\[
\dtx^2 x w=\al^{-1}\dtx^2 x \al w=(\al^{-1}\dtx^2 x-a_{x, 1})\al w+a_{x, 1} \al w
\]
where $a_{x, 1}$ is defined  in Lemma \ref{minus norm dx} where $u=\dtx x$.
Meanwhile, $\dtx^2 z w$ can be expressed in a similar way. As a result, thanks to Lemma \ref{minus norm dx} and \eqref{T1c def},  we know that the corresponding weighted limit for $\cT''$ can be written in the following form:  
\[
T_{2,c}=\chi_{l, 1}T_{l,1}-\beta_r\chi_{r,1} T_{r,1}.
\]
  These expressions lead to
\[
\big|w|\cT'|\big|_{H^{k+1/2}}\le C\big(1+\big|\al^{-1}\cT''-T_{2,c}\big|_{H^{k-1/2}}+|a_{x, i}|_{L^\infty}+|a_{z, i}|_{L^\infty}\big)
|\al w|_{H^{k+1/2}}
\]
where $C$ is a constant depending on the bound $L$ given in Definition \ref{initial value bound}.

Similarly, we can also prove that
\[
\begin{split}
\big|\wt\kappa_b|\cT'|\big|_{H^{k+1/2}(\G_b)}&\le C\big(1+\big|\al^{-1}\cT''-T_{2,c}\big|_{H^{k-1/2}(\G_b)}+|a_{x, i}|_{L^\infty}+|a_{z, i}|_{L^\infty}\big)|\al \wt\kappa_b|_{H^{k+1/2}(\G_b)}\\
&\le C\big(1+\big|\al^{-1}\cT''-T_{2,c}\big|_{H^{k-1/2}(\G_b)}+|a_{x, i}|_{L^\infty}+|a_{z, i}|_{L^\infty}\big),
\end{split}
\]
where $|\al \wt\kappa_b|_{H^{k+1/2}(\G_b)}$ can be controlled by the  the initial bound of in $\Lam(T,k)$ and the bottom $\G_b$.

As a result, going back to the estimate of $u_r$ and applying Lemm \ref{trace}, we derive
\[
\begin{split}
\|u_r\|_{H^{k+1}}\le &C\big(1+\big|\al^{-1}\cT''-T_{2,c}\big|_{H^{k-1/2}}+\big|\al^{-1}\cT''-T_{2,c}\big|_{H^{k-1/2}(\G_b)}+|a_{x, i}|_{L^\infty}+|a_{z, i}|_{L^\infty}\big)\times\\
& \big(1+|\al w|_{H^{k+1/2}}\big)\\
\le & C\big(1+\big\|\al^{-1}\cT''-T_{2,c}\big\|_{H^{k}}+|a_{x, i}|_{L^\infty}+|a_{z, i}|_{L^\infty}\big)\big(1+|\al w|_{H^{k+1/2}}\big)
\end{split}
\]

\bigskip

{\bf Step 2}. Retrieve the estimate of $u_i$ from $u_r$. {\it Different from the infinite-depth case, direct computations show that the $L^2$ norm of $u_i$ cannot be bounded by $\|u_r\|_{L^2}$.} But  we can  still have higher-order estimates for $u_i$. 
Applying C-R equations, we have
\[
\|\na u_i\|_{H^k}=\|\na u_r\|_{H^k}\le C\big(1+\big\|\al^{-1}\cT''-T_{2,c}\big\|_{H^{k}}+|a_{x, i}|_{L^\infty}+|a_{z, i}|_{L^\infty}\big)\big(1+|\al w|_{H^{k+1/2}}\big).
\]

As a result, going back to the estimate of the holomorphic function $ \f{\cT''}{i \cT'}$, we need to take one more derivative and set
\[
H:= \f{\cT'''}{i\cT'}-\f{\cT''i\cT''}{(i\cT')^2}=\Big(\f{\cT''}{i \cT'}\Big)'=\dtx u_r+i \dtx u_i.
\]
Similarly as before, the corresponding weighted limit for $\cT'''$ is denoted by 
\[
T_{3,c}=\chi_{l, 2}T_{l,1}+\beta_r^2\chi_{r,2} T_{r,1}
\]
with $\chi_{i, 2}$ defined in Lemma \ref{minus norm dx}.
We plan to derive the estimate of $\cT'$ from $\dtx u_r, \dtx u_i$. Considering the weighted limits of derivatives of $\cT$ as above and rewriting $H$ into the following form
\[
H=\f{\cT'''-\al T_{3,c}}{i\cT'}-\f{i\cT''(\cT''-\al T_{2,c})}{(i\cT')^2}-\f{i\al (T_{2,c}\cT''-T_{3,c}\cT')}{(i\cT')^2},
\]
we arrive at
\begin{equation}\label{T3 exp}
\begin{split}
\cT'''-\al T_{3,c}=&\f{i\cT''(\cT''-\al T_{2,c})}{i\cT'}+\f{i\al T_{2,c}(\cT''-\al T_{2,c})}{i\cT'}-\f{i\al T_{3,c}(\cT'-\al T_{1,c})}{i\cT'}\\
&+\f{i\al^2(T^2_{2,c}-T_{3,c}T_{1,c})}{i\cT'}+i\cT'(\dtx u_r+i \dtx u_i).
\end{split}
\end{equation}
Noticing that 
\[
\begin{split}
T^2_{2,c}-T_{3,c}T_{1,c}&=(\chi_{l, 1}T_{l,1}-\beta_r\chi_{r,1} T_{r,1})^2-(\chi_{l, 2}T_{l,1}+\beta_r^2\chi_{r,2} T_{r,1})(\chi_{l}T_{l,1}+\chi_{r} T_{r,1})\\
&=(\chi_{l,1}^2-\chi_{l,2}\chi_l)T^2_{l,1}-\beta_r^2(\chi^2_{r,1}-\chi_{r,2}\chi_r)T^2_{r,1},
\end{split}
\]
we conclude 
\[
\begin{split}
\|\al^{-1}\cT'''-T_{1,c}\|_{H^k}\le& \Big\|\f{(\al^{-1}\cT''-T_{2,c})^2}{\cT'}\Big\|_{H^k}+2\Big\|\f{T_{2,c}(\cT''-\al T_{2,c})}{\cT'}\Big\|_{H^k}+\Big\|\f{ T_{3,c}(\cT'-\al T_{1,c})}{\cT'}\Big\|_{H^k}\\
&+\Big\|\f{\al(T^2_{2,c}-T_{3,c}T_{1,c})}{\cT'}\Big\|_{H^k}+\big\|\al^{-1}\cT'(\dtx u_r+i \dtx u_i)\big\|_{H^k}.
\end{split}
\]
Applying Proposition \ref{CM estimate},  Lemma \ref{minus norm dx} and the estimates of $u_r, u_i$ above leads to 
\begin{equation}\label{T' higher order mid}
\begin{split}
&\|\al^{-1}\cT'''-T_{1,c}\|_{H^k}\\
&\le C\big[\big\|\al^{-1}\cT''- T_{2,c}\big\|^2_{H^k}+\big\|\al^{-1}\cT'- T_{1,c}\big\|^2_{H^k}+|a_{x, i}|^2_{L^\infty}+|a_{z, i}|^2_{L^\infty}+\|\dtx u_r\|_{H^k} +\|\dtx u_i\|_{H^k}\big]\\
&\le  C\big[\big\|\al^{-1}\cT'-T_{1,c}\big\|^2_{H^{k+1}}+|a_{x, i}|^2_{L^\infty}+|a_{z, i}|^2_{L^\infty}+\big(1+\big\|\al^{-1}\cT'-T_{1,c}\big\|_{H^{k+1}} +|a_{x, i}|_{L^\infty}\\
&\quad+|a_{z, i}|_{L^\infty}\big) \big(1+|\al w|_{H^{k+1/2}}\big)\big].
\end{split}
\end{equation}

On the other hand, considering $H^{k-1}$ estimate for $\al^{-1}\cT'''-\chi_lT_{1,c}$ in a similar way as above (where one $\cT''$ from the first term on the right side of \eqref{T3 exp} is bounded by Proposition \ref{CM estimate}), we obtain 
\[
\begin{split}
&\|\al^{-1}\cT'''-T_{3,c}\|_{H^{k-1}}\\
&\le C \big[(1+|a_{x, i}|_{L^\infty}+|a_{z, i}|_{L^\infty})\big\|\al^{-1}\cT'-T_{1,c}\big\|_{H^{k}}+|a_{x, i}|^2_{L^\infty}+|a_{z, i}|^2_{L^\infty}+\big(1+\big\|\al^{-1}\cT'-T_{1,c}\big\|_{H^{k}}\\
&\quad+|a_{x, i}|_{L^\infty}+|a_{z, i}|_{L^\infty}\big)\ \big(1+|\al w|_{H^{k-1/2}}\big)\big].
\end{split}
\]
Moreover, applying an interpolation in Sobolev spaces leads to
\[
\begin{split}
 \|\al^{-1}\cT'''-T_{3,c}\|_{H^{k-1}}
\le 
 & \epsilon\,(1+|a_{x, i}|_{L^\infty}+|a_{z, i}|_{L^\infty})\big\|\al^{-1}\cT'- T_{1,c}\big\|_{H^{k+1}}\\
 &+Q\big(\epsilon^{-1}, \big\|\al^{-1}\cT'- T_{1,c}\big\|_{L^2}, |a_{x, i}|_{L^\infty}, |a_{z, i}|_{L^\infty}\big) \big(1+|\al w|_{H^{k-1/2}}\big)
\end{split}
\] 
where $\epsilon>0$ is a small constant.  Consequently, we obtain by  Lemma \ref{minus norm dx} that
\[
\big\|\al^{-1}\cT'- T_{1,c}\big\|_{H^{k+1}}\le Q\big(\big\|\al^{-1}\cT'- T_{1,c}\big\|_{L^2}, |a_{x, i}|_{L^\infty}, |a_{z, i}|_{L^\infty}\big) \big(1+|\al w|_{H^{k-1/2}}\big),
\]
where the bound $L$ in Definition \ref{initial value bound} is applied to $|a_{x, i}|_{L^\infty}+|a_{z, i}|_{L^\infty}$.
As a result, substituting this estimate into \eqref{T' higher order mid} and applying Lemma \ref{minus norm dx} again, we obtain the desired estimate for $\cT$. The estimate for $T_{1,c}$ follows directly from \eqref{ax az}.

\end{proof}

We present a useful estimate which has appeared in the proof above. 
\begin{corollary}\label{estimate of x z}
Let $f$ be a function smooth enough in $\cSt$ and the integer $k\ge 2$. Then there hold
\[
\|\dtx^2x f\|_{H^k}+\|\dtx^2z f\|_{H^k}\le Q(\cT', w, k-1/2)\|\al f\|_{H^k}
\]
and
\[
|\dtx^2x f|_{H^k(\G_j)}+|\dtx^2z f|_{H^k(\G_j)}\le Q(\cT', w, k)|\al f|_{H^k(\G_j)}
\]
where $j=t, b$.
\end{corollary}
\begin{proof} 
In fact, similar argument as in the proof above shows that
\[
\begin{split}
\|\dtx^2x f\|_{H^k}&\le \big\|(\al^{-1}\dtx^2 x-a_{x, 1})\al f\big\|_{H^k}+\|a_{x, 1} \al f\big\|_{H^k}\\
&\le C\big(\|\al^{-1}\dtx x-a_{x, 1}\|_{H^{k+1}}+|a_{x, i}|_{L^\infty}\big)\|\al f\|_{H^k},
\end{split}
\]
where $a_{x, 1}$ is defined in Lemma \ref{minus norm dx} with $u=\dtx x$.
Applying Proposition \ref{dxT estimate}  leads to the desired estimate immediately. As a result, the remaining estimates can be  proved in a similar way.

\end{proof}

Meanwhile,  considering $\dtx z$ component of $\cT'$ on $\G_t$, we obtain  immediately the following estimate.
\begin{corollary}\label{zeta to w estimate} Let the integer $k\ge 1$. Then one has
\[
\big|\al^{-1}\dtx\tzeta\big|_{H^{k+3/2}}\le Q(\cT, w, k+1/2).
\]
\end{corollary}

\begin{proof} 
In fact, noticing that  
\[
a_z|_{\tz=0}=0,\quad \dtx z|_{\tz=0}=\dtx\tzeta
\] 
 and applying Lemma \ref{trace} lead to 
\[
\big|\al^{-1}\dtx \tzeta\big|_{H^{k+3/2}}\le C\|\al^{-1}\cT'-T_{1,c}\|_{H^{k+2}}.
\]
Therefore, the proof is finished by applying Proposition \ref{dxT estimate}. 

\end{proof}

Besides, we also need another ``local" higher-order estimate of $\cT$ near the bottom $\G_b$, which will be used in the a priori estimate. Notice that these kinds of higher-order estimate work thanks to the separation of $\G_t, \G_b$ for $\cSt$,  which is handled by localization near the bottom as in classical water waves, see  \cite{ABZ2009}.

\begin{corollary}\label{higher order local}Let the integer $k\ge 1$. There holds
\[
\big|\al^{-1}\cT'-T_{1,c}\big|_{H^{k+2}(\G_b)}\le Q(\cT', w, k)
\]
\end{corollary}
\begin{proof} In fact, we only need to modify  the proof of Proposition \ref{dxT estimate}. To begin with, we set
\[
\chi_b(\tz)=\begin{cases} 1,\quad \tz=-\om,\\
0, \quad \tz\ge -\om+\delta
\end{cases}
\]
for some small constant $\delta>0$. As a result, we find that $\chi_b u_r$ with $u_r$ from \eqref{holomorphic T} satisfies
\[
\begin{cases}
\Delta \chi_bu_r=[\Delta, \chi_b]u_r\qquad \hbox{in}\quad \cSt,\\
\chi_b u_r|_{\G_t}=0,\quad  u_r|_{\G_b}=\wt\kappa_b|\cT'|\,|_{\G_b}.
\end{cases}
\]
We have by standard elliptic estimates, Proposition \ref{CM estimate} and Corollary \ref{estimate of x z} that
\[
\begin{split}
\|\chi_bu_r\|_{H^{k+3/2}}&\le C\big(\|[\Delta, \chi_b]u_r\|_{H^{k-1/2}}+\big|\wt\kappa_b|\cT'|\big|_{H^{k+1}(\G_b)}\big)\\
&\le C\big(\|u_r\|_{H^{k+1/2}}+Q(\cT', w, k)\big),
\end{split}
\]
while notice that the sufficient smoothness of the bottom $\G_b$ of $\Om$ is used here. 

Moreover, applying the estimate of $u_r$ from the proof  of Proposition \ref{dxT estimate} and using an interpolation, we obtain the following estimate
\[
\|\chi_bu_r\|_{H^{k+3/2}}\le Q(\cT', w, k),
\]
which implies immediately
\[
\|u_r\|_{H^{k+3/2}(\mathcal S_b)}\le Q(\cT', w, k)
\]
where $\mathcal S_b\subset \cSt$ is a flat strip near $\G_b$ with the height smaller than $\delta$. 

Besides, we also derive 
\[
\|\na u_i\|_{H^{k+1/2}(\mathcal S_b)}=\|\na u_r\|_{H^{k+1/2}(\mathcal S_b)}\le Q(\cT', w, k).
\]

Therefore, repeating estimate \eqref{T' higher order mid} from the proof  of Proposition \ref{dxT estimate} and applying the new estimates for $u_r, u_i$ above, we derive a higher-order estimate of $\cT$ in $\mathcal S_b$:
\[
\begin{split}
&\big\|\al^{-1}\cT'''-T_{3,c}\big\|_{H^{k+1/2}(\mathcal S_b)}\\
&\le  C\big[\big\|\al^{-1}\cT'- T_{1,c}\big\|^2_{H^{k+3/2}(\mathcal S_b)}+|a_{x, i}|^2_{L^\infty}+|a_{z, i}|^2_{L^\infty}+\|\dtx u_r\|_{H^{k+1/2}(\mathcal S_b)} +\|\dtx u_i\|_{H^{k+1/2}(\mathcal S_b)}\big]\\
&\le Q(\cT', w, k)
\end{split}
\]
In the end, applying Proposition \ref{dxT estimate} and a trace theorem again, the proof is finished.

\end{proof}

\bigskip

 Next, we investigate the time dependence of the conformal mapping, where a loss of $1/2$ derivative appears due to the space derivative in \eqref{kinematic condi}.
\begin{proposition}\label{dtT estimate} Let  $\ga$ satisfy
\[
0<\ga+1\le \min\{\f{2\pi}{\om}, \f{2\pi}{\om_r}\}.
\] Then there holds
 for $k\ge 0$ that
\[
\|\al^{-(\ga+3)/2}\dt \cT\|_{H^{k+2}}\le  Q(\cT', w, k+1/2)\big(\big|\al^{-\ga-2}\wt{\dt \zeta}-a_{\zeta_t}\big|_{H^{k+3/2}}+|a_{\zeta_t, i}|_{L^\infty}\big),
\]
where  $a_{\zeta_t}$ is the weighted limit defined in \eqref{a zeta t def}.
\end{proposition}
\begin{proof} The proof here is similar as the proof of the previous proposition. In fact, inspired by \cite{SWZ}, we start with computing the following holomorphic function 
\[
\f{\dt \cT}{\cT'}=\f{\dt x+i\dt z}{\dtx x+i \dtx z}=\f1{|\cT'|^2}\big[\big(\dtx x\dt x+\dtx z\dt z\big)+i\big(\dtx x\dt z-\dtx z \dt x\big)\big].
\]
We set again 
\[
\f{\dt \cT}{\cT'}:= u_r+i u_i.
\]

{\bf Step 1.} $u_i$ estimate. Since  
\[
\big(\dt z-(\dtx x)^{-1}\dtx z\dt x\big)|_{\G_t: \tz=0}=\wt{\dt \zeta}\qquad\hbox{and}\quad 
\big(\dt z-(\dtx x)^{-1}\dtx z\dt x\big)|_{\G_b: \tz=-\om}=0,
\]
we obtain immediately the Dirichlet system of $u_i$:
\[
\begin{cases}
\Delta u_i=0\qquad \hbox{in}\quad \cSt,\\
u_i|_{\G_t}=|\cT'|^{-2}\dtx x\wt{\dt \zeta},\quad u_i|_{\G_b}=0.
\end{cases}
\]
Applying Proposition \ref{Dirichlet system estimate}  leads to 
\[
\big\|\al^{-(\ga+1)/2}u_i\big\|_{H^{k+2}}\le C\big|\al^{-(\ga+1)/2}|\cT'|^{-2}\dtx x\wt{\dt \zeta}\big|_{H^{k+3/2}},
\]
where one requires
\[
\f{|\ga+1|}2\le \min\{\f\pi{\om}, \f\pi{\om_r}\}.
\]
For the right-side term, direct computations and applying Proposition \ref{CM estimate} and Corollary \ref{estimate of x z} lead to
\[
\big|\al^{-(\ga+1)/2}|\cT'|^{-2}\dtx x\wt{\dt \zeta}\big|_{H^{k+3/2}}\le Q(\cT', w, k+1/2)\big|\al^{-(\ga+1)/2}\al^{-1}\ttze\big|_{H^{k+3/2}},
\]
where the right side can be expanded as below following the spirit of Corollary \ref{estimate of x z}:
\[
\big|\al^{-(\ga+1)/2}\al^{-1}\ttze\big|_{H^{k+3/2}}\le \big| \al^{(\ga+1)/2}(\al^{-\ga-2}\ttze-a_{\zeta_t})\big|_{H^{k+3/2}}+\big|\al^{(\ga+1)/2}a_{\zeta_t}\big|_{H^{k+3/2}}.
\]
Noticing that when $\ga+1> 0$, the second term on the right side above is integrable so that we conclude the estimate of $u_i$:
\[
\big\|\al^{-(\ga+1)/2}u_i\big\|_{H^{k+2}}\le Q(\cT', w, k-1)
\big(\big|\al^{-\ga-2}\ttze-a_{\zeta_t}\big|_{H^{k+3/2}}+|a_{\zeta_t}|_{L^\infty}\big).
\]

\bigskip 

{\bf Step 2. } $u_r$ estimate. In fact, using C-R equations leads directly to  the higher-order estimate of $u_r$:
\[
\big\|\al^{-(\ga+1)/2}\na u_r\big\|_{H^{k+1}}= \big\|\al^{-(\ga+1)/2}\na u_i\big\|_{H^{k+1}}\le Q(\cT', w, k+1/2)
\big(\big|\al^{-\ga-2}\ttze-a_{\zeta_t}\big|_{H^{k+3/2}}+|a_{\zeta_t}|_{L^\infty}\big).
\]

It remains to prove the weighted $L^2$ estimate of $u_r$. We plan to recover it by using the Laplace Transform in $\cSt$. 
To begin with, C-R equations imply the following system of $u_r$:
\[
\begin{cases}
\Delta u_r=0\qquad \hbox{in}\quad \cSt,\\
\dtz u_r|_{\G_t}=-\dtx u_i|_{\G_t},\quad \dtz u_i|_{\G_b}=-\dtx u_i|_{\G_b}=0.
\end{cases}
\]
Recalling the Laplace Transform on a sufficiently smooth function $u=u(\tx)$: 
\[
\check{u}(\lam)=\mathcal L u(\lam)=\int_\RR e^{-\lam \tx}u(\tx)d\tx,\qquad \forall \lam\in \mathbb C,
\]
we have immediately the following system of $\check u_r(t,\lam, \tz)$:
\[
\begin{cases}
\Delta \check u_r=0,\qquad \tz\in[-\om, 0],\\
\dtz \check u_r|_{\tz=0}=-\dtx \check u_i|_{\tz=0},\quad \dtz \check u_r|_{\tz=-\om}=0.
\end{cases}
\]
Solving the system above directly, one has
\[
\check u_r=\f{\cos \big(\lam(\om+\tz)\big)}{\sin(\lam\om)}\check u_i|_{\tz=0}=\f{e^{i\lam(\om+\tz)}+e^{-i\lam(\om+\tz)}}{e^{i\lam\om}-e^{-i\lam\om}}\check u_i|_{\tz=0}.
\]

Therefore, denoting by 
\[
\ga_1=\f{\ga+1}2,
\] 
the weighted $L^2$ norm of $u_r$ is decomposed as below
\begin{equation}\label{u r middle}
\begin{split}
\|\al^{-\ga_1}u_r\|_{L^2}&\le \|\al^{-\ga_1}\chi_l u_r\|_{L^2}+\|\al^{-\ga_1}\chi_r u_r\|_{L^2}+\|\al^{-\ga_1}(1-\chi_l -\chi_r)u_r\|_{L^2}\\
&\le C\big(\|e^{-\ga_1\tx}u_r\|_{L^2}+\|e^{\ga_1\beta_r\tx}u_r\|_{L^2}\big),
\end{split}
\end{equation}
where both the two weighted terms on the right side above need to be handled.

For the first term,  applying the Parseval equality for the Laplace Transform, one obtains
\[
\big\|e^{-\ga_1\tx}u_r\big\|^2_{L^2}=\int^0_{-\om}\f 1{2\pi i}\int_{Re\lam=\ga_1}\big|\check u_r(\lam, \tz)\big|^2d\lam d\tz
\]
with 
\[
\lam=\ga_1+i\xi.
\]
Substituting the expression of $\check u_r$ into this integral, we derive
\[
\begin{split}
\big\|e^{-\ga_1\tx}u_r\big\|^2_{L^2}&=\int^0_{-\om}\f 1{2\pi i}\int_{Re\lam=\ga_1}\Big|\f{\cos \big(\lam(\om+\tz)\big)}{\sin(\lam\om)}\check u_i|_{\tz=0}\Big|^2d\lam d\tz\\
&:=\int^0_{-\om}\f 1{2\pi i}\int_{Re\lam=\ga_1} B\,d\lam d\tz,
\end{split}
\]
where 
\[
\begin{split}
B&=\Big|\f{e^{i(\ga_1+i\xi)(\om+\tz)}+e^{-i(\ga_1+i\xi)(\om+\tz)}}{e^{i(\ga_1+i\xi)\om}-e^{-i(\ga_1+i\xi)\om}}\Big|^2\big|\check u_i|_{\tz=0}\big|^2\\
&=\Big|\f{e^{-\xi(\om+\tz)}e^{i\ga_1(\om+\tz)}+e^{\xi(\om+\tz)}e^{-i\ga_1(\om+\tz)}}{e^{-\xi\om}e^{i\ga_1\om}-e^{\xi\om}e^{-i\ga_1\om}}\Big|^2\big|\check u_i|_{\tz=0}\big|^2
\le C\big|\check u_i|_{\tz=0}\big|^2
\end{split}
\]
with the constant $C$ depending on $\om, \ga$.

Consequently, we arrive at the following estimate
\[
\big\|e^{-\ga_1\tx}u_r\big\|^2_{L^2}\le C\int^0_{-\om}\f 1{2\pi i}\int_{Re\lam=\ga_1}\big|\check u_i|_{\tz=0}\big|^2d\lam d\tz\le C \big|e^{-\ga_1\tx}u_i|_{\tz=0}\big|^2_{L^2}.
\]

Meanwhile, for the second term on the right side of \eqref{u r middle}, similar arguments as above lead to 
\[
\big\|e^{\ga_1\beta_r\tx}u_r\big\|^2_{L^2}\le C \big|e^{\ga_1\beta_r\tx}u_i|_{\tz=0}\big|^2_{L^2}.
\]
Substituting the boundary value 
\[
u_i|_{\tz=0}=|\cT'|^{-2}\dtx x\wt{\dt \zeta}
\]
into the two estimates above, applying Proposition \ref{CM estimate} and going back to \eqref{u r middle}, one derives
\[
\begin{split}
\big\|\al^{-\ga_1}u_r\big\|_{L^2}&\le C\big(\big|e^{-\ga_1\tx}\al^{-1}\ttze\big|_{L^2}+\big|e^{\ga_1\beta_r\tx}\al^{-1}\ttze\big|_{L^2}\big)\\
&\le C\big|\al^{-\ga_1}\al^{-1}\ttze\big|_{L^2}\le C\big(\big|\al^{-\ga-2}\ttze-a_{\zeta_t}\big|_{L^2}+|a_{\zeta_t}|_{L^\infty}\big),
\end{split}
\]
where $\ga_1=(\ga+1)/2>0$ is needed again.

\medskip

As a result, summing up all these estimates of $u_r$, we finally conclude that 
\[
\big\|\al^{-(\ga+1)/2} u_r\big\|_{H^{k+2}}\le Q(\cT', w, k+1/2)
\big(\big|\al^{-\ga-2}\ttze-a_{\zeta_t}\big|_{H^{k+3/2}}+|a_{\zeta_t}|_{L^\infty}\big).
\]

\medskip

{\bf Step 3. } End of the proof. Going back to the $\p_t\cT$ expression in the beginning of the proof, we can finish the proof thanks to Step 1-2 and Corollary \ref{estimate of x z}.

\end{proof}

\begin{remark}\label{dt T corner} Combining Proposition \ref{dtT estimate} with \eqref{T1c def}, \eqref{ax az} and  Proposition \ref{dxT estimate}, we conclude that the weighted limits
\[
\dt T_{1,i}=0,\,(i=l,r),\qquad\hbox{i.e.}\quad b'(t)=b_r'(t)=b_{r2}'(t)=0.
\]
This happens due to the fact that a higher-order  weight is used for  $\dt\cT$ in our settings.
\end{remark}

On the other hand, we  consider the weighted estimate for  $\ttze$, which is  a part of the boundary value of $\dt\cT$ and  deals with the relationship between $\wt D_t w$ and $\wt{\dt\zeta}$. Notice that  a loss of $1/2$ order space derivative still exists here.

\begin{proposition}\label{dt Ka estimate} Let the  integer $k\ge 0$ and
\[
a_{w_t}(t, \tx)=\chi_la_{w_t, l}+\chi_r a_{w_t, r}
\]
with 
\[
a_{w_t, l}(t)=\lim_{\tx\rightarrow-\infty} \al^{-\ga}\wt D_t w,\quad a_{w_t, r}(t)=\lim_{\tx\rightarrow+\infty} \al^{-\ga}\wt D_t w.
\]
Then the following estimate holds:
\[
\begin{split}
\big|\al^{-\ga-2}\wt{\dt\zeta}-a_{\zeta_t}\big|_{H^{k+3/2}}\le  &Q(\cT', \tPhi,   w, k+1/2)
\big(\big|\al^{-\ga}\wt D_t w-a_{w_t}\big|_{H^{k-1/2}}+\big|\al^{-\ga-2}\tv-a_v\big|_{H^{k+3/2}}\\
&+|a_v|_{L^\infty}+\big|\al^{-\ga-2}\ttze-a_{\zeta_t}\big|_{L^2}
+|a_{\zeta_t, i}|_{L^\infty}\big).
\end{split}
\]
\end{proposition}

\begin{proof} 
To begin with, we recall another expression for $D_t\ka$ from \cite{SZ}:
\[
D_t\ka=-\Delta_{\G_t}(v\cdot n_t)-|\na_{\tau_t}n_t|^2(v\cdot n_t)+\na_{\tau_t}\na_{v^\top}n_t\cdot \tau_t-(\na_{\tau_t}v^\top)^\top\cdot \na n_t\cdot \tau_t\qquad\hbox{on}\quad \G_t.
\]
Applying the parametrization $z=\zeta$ and noticing that
\[
v\cdot n_t=(1+|\dx\zeta|^2)^{-\f12} \dt\zeta
\]
 by the kinematic condition \eqref{kinematic condi}, we rewrite the equation above as
 \[
 D_t\ka=-\na_{\tau_t}\na_{\tau_t}\big((1+|\dx\zeta|^2)^{-\f12} \dt\zeta\big)+F(\p v, \p^2 n_t, \p \tau_t),
 \]
where
\[
F(\p v, \p^2 n_t, \p \tau_t)=-|\na_{\tau_t}n_t|^2(v\cdot n_t)+\na_{\tau_t}\na_{v^\top}n_t\cdot \tau_t-(\na_{\tau_t}v^\top)^\top\cdot \na n_t\cdot \tau_t
\]
with at most two derivatives in one term.

Pulling this equation onto the upper boundary of $\cSt$ by $\cT$ and applying \eqref{na X matrix Gt}, \eqref{Dt transform}, we have
\[
\wt D_t w=\wt D_t\wt\ka=-(\dtx x)^{-1}b_\zeta^{-\f12}\dtx\big[(\dtx x)^{-1}b_\zeta^{-\f12}\dtx\big(b_\zeta^{-\f12} \ttze\big)\big]+\wt F
\]
where
\[
(\na_{\tau_t})\circ \cT|_{\G_t}=(\na_X\wt X)^T\circ \cT\,\wt\tau_t\cdot \na=(\dtx x)^{-1}b_\zeta^{-\f12}\dtx
\] and we write
\begin{equation}\label{b zeta}
b_\zeta:=1+|(\dtx x)^{-1}\dtx\tzeta|^2.
\end{equation}

Recalling \eqref{bdy condition ax az}, \eqref{bdy condition ax az right}, the definition of $a_{\zeta_t}$ in Proposition \ref{dtT estimate} and  the space of $\al\tzeta$ and taking the limit of $\al^{-\ga}\wt D_t w$, we find immediately 
\[
\begin{split}
a_{w_t, l}&=\lim_{\tx\rightarrow-\infty} \al^{-\ga}\wt D_t w=-(\ga+1)(\ga+2)a_{\zeta_t, l}(t),\\
a_{w_t, r}&=\lim_{\tx\rightarrow+\infty} \al^{-\ga}\wt D_t w=-(\ga+1)(\ga+2)\beta_r^2b^2_r(t)a_{\zeta_t, r}(t).
\end{split}
\]
This means that $\al^{-\ga}\wt D_t w$ doesn't vanish near $\infty$ in general. 

\medskip

As a result, denoting by 
\[
a_{\zeta_t, x}=\chi_la_{\zeta_t, l}(t)+\chi_r b^2_r(t)a_{\zeta_t, r}(t)
\]
and applying $L^2(\RR)$ inner product on the expression $\al^{-\ga}\wt D_t w -a_{w_t}$ with  $\al^{-\ga}(\dtx x)^{-2}\ttze -a_{\zeta_t, x}$ leads to 
\begin{equation}\label{dt Ka integral}
\begin{split}
&\int_{\RR}\big(\al^{-\ga}\wt D_t w -a_{w_t}\big)\big(\al^{-\ga}(\dtx x)^{-2}\ttze -a_{\zeta_t, x}\big) d\tx\\
&=
-\int_{\RR}\Big(\al^{-\ga}(\dtx x)^{-1}b_\zeta^{-\f12}\dtx\big[(\dtx x)^{-1}b_\zeta^{-\f12}\dtx\big(b_\zeta^{-\f12} \ttze\big)\big]+a_{w_t}\Big)\,\big(\al^{-\ga}(\dtx x)^{-2}\ttze-a_{\zeta_t, x}\big) d\tx \\
&\quad +\int_{\RR}\al^{-\ga}\wt F\,\big(\al^{-\ga}(\dtx x)^{-2}\ttze-a_{\zeta_t, x}\big) d\tx\\
&:= A_1+A_2,
\end{split}
\end{equation}
and we deal with these two integrals one by one.  

\medskip

Frist,  applying integration by parts shows immediately that
\[
\begin{split}
A_1=&-\al^{-\ga}(\dtx x)^{-1}b_\zeta^{-\f12}(\dtx x)^{-1}b_\zeta^{-\f12}\dtx\big(b_\zeta^{-\f12} \ttze\big)
\big(\al^{-\ga}(\dtx x)^{-2}\ttze-a_{\zeta_t, x}\big) \big|^{+\infty}_{-\infty}\\
&+\int_{\RR}\dtx\big(\al^{-\ga}(\dtx x)^{-1}b_\zeta^{-\f12}\big)(\dtx x)^{-1}b_\zeta^{-\f12}\dtx\big(b_\zeta^{-\f12} \ttze\big)\big(\al^{-\ga}(\dtx x)^{-2}\ttze-a_{\zeta_t, x}\big)d\tx\\
&+\int_{\RR}\al^{-\ga}(\dtx x)^{-1}b_\zeta^{-\f12}(\dtx x)^{-1}b_\zeta^{-\f12}\dtx\big(b_\zeta^{-\f12} \ttze\big)\dtx\big(\al^{-\ga}(\dtx x)^{-2}\ttze-a_{\zeta_t, x}\big)d\tx\\
&-\int_{\RR}a_{w_t}\big(\al^{-\ga}(\dtx x)^{-2}\ttze-a_{\zeta_t, x}\big)d\tx\\
:=& A_{11}+A_{12}+A_{13}+A_{14},
\end{split}
\]
where direct computations as before lead to
\[
A_{11}=0.
\]

\noindent - $A_{12}$ part. Expanding the derivative  $\dtx(\cdot)$ terms directly leads to 
\[
\begin{split}
A_{12}=&\int_{\RR}\dtx\big(\al^{-\ga}(\dtx x)^{-1}\big)b_\zeta^{-\f12}(\dtx x)^{-1}b_\zeta^{-\f12}b_\zeta^{-\f12} \dtx\ttze\big(\al^{-\ga}(\dtx x)^{-2}\ttze-a_{\zeta_t, x}\big)d\tx\\
&+\int_{\RR}\dtx\big(\al^{-\ga}(\dtx x)^{-1}\big)b_\zeta^{-\f12}(\dtx x)^{-1}b_\zeta^{-\f12}(\dtx b_\zeta^{-\f12}) \ttze\big(\al^{-\ga}(\dtx x)^{-2}\ttze-a_{\zeta_t, x}\big)d\tx\\
&+\int_{\RR}\al^{-\ga}(\dtx x)^{-1}\dtx(b_\zeta^{-\f12})(\dtx x)^{-1}b_\zeta^{-\f12}\dtx \big(b_\zeta^{-\f12} \ttze\big)\big(\al^{-\ga}(\dtx x)^{-2}\ttze-a_{\zeta_t, x}\big)d\tx\\
:=& A_{121}+A_{122}+A_{123},
\end{split}
\]
where one notes that $\dtx (b_\zeta^{-\f12}) $ in $A_{122}, A_{123}$  can be estimated directly by some corresponding Sobolev norm of $\al^{-1}\tzeta$. 

Similarly as before, adding the ``limit part" to $A_{121}$ and noticing that 
\[
\dtx\big(\al^{-\ga}(\dtx x)^{-1}\big)\approx -(\ga+1)\al^{-\ga-1},\quad \tx\rightarrow-\infty
\]
where the computations are similar as those for $a_{w_t}$, 
 we obtain
\[
\begin{split}
A_{121}=&\int_{\RR}b_\zeta^{-\f32}\Big(\dtx\big(\al^{-\ga}(\dtx x)^{-1}\big)(\dtx x)^{-1}\dtx\ttze-a_{w_t}\Big)\big(\al^{-\ga}(\dtx x)^{-2}\ttze-a_{\zeta_t, x}\big)d\tx\\
&+\int_{\RR}b_\zeta^{-\f32}a_{w_t}\big(\al^{-\ga}(\dtx x)^{-2}\ttze-a_{\zeta_t, x}\big)d\tx\\
:=&B_1+B_2.
\end{split}
\]
As a result, similar arguments and applying Corollary \ref{estimate of x z} as before lead to 
\[
\begin{split}
|B_1|\le& Q(\cT', w, 1)\big(\big|\al^{-\ga-2}\dtx\ttze-a_{\zeta_t, 1}\big|_{L^2}+|a_{\zeta_t}|_{L^\infty}\big)\big(\big|\al^{-\ga-2}\ttze-a_{\zeta_t}\big|_{L^2}+|a_{\zeta_t}|_{L^\infty}\big)
\end{split}
\]
with $a_{\zeta_t, 1}$ defined in Lemma \ref{minus norm dx}.

Moreover, $A_{122}$ can be handled in a similar way as $B_1$  with the help of Corollary \ref{zeta to w estimate} where  we take $L^2$ norm on $\dtx (b_\zeta^{-\f12}) $ part:
\[
\begin{split}
|A_{122}|
\le& Q(\cT', w, 1)\big(|\al^{-1}\dtx \tzeta|_{L^2}+|\al^{-1}\dtx^2\tzeta|_{L^2}\big)\big(\big|\al^{-\ga-2}\ttze-a_{\zeta_t}\big|_{L^\infty}+|a_{\zeta_t}|_{L^\infty}\big)
\big(\big|\al^{-\ga-2}\ttze-a_{\zeta_t}\big|_{L^2}\\
&+|a_{\zeta_t}|_{L^\infty}\big)\\
\le &Q(\cT', w, 1)\big(\big|\al^{-\ga-2}\ttze-a_{\zeta_t}\big|_{H^1}+|a_{\zeta_t}|_{L^\infty}\big)\big(\big|\al^{-\ga-2}\ttze-a_{\zeta_t}\big|_{L^2}
+|a_{\zeta_t}|_{L^\infty}\big)
\end{split}
\]

For the last term in $A_{12}$, expanding  $\dtx$ part and proceeding as above show that
\[
\begin{split}
A_{123}=&\int_{\RR}\dtx(b_\zeta^{-\f12})b_\zeta^{-1}\Big(\al^{-\ga}(\dtx x)^{-2}\dtx\ttze-(\ga+2)\chi_la_{\zeta_t,l}-(\ga+2)\beta_rb_r^2(t)a_{\zeta_t, r}\Big)\times\\
&\big(\al^{-\ga}(\dtx x)^{-2}\ttze-a_{\zeta_t, x}\big)d\tx
+\int_{\RR}\dtx(b_\zeta^{-\f12})b_\zeta^{-1}\big((\ga+2)\chi_la_{\zeta_t,l}+(\ga+2)\beta_rb_r^2(t)a_{\zeta_t, r}\big)\times\\
&\big(\al^{-\ga}(\dtx x)^{-2}\ttze-a_{\zeta_t, x}\big)d\tx
+\int_{\RR}\al^{-\ga}(\dtx x)^{-1}\dtx(b_\zeta^{-\f12})(\dtx x)^{-1}b_\zeta^{-\f12}\dtx(b_\zeta^{-\f12}) \ttze\times
\\
&\big(\al^{-\ga}(\dtx x)^{-2}\ttze-a_{\zeta_t, x}\big)d\tx\\
:=&B_3+B_4+B_5,
\end{split}
\]
where $B_3, B_4, B_5$ can be handled as above and hence the estimates are omitted here.

\medskip

\noindent - $A_{13}$ part. Expanding the derivative term $\dtx\big(b_{\zeta}^{-\f12}\ttze\big)$ and adding corresponding limit part lead to
\[
\begin{split}
A_{13}=&\int_{\RR}b_\zeta^{-1}\dtx(b_\zeta^{-\f12})\big(\al^{-\ga}(\dtx x)^{-2}\ttze-a_{\zeta_t, x}\big)\dtx\big(\al^{-\ga}(\dtx x)^{-2}\ttze-a_{\zeta_t, x}\big)d\tx\\
&+\int_{\RR}b_\zeta^{-1}\dtx(b_\zeta^{-\f12})a_{\zeta_t}\dtx\big(\al^{-\ga}(\dtx x)^{-2}\ttze-a_{\zeta_t, x}\big)d\tx\\
&+\int_{\RR}b_\zeta^{-\f32}\dtx\big(\al^{-\ga}(\dtx x)^{-2}\ttze-a_{\zeta_t, x}\big)\dtx\big(\al^{-\ga}(\dtx x)^{-2}\ttze-a_{\zeta_t, x}\big)d\tx\\
&+\int_{\RR}b_\zeta^{-\f32}\chi_l'a_{\zeta_t}\dtx\big(\al^{-\ga}(\dtx x)^{-2}\ttze-a_{\zeta_t, x}\big)d\tx\\
&-\int_{\RR}b_\zeta^{-\f32}\Big(\dtx\big(\al^{-\ga}(\dtx x)^{-2}\big)\ttze+(\ga+2)\chi_la_{\zeta_t,l}+(\ga+2)\beta_rb_r^2(t)a_{\zeta_t, r}\Big)\dtx\big(\al^{-\ga}(\dtx x)^{-2}\ttze-a_{\zeta_t, x}\big)d\tx\\
&+\int_{\RR}b_\zeta^{-\f32}\big((\ga+2)\chi_la_{\zeta_t,l}+(\ga+2)\beta_rb_r^2(t)a_{\zeta_t, r}\big)\dtx\big(\al^{-\ga}(\dtx x)^{-2}\ttze-a_{\zeta_t, x}\big)d\tx\\
:=& B_6+\cdots+B_{11},
\end{split}
\]
where $B_6, B_7, B_9, B_{10}$ can be handled similarly as above and the estimates are omitted again. 

Moreover, integrating by parts in $B_{11}$ leads to 
\[
|B_{11}|\le Q(\cT', w, 1)|a_{\zeta_t}|_{L^\infty}\big(\big|\al^{-\ga-2}\ttze-a_{\zeta_t}\big|_{L^2}
+|a_{\zeta_t}|_{L^\infty}\big)
\]

For the main non-negative term $B_8$, one obtains directly by Corollary \ref{estimate of x z} and Corollary \ref{zeta to w estimate}  the following estimate:
\begin{equation}\label{B8 term}
\begin{split}
\big|\al^{-\ga-2}\ttze-a_{\zeta_t}\big|^2_{H^1}\le&Q(\cT', w, 1)\big(B_8+|a_{\zeta_t}|^2_{L^\infty}\big)
+\big|\al^{-\ga-2}\ttze-a_{\zeta_t}\big|^2_{L^2}.
\end{split}
\end{equation}

\medskip

Consequently, summing up all these estimates, we obtain for $A_1$ the following inequality
\[
\begin{split}
A_1\ge  A_{14}+B_2+B_8-Q\big(\cT',   w, 1\big)\big(\big|\al^{-\ga-2}\ttze-a_{\zeta_t}\big|_{H^1}+|a_{\zeta_t}|_{L^\infty}\big)\big(\big|\al^{-\ga-2}\ttze-a_{\zeta_t}\big|_{L^2}
+|a_{\zeta_t}|_{L^\infty}\big),
\end{split}
\]
where 
\[
\begin{split}
A_{14}+B_2=&\int_{\RR}\big(1-b_\zeta^{-\f32}\big)a_{w_t}\big(\al^{-\ga}(\dtx x)^{-2}\ttze-a_{\zeta_t, x}\big)d\tx
\end{split}
\]
and can be handled similarly as above thanks to $1-b_\zeta^{-\f32}$. 

Therefore, we conclude that
\begin{equation}\label{A1 term}
\begin{split}
 A_1\ge  B_8-Q\big(\cT',   w, 1\big)\big(\big|\al^{-\ga-2}\ttze-a_{\zeta_t}\big|_{H^1}+|a_{\zeta_t}|_{L^\infty}\big)\big(\big|\al^{-\ga-2}\ttze-a_{\zeta_t}\big|_{L^2}
+|a_{\zeta_t}|_{L^\infty}\big).
\end{split}
\end{equation}

\bigskip

\noindent - $A_2$ part. For the $\na_{\tau_t}\na_{v^\top}n_t\cdot \tau_t$ term in $F$, we perform an integration by parts again (especially for the higher-order estimate) and proceed as before. Meanwhile, notice that each term in  $\wt F$ always contains $\p \wt n_t$ when one adds the ``weighted limit" to $\tv$ term and there is at most order-one derivative for $\tv$. Here we define the weighted limit of $\tv$ as
\begin{equation}\label{av def}
a_v=\chi_l \lim_{\tx\rightarrow -\infty}\al^{-\ga-2}\tv+\chi_r \lim_{\tx\rightarrow +\infty}\al^{-\ga-2}\tv.
\end{equation}
 Consequently,  we have  the following estimate
\begin{equation}\label{A2 term}
\begin{split}
 A_2\le & Q\big(\cT',   w, 1\big)\big(\big|\al^{-\ga-2}\tv-a_v\big|_{H^1}+|a_v|_{L^\infty}\big)\big(\big|\al^{-\ga-2}\ttze-a_{\zeta_t}\big|_{L^2}
+|a_{\zeta_t}|_{L^\infty}\big)
\end{split}
\end{equation}
where Corollary \ref{zeta to w estimate} and Proposition \ref{v estimate} are applied.

\medskip

\noindent - Weighted $H^1$ estimate of $\ttze$. Summing up all these estimates \eqref{B8 term}, \eqref{A1 term} and \eqref{A2 term} above, we finally conclude that 
\[
\begin{split}
\big|\al^{-\ga-2}\ttze-a_{\zeta_t}\big|^2_{H^1}\le & Q\big(\cT',  w, 1\big)\big(\big|\al^{-\ga}\wt D_t w-a_{w_t}\big|^2_{L^2}+\big|\al^{-\ga-2}\tv-a_v\big|^2_{H^1}+|a_v|^2_{L^\infty}\\
&+\big|\al^{-\ga-2}\ttze-a_{\zeta_t}\big|^2_{L^2}
+|a_{\zeta_t}|^2_{L^\infty}\big).
\end{split}
\]

In the end, for higher-order derivative estimates, one can proceed in a standard way by taking $\dtx$ derivatives and using $L^2$ inner product from the beginning of this proof. Moreover, counting the order of derivatives for $\tv, \tzeta$,  we obtain the desired estimate.

\end{proof}

\subsection{Weighted estimates for the velocity and other quantities}\label{v estimate section}

Based on the estimates of $\cT$,  we now turn to the estimates for the velocity  $\tv=v\circ \cT$, which is also very delicate and involves several steps to reach the desired regularity. Due to the natural loss of $1/2$ order  space derivative from the kinematic condition \eqref{kinematic condi},  we firstly use the potential $\tPhi=\Phi\circ\cT$  to prove a relatively lower-order estimate for $v$.  

 In fact,  \eqref{kinematic condi} is rewritten into
\begin{equation}\label{dt zeta eqn}
\wt{\dt\zeta}=\big(1+|(\dtx x)^{-1}\dtx\tzeta|^2\big)^{\f12}\tv\cdot\wt n_t=-(\dtx x)^{-1}\dtz\tPhi\big|_{\tz=0}.
\end{equation}
We have the following Neumann system for $\tPhi$:
\begin{equation}\label{tPhi system}
\begin{cases}
\Delta \tPhi=0\qquad \hbox{in}\quad \cSt,\\
\dtz\tPhi\big|_{\G_t}=-(\dtx x)\wt{\dt\zeta},\quad \dtz\tPhi\big|_{\G_b}=0
\end{cases}
\end{equation}
where a loss of $1/2$ derivative for $\Phi$ takes place.

\begin{proposition}\label{v estimate} Let real $\ga>-1$, integer  $k\ge 1$, and $\tPhi$ satisfy \eqref{tPhi system}. 

(1) Let
\[
a_{\Phi_z, l}=\lim_{\tx\rightarrow -\infty}\al^{-\ga-3}\dtz\tPhi, \quad a_{\Phi_z, r}=\lim_{\tx\rightarrow +\infty}\al^{-\ga-3}\dtz\tPhi
\]
 and 
\[
a_{\Phi_z}(t,\tz)=\chi_l a_{\Phi_z, l}+\chi_r a_{\Phi_z, r}.
\]

Then the following estimates hold:
\[
\begin{split}
\|\al^{-\ga-3}\dtz\tPhi-a_{\Phi_z}\|_{H^{k+2}}\le Q(\cT', \dtz\tPhi, w, k+1/2)\big[\big|\al^{-\ga-2}\wt{\dt\zeta}- a_{\zeta_t}\big|_{H^{k+3/2}}+|a_{\zeta_t, i}|_{L^\infty}\big]
\end{split}
\]
and 
\[
|a_{\Phi_z, i}|_{W^{k+2, \infty}_{z}}\le C|a_{\zeta_t, i}|_{L^\infty}.
\]

(2) Let $a_v$ be defined in \eqref{av def}.
One also has 
\[
\begin{split}
\|\al^{-\ga-2}\tv-a_v\|_{H^{k+2}}\le Q(\cT', \na\tPhi, w, k+1/2)\big[\big|\al^{-\ga-2}\wt{\dt\zeta}- a_{\zeta_t}\big|_{H^{k+3/2}}+|a_{\zeta_t, i}|_{L^\infty}\big],
\end{split}\]
where $a_v$
satisfies 
\[
|a_v|_{W^{k+2, \infty}}\le Q_b |a_{\zeta_t, i}|_{L^\infty}.
\]
\end{proposition}

\begin{proof} {\bf Step 1.} Estimate of $\dtz\tPhi$.
Similarly as in the above estimate for $\cT$, we introduce
\[
w_\Phi=\al^{-\ga-3}\dtz\tPhi-a_{\Phi_z},
\]
which satisfies the following system
\[
\begin{cases}
\Delta w_\Phi=[\Delta, \al^{-\ga-3}]\dtz\tPhi-\Delta a_{\Phi_z}\qquad \hbox{in}\quad \cSt,\\
w_\Phi\big|_{\G_t}=-\al^{-\ga-3}(\dtx x)\wt{\dt\zeta}- a_{\Phi_z}|_{\tz=0},\quad w_\Phi\big|_{\G_b}=0.
\end{cases}
\]
Here $a_{\Phi_z}$ satisfies corresponding boundary conditions for $i=l, r$ naturally:
\begin{equation}\label{a Phi boundary}
 a_{\Phi_z, l}|_{\tz=0}=-a_{\zeta_t, l},\quad a_{\Phi_z, r}|_{\tz=0}=-b_r a_{\zeta_t, r},\quad
  a_{\Phi_z, i}|_{\tz=-\om}=0.
\end{equation}

Moreover, thanks to  Lemma \ref{minus norm dx} and 
\[
\Delta \dtz \tPhi=0\qquad\hbox{in}\quad \cSt, 
\] 
direct computations show that $a_{\Phi_z, l}$ satisfies the following equation
\begin{equation}\label{a phi eqn}
\dtz^2 a_{\Phi_z, l}=-(\ga+3)^2 a_{\Phi_z, l}.
\end{equation}


As a result, the ODE system \eqref{a phi eqn} and \eqref{a Phi boundary} of $a_{\Phi_z, l}$ can be solved similarly as for $a_{x, i}, a_{z, i}$ from \eqref{limit T'}, and we have immediately 
\begin{equation}\label{estimate a phi}
\begin{split}
a_{\Phi_z, l}=-a_{\zeta_t, l}\cos\big((\ga+3)\tz\big)-a_{\zeta_t, l}\cot\big((\ga+3)\om\big)\sin\big((\ga+3)\tz\big)
\end{split}
\end{equation} 
with the constant $C$ depending on $\ga, m$. Moreover, similar expression and estimate hold for $a_{\Phi_z, r}$ depending on $a_{\zeta_t}$ and $b_r$.

On the other hand, when $\tx\le -c_0$ (in order to be more clear),  recalling  the definition  of the weight $\al$ from \eqref{weight function} and direct computations again lead to 
\[
\begin{split}
\Delta w_\Phi=&(\al^{-\ga-3})''\dtz\tPhi+2(\al^{-\ga-3})'\dtx\dtz\tPhi-\chi_l''a_{\Phi_z, l}-\chi_l\dtz^2a_{\Phi_z, l}\\
=&(\ga+3)^2\al^{-\ga-3}\dtz\Phi-2(\ga+3)\al^{-\ga-3}\dtx\dtz\tPhi-\chi_l''a_{\Phi_z, l}-\chi_l\dtz^2a_{\Phi_z, l}\\
=&(\ga+3)^2\big(\al^{-\ga-3}\dtz\Phi-\chi_la_{\Phi_z,l}\big)-2(\ga+3)\big(\al^{-\ga-3}\dtx\dtz\tPhi-(\ga+3)\chi_{l, 1}a_{\Phi_z, l}\big)
-\chi_l''a_{\Phi_z, l}\\
&+2(\ga+3)^2(\chi_l-\chi_{l,1})a_{\Phi_z,l}
\end{split}
\]
where \eqref{a phi eqn} is applied for $a_{\Phi_z}$ and $\chi_{l,1}$ is defined in Lemma \ref{minus norm dx}.

As a result, applying Proposition \ref{Dirichlet system estimate}, Proposition \ref{CM estimate}, Corollary \ref{dx x estimate}, Corollary \ref{estimate of x z} and Lemma \ref{minus norm dx}, we obtain
\[
\begin{split}
\|w_\Phi\|_{H^{k+2}}
\le& C\big( \big\|\al^{-\ga-3}\dtz\Phi-a_{\Phi_z}\big\|_{H^{k}}+\big\|\al^{-\ga-3}\dtx\dtz\tPhi-a_{\Phi_z,1}\big\|_{H^{k}}+\|\chi_i''a_{\Phi_z}\|_{H^{k}}\\
&+\|(\chi_i-\chi_{i, 1})a_{\Phi_z}\|_{H^{k}}+\big|\al^{-\ga-3}(\dtx x)\wt{\dt\zeta}+a_{\Phi_z}\big|_{H^{k+3/2}}\big)\\
\le & C\|w_\Phi\|_{H^{k+1}}+Q(\cT', w, k+1/2)\big[\big|\al^{-\ga-2}\wt{\dt\zeta}-a_{\zeta_t}\big|_{H^{k+3/2}}+|a_{\zeta_t, i}|_{L^\infty}\big].
\end{split}
\]
Notice that here one has thanks to \eqref{a Phi boundary} that
\[
\begin{split}
&\big|\al^{-\ga-3}(\dtx x)\wt{\dt\zeta}+a_{\Phi_z}\big|_{H^{k+3/2}}\\
&\le \big|\al^{-1}\dtx x\chi_l(\al^{-\ga-2}\ttze- a_{\zeta_t, l})\big|_{H^{k+3/2}}+\big|\al^{-1}\dtx x\,\chi_r(\al^{-\ga-2}\ttze- a_{\zeta_t, r})\big|_{H^{k+3/2}}\\
&\quad +|a_{\zeta_t, l}\chi_l(\al^{-1}\dtx x-1)|_{H^{k+3/2}}+|a_{\zeta_t, r}\chi_r(\al^{-1}\dtx x-b_r)|_{H^{k+3/2}}+\big|(1-\chi_l-\chi_r)\al^{-\ga-3}(\dtx x)\wt{\dt\zeta}\big|_{H^{k+3/2}}\\
&\le Q(\cT', w, k+1/2)\big[\big|\al^{-\ga-2}\wt{\dt\zeta}- a_{\zeta_t}\big|_{H^{k+3/2}}+|a_{\zeta_t, i}|_{L^\infty}\big],
\end{split}
\]
where we recall \eqref{limit T'} and \eqref{limit T' right} for the limits of $\al^{-1}\dtx x$.

Consequently, applying an interpolation to $\|w_\Phi\|_{H^{k}}$, we derive the desired estimate for $\dtz\tPhi$. 

\medskip

{\bf Step 2.} Estimate of $\dtx\tPhi$ and $\tv$. In fact,  using the equation of $\tPhi$ and applying the estimate of $\dtz\tPhi$,  we can have
\[
\begin{split}
\big\|\al^{-\ga-3}\dtx\tPhi-a_{\Phi_x}\big\|_{H^{k+2}}
\le Q(\cT', \na\tPhi, w, k+1/2)\big[\big|\al^{-\ga-2}\wt{\dt\zeta}-a_{\zeta_t}\big|_{H^{k+3/2}}+|a_{\zeta_t, i}|_{L^\infty}\big]
\end{split}
\]
where 
\[
a_{\Phi_x}=\chi_l a_{\Phi_x}+\chi_r a_{\Phi_x, r}
\] 
with 
\[
a_{\Phi_x, l}=\lim_{\tx\rightarrow -\infty}\al^{-\ga-3}\dtx\tPhi,\quad a_{\Phi_x, r}=\lim_{\tx\rightarrow +\infty}\al^{-\ga-3}\dtx\tPhi
\]
 satisfying
\[
a_{\Phi_x, l}=-\f 1{\ga+3}\dtz a_{\Phi_z, l},\quad a_{\Phi_x, r}=-\f 1{(\ga+3)\beta_r}\dtz a_{\Phi_z, r}.
\]

In the end, for the estimate of the velocity $v$, recalling that $\tv=\wt{\na \Phi}$ and applying \eqref{na X matrix} lead to 
\[
\tv=(\na_X\wt X)\circ \cT \na \tPhi=|\cT'|^{-2}\left(\begin{matrix}
\dtx x \dtx \tPhi+\dtx z\dtz\tPhi\\
\dtx z\dtx\tPhi -\dtx x\dtz\tPhi
\end{matrix}\right),
\]
which implies 
\begin{equation}\label{a v expression}
\begin{split}
a_{v, l}&=\lim_{\tx\rightarrow -\infty}\al^{-\ga-2}\tv=\lim_{\tx\rightarrow -\infty}\f{1}{\al^{-2}(\dtx x)^2+\al^{-2}(\dtx z)^2}
\left(\begin{matrix}
\al^{-1}\dtx x \al^{-\ga-3}\dtx \tPhi+\al^{-1}\dtx z\al^{-\ga-3}\dtz\tPhi\\
\al^{-1}\dtx z\al^{-\ga-3}\dtx\tPhi -\al^{-1}\dtx x\al^{-\ga-3}\dtz\tPhi
\end{matrix}\right)\\
&=\f{1}{a_{x, l}^2+a_{z, l}^2}\left(\begin{matrix}a_{x, l} a_{\Phi_x, l}+a_{z, l}a_{\Phi_z, l}\\
a_{z, l} a_{\Phi_x, l}-a_{x, l}a_{\Phi_z, l}\end{matrix}\right).
\end{split}
\end{equation}
Moreover, one has  a similar expression for $a_{v, r}$.

As a result, we can have the $W^{k, \infty}$ estimate of $a_{v, i}$ thanks to \eqref{ax az} and \eqref{estimate a phi}. Applying Corollary \ref{estimate of x z}, we obtain the desired estimate for $\tv$. 

\end{proof}

Second, we obtain a higher  order estimate for $\tv$ near the bottom $\G_b$ thanks to Corollary \ref{higher order local}.
\begin{corollary}\label{v higher order} Under the same assumption of  Proposition \ref{v estimate}, one has  
\[
\begin{split}
\|\al^{-\ga-2}\tv-a_v\|_{H^{k+5/2}(\cS_b)}\le Q(\cT', \na\tPhi, w, k)\big[\big|\al^{-\ga-2}\wt{\dt\zeta}- a_{\zeta_t}\big|_{H^{k+1}}+|a_{\zeta_t, i}|_{L^\infty}\big],
\end{split}
\]
where $\cS_b\subset \cSt$ is the flat strip near $\G_b$.
\end{corollary}
\begin{proof} The proof is similar as the proof of Corollary \eqref{higher order local}. We first prove the estimate of $\dtz\tPhi$ and use again $w_\Phi=\al^{-\ga-3}\dtz\tPhi-a_{\Phi_z}$ from the proof of Proposition \ref{v estimate}. 

In fact, recalling the  cut-off function $\chi_b$ near $\G_b$ from Corollary \eqref{higher order local}, $\chi_b w_{\Phi}$ satisfies 
\[
\begin{cases}
\Delta (\chi_bw_\Phi)=[\Delta, \chi_b]w_\Phi-\chi_b\Delta w_\Phi\qquad \hbox{in}\quad \cSt,\\
\chi_bw_\Phi\big|_{\G_t}=0,\quad \chi_bw_\Phi\big|_{\G_b}=0.
\end{cases}
\]
Applying Proposition \ref{v estimate} and the expression of $\Delta w_\Phi$ in that proof, we have
\[\begin{split}
\|\chi_bw_\Phi\|_{H^{k+5/2}}&\le C\big(\|w_\Phi\|_{H^{k+3/2}}+\|\chi_b\Delta w_\Phi\|_{H^{k+1/2}}\big)\\
&\le C\big(\|w_\Phi\|_{H^{k+3/2}}+|a_{\zeta_t, i}|_{L^\infty}\big)\\
&\le Q(\cT', \dtz\tPhi, w, k)\big[|a_{\zeta_t, i}|_{L^\infty}+\big|\al^{-\ga-2}\wt{\dt\zeta}- a_{\zeta_t}\big|_{H^{k+1}}\big].
\end{split}\]

Moreover, similar analysis as before, we also derive a  similar estimate for $\dtx \tPhi$ near $\G_b$.  Consequently, thanks to the higher order estimate of $\cT'$ in Corollary \eqref{higher order local} and  the expression of $\tv$ in terms of $\cT', \na\tPhi$ in the proof of  Proposition \ref{v estimate}, we conclude immediately the desired estimate for $\tv$.

\end{proof}

Third, we present the  higher order estimate for $\tv$.
\begin{proposition}\label{v estimate higher order} Under the same assumption of  Proposition \ref{v estimate}, one has 
\[
\begin{split}
\|\al^{-\ga-2}\tv-a_v\|_{H^{k+5/2}}\le &Q(\cT', \na\tPhi, \ttze, w, k+1/2)\big[1+\big|\al^{-\ga-2}\wt D_t w-a_{w_t}\big|_{H^k}+|a_{\zeta_t, i}|_{L^\infty}\big].
\end{split}
\]
\end{proposition}
\begin{proof} The proof adopts the idea of  the proof for Proposition 4.3 \cite{SZ}, while  the key point is the estimate of the boundary condition $\na_{n_t} v|_{\G_t}$.  As mentioned in the introduction, the key difference compared to the smooth-boundary estimates lies in the estimate for  the higher-order terms in the bottom.

{\bf Step 1}. Decomposition of the boundary condition on $\G_t$ for $v$. 
In fact, remembering that the fluid is irrotational (see \eqref{incomp-irro}), we rewrite
\begin{equation}\label{uv def}
\na_{n_t}v=(\na v)^T n_t:= u_v\qquad\hbox{on}\quad \G_t.
\end{equation}
Decomposing $u_v$ into
\[
u_v=u^\top_v+u^\perp_v n_t:=(u_v\cdot \tau_t)\tau_t+(u_v\cdot n_t)n_t,
\]
we deal with the estimate of $u^\top_v$ and $u^\perp_v$ one by one.

\medskip

{\bf Step 2}. The tangential part $u^\top_v$. Direct computations show that
\[
(\Delta_{\G_t} u^\top_v)\cdot \tau_t=(\na_{\tau_t}\na_{\tau_t}u^\top_v)\cdot \tau_t=\na_{\tau_t}(\cD_{\tau_t}u^\top_v\cdot \tau_t)=\na_{\tau_t}(\cD\cdot u^\top_v),
\]
where $\cD$ is the covariant (tangential) derivative on $\G_t$. Therefore, we have
\[
-\Delta_{\G_t}u^\top_v=-\na_{\tau_t}(\cD\cdot u^\top_v)\tau_t\qquad\hbox{on}\quad \G_t.
\]

On the other hand, we also know that 
\[
\begin{split}
\cD\cdot u^\top_v&=\na_{\tau_t}u^\top_v\cdot \tau_t=\na_{\tau_t}\big[\big(n_t\cdot(\na v)\tau_t\big)\tau_t\big]\cdot \tau_t\\
&=\na_{\tau_t}\big(n_t\cdot (\na v)\tau_t\big)=\na_{\tau_t}(n_t\cdot \na_{\tau_t}v)
=\na_{\tau_t}\na_{\tau_t}v\cdot n_t+\na_{\tau_t}v\cdot \na_{\tau_t}n_t.
\end{split}
\]
Since 
\[
\Delta_{\G_t}=\na_{\tau_t}\na_{\tau_t}-\na_{\cD_{\tau_t}\tau_t}=\na_{\tau_t}\na_{\tau_t}
\]
in our 2D settings, substituting \eqref{Dtk eqn} into the above expression of $\cD\cdot u^\top_v$ leads to
\[
\cD\cdot u^\top_v=-D_t\kappa-\na_{\tau_t}n_t\cdot \na_{\tau_t} v\qquad\hbox{on}\quad \G_t.
\]
Consequently, we derive the following equation
\[
-\Delta_{\G_t}u^\top_v=\na_{\tau_t}(D_t\kappa)\tau_t+\na_{\tau_t}\big(\na_{\tau_t}n_t\cdot \na_{\tau_t} v\big)\tau_t 
\qquad\hbox{on}\quad \G_t,
\]
which is an elliptic equation on $\G_t$ and can be handled similar as before. 

Performing $\cT$ on both sides of the above equation, we arrive at 
\[
-\big((\dtx x)^{-1}b^{-\f12}_\zeta\dtx\big)^2\wt{u^\top_v}=(\dtx x)^{-1}b^{-\f12}_\zeta\dtx(\wt D_tw)\wt \tau_t+
(\dtx x)^{-1}b^{-\f12}_\zeta\dtx\big(\wt{\na_{\tau_t}n_t}\cdot \wt{\na_{\tau_t} v}\big)\wt \tau_t
\]
on the upper boundary of $\cSt$, where $b_\zeta$ is defined in \eqref{b zeta}.

As a result, similar analysis as in the proof of Proposition \ref{dt Ka estimate}, we have immediately the following elliptic estimate
\begin{equation}\label{uv tangential estimate}
\big|\al^{-\ga-1}\wt{u^\top_v}-a_{u_v,\tau}\big|_{H^{k+1}}\le Q(\cT', w, k)\big(\big|\al^{-\ga}\wt D_tw-a_{w_t}\big|_{H^k}+\big|\al^{-\ga-2}\tv-a_v\big|_{H^{k+1}}+|a_{\zeta_t,i}|_{L^\infty}\big),
\end{equation}
where the weighted limit $a_{u_v,\tau}$ is defined for $\wt{u^\top_v}$ in a similar way as before and  can be handled (depending on $a_{\zeta_t, i}$). Here we omit the analysis for weighted limits since they are decided by the corresponding equations and can always be handled algebraically as before.

\medskip

{\bf Step 3}. The normal part $u^\perp_v$.  Since 
\[
u^\perp_v=u_v\cdot n_t=n_t\cdot \na_{n_t}v,
\]
we extend $n_t$ as $\bar n_t$ in $\Om$ (to be fixed later) and denote by $\bar u^\perp_v$ the extension of $u^\perp_v$ with $\bar n_t$.
 
Direct computations show that
\[
\begin{split}
\na\cdot u_v&=\na_{\tau_t}u_v\cdot \tau_t+\na_{n_t}u_v\cdot n_t\\
&=(\na_{\tau_t}\na_{\tau_t}v+\na_{n_t}\na_{\bar n_t} v)\cdot n_t+\na_{\tau_t}v\cdot \na_{\tau_t}n_t+\na_{ n_t}v\cdot \na_{n_t}\bar n_t-n_t\cdot \na v(\na_{\tau_t}\tau_t+\na_{n_t}\bar n_t)\\
&=(\Delta v)\cdot n_t+r(\p v, \p n_t, \p \tau_t)=r(\p v, \p \bar n_t, \p \tau_t),
\end{split}
\]
where the equation $\Delta v=0$ is applied and 
\[
r(\p v, \p \bar n_t, \p \tau_t)=\na_{\tau_t}v\cdot \na_{\tau_t}n_t+\na_{ n_t}v\cdot \na_{n_t}\bar n_t-2n_t\cdot \na v(\na_{\tau_t}\tau_t+\na_{n_t}\bar n_t)
\]
is the lower-order part.

Moreover, we also have on $\G_t$ that
\[
\begin{split}
\na\cdot u_v&=\na_{\tau_t}u_v\cdot \tau_t+\na_{n_t}u_v\cdot n_t\\
&=\na_{\tau_t}u^\top_v\cdot \tau_t+\na_{\tau_t}(u^\perp n_t)\cdot \tau_t+\na_{n_t}u_v\cdot n_t\\
&=\cD\cdot u^\top_v+\kappa u^\perp_v+\na_{n_t}u_v\cdot n_t.
\end{split}
\]
Therefore, summing up these expressions above leads to
\begin{equation}\label{N uv eqn}
\na_{n_t}u_v\cdot n_t=-\cD\cdot u^\top_v-\kappa u^\perp_v+r(\p v, \p n_t, \p \tau_t).
\end{equation}

Next, we plan to derive $\cN u^\perp_v$ from the left side of \eqref{N uv eqn} as in \cite{SZ}, which will be adapted to our weighted version in $\cSt$. 

 In fact, we start with rewriting the left side as 
\[
\na_{n_t}u_v\cdot n_t=\na_{n_t}(\bar n_t\cdot \na_{\bar n_t}v)-(\na v)^Tn_t\cdot \na_{n_t}\bar n_t.
\]
Substituting it into \eqref{N uv eqn}, we obtain
\[
\na_{n_t}\bar u^\perp_v=\na_{n_t}(\bar n_t\cdot \na_{\bar n_t}v)=-\cD\cdot u^\top_v+r_1(\p v, \p \bar n_t, \p \tau_t),
\]
where 
\[
r_1(\p v, \p \bar n_t, \p \tau_t)=(\na v)^Tn_t\cdot \na_{n_t}\bar n_t-\kappa u^\perp_v+r(\p v, \p \bar n_t, \p \tau_t).
\]

Performing $\cT$ on both sides of the above equation and noticing that 
\[
\wt{\na_{n_t}F}=b^{-\f12}_\zeta(\dtx x)^{-1}\dtz \wt F\qquad\hbox{on}\quad \G_t,
\]
we have
\begin{equation}\label{uv n eqn middle}
b^{-\f12}_\zeta(\dtx x)^{-1}\dtz\wt{\bar u^\perp_v}=-\wt{\cD\cdot u^\top_v}+r_1(\wt{\p v}, \wt{\p \bar n_t}, \wt{\p \tau_t})\qquad\hbox{on}\quad \G_t.
\end{equation}

Denoting by $a_{\bar u_v,n}$ the weighted limit of $\wt{\bar u^\perp_v}$ and $a_{u_v, n}$ the weighted limit of $\wt{u^\perp_v}$ on the boundary, the left side above is rewritten into 
\[
\begin{split}
b^{-\f12}_\zeta(\dtx x)^{-1}\dtz\wt{\bar u^\perp_v}
&=\al^{\ga+1}b^{-\f12}_\zeta(\dtx x)^{-1}\dtz\big(\al^{-\ga-1}\wt{\bar u^\perp_v}-a_{\bar u_v,n}\big)+\al^{\ga+1}b^{-\f12}_\zeta(\dtx x)^{-1}\dtz a_{\bar u_v,n}\\
&=\al^{\ga+1}b^{-\f12}_\zeta(\dtx x)^{-1}\big[\tN_0(\al^{-\ga-1}\wt{\bar u^\perp_v}-a_{\bar u_v,n})+\dtz w_v\big]
+\al^{\ga+1}b^{-\f12}_\zeta(\dtx x)^{-1}\dtz a_{\bar u_v,n}.
\end{split}
\]
Here 
\[
w_v =(\al^{-\ga-1}\wt{\bar u^\perp_v}-a_{\bar u_v,n})-\cH(\al^{-\ga-1}\wt{ u^\perp_v}-a_{u_v,n})
\]
satisfies
\[
\begin{cases}
\Delta w_v=\Delta(\al^{-\ga-1}\wt{\bar u^\perp_v})-\Delta a_{\bar u_v,n}=\Delta\big(\al^{-\ga-1}\wt {\bar n_t}\cdot \wt{\na_{\bar n_t}v}\big)-\Delta a_{\bar u_v,n}\qquad\hbox{in}\quad \cSt,\\
w_v|_{\tz=0}=0,\\
\dtz w_v|_{\tz=-\om}=\dtz\big(\al^{-\ga-1}\wt {\bar n_t}\cdot \wt{\na_{\bar n_t}v}\big)|_{\tz=-\om}-\dtz a_{\bar u_v,n}|_{\tz=-\om}:=g(\p^2\tv, \p \wt{\bar n_t}).
\end{cases}
\]

Substituting this expression into \eqref{uv n eqn middle}, we conclude that
\[
\tN_0(\al^{-\ga-1}\wt{ u^\perp_v}-a_{ u_v,n})=-\dtz w_v-\dtz a_{\bar u_v,n} -\al^{-\ga-1}b^{\f12}_\zeta(\dtx x)\big[\wt{\cD\cdot u^\top_v}-r_1(\wt{\p v}, \wt{\p \bar n_t}, \wt{\p \tau_t})\big].
\]
As a result, we perform the following estimate
\[
\big|\tN_0(\al^{-\ga-1}\wt{ u^\perp_v}-a_{ u_v,n})\big|_{H^k}\le |\dtz w_v|_{H^k}+\big|\dtz a_{\bar u_v,n} +\al^{-\ga-1}b^{\f12}_\zeta(\dtx x)\big[\wt{\cD\cdot u^\top_v}-r_1(\wt{\p v}, \wt{\p \bar n_t}, \wt{\p \tau_t})\big]\big|_{H^k}.
\]
Since the corresponding weighted limits for $\wt{u^\top_v}$ and $\tv$ can be retrieved from $\dtz a_{\bar u_v,n}$ on the right side above and  the analysis is omitted again, we derive by Lemma \ref{trace}, Proposition \ref{CM estimate} and \eqref{uv tangential estimate} that
\begin{equation}\label{N0 uv estimate mid}
\begin{split}
&\big|\tN_0(\al^{-\ga-1}\wt{ u^\perp_v}-a_{u_v,n})\big|_{H^k}\\
&\le C\| w_v\|_{H^{k+3/2}}+Q(\cT', w, k)\big(|\al^{-\ga-1}\wt{u^\top_v}-a_{u_v,\tau}\big|_{H^{k+1}}+\|\al^{-\ga-2}\tv-a_v\|_{H^{k+3/2}}\big)\\
&\le C\| w_v\|_{H^{k+3/2}}+Q(\cT', w, k)\big(|\al^{-\ga}\wt D_t w-a_{w_t}|+\|\al^{-\ga-2}\tv-a_v\|_{H^{k+3/2}}+|a_{\zeta_t,i}|_{L^\infty}\big).
\end{split}
\end{equation}

Moreover, similar analysis as in the proof of Proposition \ref{v estimate},  one has  the following elliptic estimate
\[
\begin{split}
\| w_v\|_{H^{k+3/2}}&\le C\big(\|\Delta w_v\|_{H^{k-1/2}}+|\dtz w_v|_{H^k(\G_b)}\big)\\
&\le Q(\cT', w, k+1/2)\big(\|\al^{-\ga-2}\tv-a_v\|_{H^{k+3/2}}+|a_{\zeta_t,i}|_{L^\infty}
+\|\al^{-\ga-2}\tv-a_v\|_{H^{k+5/2}(\cS_b)}\big).
\end{split}
\]
Applying the high-order derivative estimate of $\tv$ in Corollary \ref{v higher order} and lower-order derivative estimate in Proposition \ref{v estimate}, we derive
\[
\| w_v\|_{H^{k+3/2}}\le Q(\cT', w, k+1/2)\big(|a_{\zeta_t,i}|_{L^\infty}+|\al^{-\ga-2}\ttze-a_{\zeta_t}|_{H^{k+1}}\big).
\]
Substituting this estimate back into \eqref{N0 uv estimate mid}, we conclude that 
\[
\big|\tN_0(\al^{-\ga-1}\wt{\bar u^\perp_v}-a_{\bar u_v,n})\big|_{H^k}\le Q(\cT', w, k+1/2)\big(|\al^{-\ga}\wt D_t w-a_{w_t}|+|a_{\zeta_t,i}|_{L^\infty}+|\al^{-\ga-2}\ttze-a_{\zeta_t}|_{H^{k+1}}\big),
\]
which implies immediately the desired  estimate for the normal component $u^\perp_v$:
\begin{equation}\label{uv normal estimate}
\begin{split}
\big|\al^{-\ga-1}\wt{u^\perp_v}-a_{u_v,n}\big|_{H^{k+1}}
&\le C\big(\big|\tN_0(\al^{-\ga-1}\wt{ u^\perp_v}-a_{ u_v,n})\big|_{H^k}+\big|\al^{-\ga-1}\wt{u^\perp_v}-a_{u_v,n}\big|_{L^2}\big)\\
&\le Q(\cT', w, k+1/2)\big(|\al^{-\ga}\wt D_t w-a_{w_t}|_{H^k}+|a_{\zeta_t,i}|_{L^\infty}+|\al^{-\ga-2}\ttze-a_{\zeta_t}|_{H^{k+1}}\big).
\end{split}
\end{equation}

\medskip

{\bf Step 4}. Elliptic system for $\tv$ and the estimate.  In the end, we are ready to show the desired higher-order derivative estimate fo $\tv$. In fact, the system for $\al^{-\ga-2}\tv-a_v$ reads
\[
\begin{cases}
\Delta (\al^{-\ga-2}\tv-a_v)=\Delta (\al^{-\ga-2}\tv)-\Delta a_v\qquad\hbox{in}\quad \cSt,\\
\dtz(\al^{-\ga-2}\tv-a_v)|_{\tz=0}=\al^{-\ga-2}b^{\f12}_\zeta (\dtx x) \wt{u_v}-\dtz a_v|_{\tz=0},\\
\al^{-\ga-2}\tv-a_v|_{\tz=-\om} \quad\hbox{is done in Corollary \ref{v higher order}.}
\end{cases}
\]
Meanwhile on the upper surface, we have the following estimate by summing up \eqref{uv tangential estimate} and \eqref{uv normal estimate}:
\[
\big|\al^{-\ga-1}\wt{u_v}-a_{u_v}\big|_{H^{k+1}}\le Q(\cT', w, k+1/2)\big(|\al^{-\ga}\wt D_t w-a_{w_t}|_{H^k}+|a_{\zeta_t,i}|_{L^\infty}+|\al^{-\ga-2}\ttze-a_{\zeta_t}|_{H^{k+1}}\big),
\]
where $a_{u_v}$ is the corresponding weighted limit of $\wt{u_v}$.

As a result, similarly as in the proof Proposition \ref{v estimate} and thanks to Corollary \ref{v higher order},  we conclude estimate for $\tv$:
\[
\begin{split}
&\|\al^{-\ga-2}\tv-a_v\|_{H^{k+5/2}}\\
&\le C\big(\|\Delta (\al^{-\ga-2}\tv)-\Delta a_v\|_{H^{k+1/2}}+\big|\al^{-\ga-2}b^{\f12}_\zeta (\dtx x) \wt{u_v}-\dtz a_v|_{\tz=0}\big|_{H^{k+1}}+|\al^{-\ga-2}\tv-a_v|_{H^{k+2}(\G_b)}\big)\\
& \le Q(\cT', w, k)\big(\|\Delta (\al^{-\ga-2}\tv)-\Delta a_v\|_{H^{k+1/2}}+\big|\al^{-\ga-1}\wt{u_v}-a_{u_v}\big|_{H^{k+1}}+\|\al^{-\ga-2}\tv-a_v\|_{H^{k+5/2}(\cS_b)}\big)\\
&\le Q(\cT', w, k+1/2)\big(|\al^{-\ga}\wt D_t w-a_{w_t}|_{H^k}+|a_{\zeta_t,i}|_{L^\infty}+|\al^{-\ga-2}\ttze-a_{\zeta_t}|_{H^{k+1}}\big).
\end{split}
\]
Moreover, applying Proposition \ref{dt Ka estimate} to the right-side term $|\al^{-\ga-2}\ttze-a_{\zeta_t}|_{H^{k+1}}$ and using an interpolation on $\tv$,   we obtain 
\[
\begin{split}
\|\al^{-\ga-2}\tv-a_v\|_{H^{k+5/2}}\le& Q(\cT', w, k+1/2)\big(|\al^{-\ga}\wt D_t w-a_{w_t}|_{H^k}+|a_{\zeta_t,i}|_{L^\infty}+\|\al^{-\ga-2}\tv-a_v\|_{H^{k+3/2}}\\
&+|\al^{-\ga-2}\ttze-a_{\zeta_t}|_{L^2}\big).
\end{split}\]
Therefore, the proof is finished.

\end{proof}

\bigskip

In the end,  we deal with $\tP$ and $\ma$. {\it Recalling Assumption \ref{assump on a}  and system \eqref{tPvv eqn} for $\tP$ (especially the boundary conditions), we know immediately that $\al^{-2\ga-4}\tP$ should not vanish  when $\tx\rightarrow \pm\infty$.} Therefore, setting
\begin{equation}\label{a p def}
a_p=\chi_l a_{p,l}+\chi_ra_{p, r}
\end{equation}
with 
\[
a_{p, l}(t, \tz)=\lim_{\tx\rightarrow -\infty}\al^{-2\ga-4}\tP,\quad a_{p, r}(t, \tz)=\lim_{\tx\rightarrow +\infty}\al^{-2\ga-4}\tP,
\]
we obtain from \eqref{tPvv eqn} the following system for $a_{p, l}$:
\begin{equation}\label{ap eqn}
\begin{cases}
\dtz^2 a_{p, l}=-(2\ga+4)^2a_{p, l}- a_{Dv^2},\\
a_{p, l}|_{\tz=0}=0,\quad \dtz a_{p, l}|_{\tz=-\om}=0
\end{cases}
\end{equation}
 where  
 \[
 a_{Dv^2}=\lim_{\tx\rightarrow -\infty}\al^{-2\ga-4}|\cT'|^2tr\big((\na_X \tX)\circ \cT\na \tv\big)^2
 \]
is  the combination of weighted limits for $\al^{-2\ga-4}|\cT'|^2tr\big((\na_X \tX)\circ \cT\na \tv\big)^2$ defined in a similar way as before.  
The computations in $ a_{Dv^2}$ are similar as in \eqref{a v expression} and is hence omitted here. Consequently, this limit depends on $T_{1,c}$, $a_{\Phi_z}$. Moreover, we know
\[
\dtz a_{p, l}|_{\tz=0} <0
\]
thanks to Assumption \ref{assump on a}. The analysis for $a_{p, r}$ follows   similarly.

\begin{proposition}\label{tPvv estimate}
Let real $\ga>-1$, integer $k\ge 0$ and $\tP$ solve \eqref{tPvv eqn}. Then there holds:
\[
\begin{split}
\|\al^{-2\ga-4}\tP-a_p\|_{H^{k+7/2}}
\le  Q(\cT', \na\tPhi, \ttze, \tP, w, k+1/2)\big(1+|\al^{-\ga}\wt D_t w-a_{w_t}|^2_{H^k}+|a_{\zeta_t, i}|^2_{L^\infty}\big)
\end{split}
\]
where  
\[
|a_{p, i}|_{W^{k+2, \infty}_\tz}\le Q_b |a_{\zeta_t, i}(t)|^2.
\]
\end{proposition}
\begin{proof} The proof is similar as before. To begin with, we define
\[
w_p=\al^{-2\ga-4}\tP-a_p,
\]
which satisfies 
\[
\begin{cases}
\Delta w_p=[\Delta, \al^{-2\ga-4}]\tP+\al^{-2\ga-4}\Delta \tP-\Delta a_p \qquad\hbox{in}\quad \cSt,\\
w_p|_{\G_t}=0,\quad \dtz w_p|_{\G_b}=\al^{-2\ga-4}\dtz\tP|_{\G_b}.
\end{cases}
\]
Here the right side above can be computed in a similar way as in the system of $w_\Phi$ from the proof of Proposition \ref{v estimate}, where we apply \eqref{tPvv eqn}, \eqref{ap eqn} and only focus on ``the left part'':
\[
\begin{split}
\Delta w_p=&\big((\al^{-2\ga-4})''\tP-(2\ga+4)^2\chi_l a_{p, l}\big)-2\big(-(\al^{-2\ga-4})'\dtx \tP-(2\ga+4)^2\chi_{l,1}a_{p, l}\big)\\
&-\big[\al^{-2\ga-4}|\cT'|^2tr\big((\na_X \tX)\circ \cT\na \tv\big)^2+\chi_l\big(\dtz^2 a_{p, l}+(2\ga+4)^2a_{p, l}\big)\big]\\
&+2(2\ga+4)^2(\chi_l-\chi_{l,1})a_{p,l}-\chi_l''a_{p, l}
\end{split}
\]

Applying standard elliptic estimate, Lemma \ref{minus norm dx} and \eqref{ap eqn} leads to
\[
\begin{split}
&\|w_p\|_{H^{k+7/2}}=\|\al^{-2\ga-4}\tP- a_p\|_{H^{k+7/2}}\\
&\le C\big(\big\|\al^{-2\ga-4}\tP- a_p\|_{H^{k+5/2}}+|a_{p,i}|_{W^{k+2,\infty}_\tz}+\big\|\al^{-2\ga-4}|\cT'|^2tr\big((\na_X \tX)\circ \cT\na \tv\big)^2-a_{Dv^2}\big\|_{H^{k+3/2}}\\
&\quad +\big|\al^{-2\ga-4}|\cT'|\big((\na_X \tX)^t\circ \cT\big)\tv\cdot \na\tn_b\cdot \tv\big|_{H^{k+2}(\G_b)}
\big),
\end{split}
\]
where we recall that $ a_{Dv^2}$ is defined in \eqref{ap eqn} and notice that the highest-order derivatives for both $\cT$ and $\tv$ appear in the bottom part.

Applying \eqref{na X matrix},  Proposition \ref{CM estimate}, Corollary \ref{estimate of x z} and Corollary \ref{higher order local}, one finds
\[
\begin{split}
\|\al^{-2\ga-4}\tP-a_p\|_{H^{k+7/2}}\le& C\big(\big\|\al^{-2\ga-4}\tP-a_p\big\|_{H^{k+5/2}}+|a_{p,i}|_{W^{k+2,\infty}_z}\big)+Q(\cT',  w, k)\big\|\al^{-\ga-2}\tv-a_v\big\|^2_{H^{k+5/2}}
\end{split}
\]
As a result, the desired estimate for $\tP$  follows by applying  Proposition \ref{v estimate higher order}.  

In the end, applying standard one dimensional elliptic estimate to  system  \eqref{ap eqn} , we have the  estimate of $a_{p, l}$ thanks to the estimate \eqref{estimate a phi} for $a_{\Phi, z}$. Therefore, the desired estimate for $a_{p, i}$ follows.

\end{proof}

Based on the weighted estimate for $\tP$, we are able to prove the estimate for $\ta$, $\dt \ta$.
\begin{proposition} \label{a estimate}
Let real $\ga>-1$ and integer $k\ge 0$.One has
\[
\begin{split}
\big|\al^{-2\ga-3}\ta-\dtz a_p\big|_{H^{k+3/2}}\le Q(\cT', \na\tPhi, \ttze, \tP,  w, k+1/2)
\big(1+|\al^{-\ga}\wt D_t w-a_{w_t}|^2_{H^{k-1/2}}+|a_{\zeta_t, i}|^2_{L^\infty}\big).
\end{split}
\]

Moreover, the following estimate for $\dt\ta$ holds:
\[
\big|\al^{-2\ga-3}\dt\ta\big|_{L^\infty}\le Q(\cT', \na\tPhi, \ttze, \tP,  w, 3)Q\big(|\al^{-\ga}\wt D_t w-a_{w_t}|_{L^2}, |a_{\zeta_t, i}|_{L^\infty}\big).
\]
\end{proposition}
\begin{proof}
As the first step,   direct computations show that
\[
\ta=-\big(\na_{n_t}P|_{\G_t}\big)\circ\cT=-\wt n_t\cdot (\na_X\wt X)\circ \cT\na \tP\big|_{\tz=0}=-|\cT'|^{-1}\dtz\tP|_{\tz=0},
\]
so we know immediately from Lemma \ref{trace}, Proposition \ref{CM estimate},  Corollary \ref{estimate of x z} and \eqref{na X matrix Gt} that
\[
\begin{split}
\big|\al^{-2\ga-3}\ta-\dtz a_p\big|_{H^{k+3/2}}
&\le Q(\cT', w, k+1/2)\big(\big|\al^{-2\ga-4}\dtz\tP-\dtz a_p\big|_{H^{k+3/2}}+|a_{p,i}|_{W^{k+3, \infty}_\tz}\big)\\
&\le Q(\cT', w, k+1/2)\big(\big\|\al^{-2\ga-4}\tP-a_p\big\|_{H^{k+3}}+|a_{\zeta_t, i}|_{L^\infty}\big).
\end{split}
\]
Therefore, applying  Proposition \ref{tPvv estimate} with the order $k+3$ leads to the desired estimate for $\ta$.\medskip

Second, we  have by \eqref{Dt transform} that
\[
\dt\ta=-(D_t\na_{n_t}P)\circ \cT\big|_{\G_t}-u_1\dtx \ta,\quad
 u_1=(\dtx x)^{-1}(\wt {\underline v}_1-\dt x)
\]
where direct computations lead to
\[
\begin{split}
D_t\na_{n_t}P&=
[D_t,\,\na_{n_t}]P+\na_{n_t}D_t P\\
&=\big(D_tn_t\cdot \na P-n_t\cdot \na v\cdot \na P\big)+\na_{n_t}\Delta^{-1}(h_1, g_1)+\na_{n_t}\Delta^{-1}(h_2, g_2)
\end{split}
\]
where
\[
\begin{split}
h_1&=-tr\big[\big(-\na^2 P-(\na v)^T\na v\big)\na v+\na v\big(-\na^2 P-(\na v)^T\na v\big)\big],\\
g_1&= -\na_vn_b\cdot(\na P+{\bf g})+\na_{(\na P+{\bf g})}n_b\cdot v-v\cdot \na v\cdot \na n_b\cdot v-\na_v\big((\na v)^Tn_b\big)^\top\cdot v,\\
h_2&=2\na v\cdot\na^2 P+\Delta v\cdot \na P,\\
g_2&=(\na_{n_b}v-\na_v n_b)\cdot \na P.
\end{split}
\]

As a result, adding corresponding limit terms and proceeding as before,  we derive by Lemma \ref{embedding}, Proposition \ref{CM estimate},  Corollary \ref{estimate of x z}, Proposition \ref{dtT estimate}  that
\[
\big|\al^{-2\ga-3}\dt\ta\big|_{L^\infty}\le Q(\cT', w, 1)Q\big(\big\|\al^{-\ga-2}\tv-a_v\big\|_{H^{5/2}}, \big\|\al^{-2\ga-4}\tP-a_p\big\|_{H^{5/2}}\big).
\]
 As a result, the desired estimate follows from  Proposition \ref{v estimate higher order}  and  Proposition \ref{tPvv estimate}.

\end{proof}

\subsection{Modification of the equation} Based on all the discussions and estimates above, we find that $\al^{-\ga}\wt D_t w$ doesn't vanish when $\tx\rightarrow \pm\infty$  in general. As a direct consequence, we need to modify euqation \eqref{w eqn} of $w$ immediately adding the ``limit part" of $\wt D_t w$.

In fact, recalling $a_{w_t}$ from Proposition \ref{dt Ka estimate}, we rewrite \eqref{w eqn} into 
\begin{equation}\label{new w eqn}
\wt D_t\big(\wt D_t w-\al^\ga a_{w_t}\big)+\wt\ma |\cT'|^{-1}\widetilde{\cN}_0w=\wt R,
\end{equation}
where 
\[
\wt R=\widetilde{R}_1+\wt R_2-\wt D_t\big(\al^\ga a_{w_t}\big),
\]
and direct computations show that
\[
\wt D_t\big(\al^\ga a_{w_t}\big)=-(\ga+1)(\ga+2)\al^\ga\big(\chi_la'_{\zeta_t,l}+\chi_r\beta_r^2(b^2_ra_{\zeta_t, r})'\big)+(\ga+1)(\ga+2)u_1\big((\al^\ga\chi_l)'a_{\zeta_t, l}+(\al^\ga\chi_r)'\beta_r^2b_r^2a_{\zeta, r}\big)
\]
with $u_1$ defined in \eqref{Dt transform}.

Naturally, we also need to investigate $a'_{\zeta_t, i}$ ($i=l, r$). Starting from \eqref{dt zeta eqn} and multiplying the weight $\al^{-\ga-2}$ on both sides, we obtain
\[
\begin{split}
\p_t\big(\al^{-\ga-2}\ttze\big)=\al^{-\ga-2}\dt \ttze =&\al^{-\ga-2}\p_t(b_{\zeta}^{\f12})\tv\cdot \wt n_t+\al^{-\ga-2}b_{\zeta}^{\f12}\p_t\tv\cdot \wt n_t+\al^{-\ga-2}b_{\zeta}^{\f12}\tv\cdot \p_t\wt n_t\\
=&-\al^{-\ga-2}b_{\zeta}^{-\f12}(\dtx x)^{-1}\dtx\tzeta\big[-(\dtx x)^{-2}\dtx\p_t x\dtx\tzeta+(\dtx x)^{-1}\dtx\p_t\tzeta\,\big]\tv\cdot \wt n_t\\
&+\al^{-\ga-2}b_{\zeta}^{\f12}\big(\wt D_t \tv-u_1\dtx \tv\big)\cdot \wt n_t+\al^{-\ga-2}b_{\zeta}^{\f12}\tv\cdot \big(\wt D_t\wt n_t-u_1\dtx \wt n_t\big).
\end{split}
\]
Substituting \eqref{dt tzeta expression}, the Euler equation \eqref{Euler} and \eqref{Dtn expression} into the equation above, one has
\begin{equation}\label{zeta tt eqn}
\begin{split}
\al^{-\ga-2}\dt \ttze =&-\al^{-\ga-2}b_{\zeta}^{-\f12}(\dtx x)^{-1}\dtx\tzeta\big[(\dtx x)^{-1}\dtx\wt{\dt\zeta}-(\dtx x)^{-1}\p_t x(\dtx x)^{-2}\dtx^2 x\dtx\tzeta\\
&+(\dtx x)^{-1}\p_t x(\dtx x)^{-1}\dtx^2\tzeta\,\big]\tv\cdot \wt n_t+\al^{-\ga-2}b_{\zeta}^{\f12}\big(-\wt{\na P}-u_1\dtx \tv\big)\cdot \wt n_t\\
&+\al^{-\ga-2}b_{\zeta}^{\f12}\tv\cdot \big(-(\na_{\tau_t}v\cdot \wt n_t)\tau_t-u_1\dtx \wt n_t\big).
\end{split}
\end{equation}
Taking the limit when $\tx\rightarrow \pm\infty$ and checking the right side term by term, we conclude
\[
a'_{\zeta_t, i}=0,\quad i=l, r.
\]
Notice that when there is gravity, one has
\[
a'_{\zeta_t, i}=\infty.
\]

Therefore, we conclude that
\begin{equation}\label{Dt al a w}
\begin{split}
\wt D_t\big(\al^\ga a_{w_t}\big)=&-(\ga+1)(\ga+2)\beta_r^2\al^\ga\chi_r 2b_rb_r' a_{\zeta_t, l}\\
&+(\ga+1)(\ga+2)(\dtx x)^{-1}(\wt {\underline v}_1-\dt x)\big((\al^\ga\chi_l)'a_{\zeta_t, l}+(\al^\ga\chi_r)'\beta_r^2b_r^2a_{\zeta_t, r}\big).
\end{split}
\end{equation}

Besides, recalling \eqref{estimate a phi}, \eqref{a v expression}, \eqref{ap eqn} and Remark \ref{dt T corner}, we know immediately that 
\begin{equation}\label{time rate}
\dt a_{\Phi_z}= \dt a_v=\dt a_p=0.
\end{equation}
These happen due to our settings of the time derivatives which own higher-order weights since $\ga+1>0$.

\section{A priori estimates} \label{energy}


we define the order-k energy ($k\ge 0$) on the surface $\G_t$ of the strip domain $\cSt$:
\begin{equation}\label{energy on S}
\begin{split}
E_{k,\ga}(t)=&\big|\alpha^{-\ga}\dtx^k(\wt D_t w-\al^\ga a_{w_t})\big|^2_{L^2}
+\big|\ta^{\frac12}\alpha^{-\ga}|\cT'|^{-\frac12}\dtx^k w\big|^2_{L^2}\\
&+\big(\tN_0(\ta^{\frac12}\alpha^{-\ga}|\cT'|^{-\frac12}\dtx^k w), \ta^{\frac12}\alpha^{-\ga}|\cT'|^{-\frac12}\dtx^k w\big)_{L^2},
\end{split}
\end{equation}
 and the lower-order energy
\[
\begin{split}
E_0(t)=&\big\|\al^{-1}\cT'-T_{1,c}\big\|^2_{L^2}+\big|\al^{-\ga-2}\ttze-a_{\zeta_t}\big|^2_{L^2} +\big\|\al^{-\ga-3}\na\Phi-(a_{\Phi_x}, a_{\Phi_z})^T\big\|^2_{L^2}+\big\|\al^{-2\ga-4}\tP-a_p\big\|^2_{L^2}\\
&+|a_{\zeta_t, i}|^2_{L^\infty}+|b|^2_{L^\infty}+|b_r|^2_{L^\infty}+|b_{r2}|^2_{L^\infty}
\end{split}
\]
where $T_{1,c}$, $a_{\zeta_t}$, $a_{\Phi_x}$, $a_{\Phi_z}$, and $a_p$ are the corresponding weighted limits given in \eqref{T1c def}, \eqref{a zeta t def}, Proposition \ref{v estimate} and \eqref{a p def}  respectively.

The total energy is then defined by 
\[
\cE(t)=\sum_{l\le k}E_{l,\ga}(t)+E_0(t).
\]

The main theorem of our paper is presented below.
\begin{theorem}\label{energy estimate} Assume that  the integer $k$ and real $\ga$ satisfy
\[
2\le k\le \min\{\f\pi\om, \f\pi{\om_r}\}, \quad 0<\ga+1\le \min\{\f{2\pi}{\om}, \f{2\pi}{\om_r}\}
\]
and $Q_E(\cdot, \cdot, \cdot), P(\cdot)$ are  fixed polynomials with positive constant coefficients depending only on $k, \ga, \om, \om_r$ and the bound $L$  from Definition \ref{initial value bound}.

Let  a solution to the water-waves system \eqref{Euler}-\eqref{P condi} be given by the free surface graph function $\zeta$ and its time derivative $\dt\zeta$ with $\al^{-1}\dtx\tzeta\in H^{k+3/2}(\RR)$ and $\al^{-\ga-2}\ttze-a_{\zeta_t}\in H^{k+2}(\RR)$ and the initial values are chosen from $\Lam_0$. Moreover, assume that  Assumption \ref{assump on a} holds initially for $t=0$. 

Then there exists $T^*>0$ depending only on $Q_E(0)=Q_E\big(E_0(0), \big|\al w(0)|^2_{H^{k+1/2}}, \big|\al^{-\ga}\wt D_t w(0)-a_{w_t}(0)\big|^2_{H^k}\big)$ 
such that when $T\le T^*$,
we have $(\zeta,  \cT, \Phi, P)\in \Lam(T,k)$  and the following  estimate holds 
\[
\cE(t)\le \cE(0)+\int^t_0P\big(\cE(s)\big) ds,\qquad \forall t\in [0, T].
\]
\end{theorem}
\begin{remark}\label{rmk on Taylor}  Based on the proof of this theorem, Assumption \ref{assump on a} holds for $t\in [0, T]$, i.e. the weighted Taylor's sign holds with a new lower bound $a_0/2$ in stead of $a_0$. 
\end{remark}

\begin{proof}
To begin with, direct computations lead to
\begin{equation}\label{dt E}
\begin{split}
\f12\f d{dt}E_{k,\ga}(t)=& 
\big(\al^{-\ga}\dtx^k\dt  (\wt D_t w-\al^\ga a_{w_t}), \al^{-\ga}\dtx^k(\wt D_t w-\al^\ga a_{w_t})\big)\\
&+\big(\dt(\ta^{1/2}\al^{-\ga}|\cT'|^{-1/2})\dtx^k w, \ta^{1/2}\al^{-\ga}|\cT'|^{-1/2}\dtx^k w\big)\\
&+ \big(\ta^{1/2}\al^{-\ga}|\cT'|^{-1/2}\dtx^k\dt w, \ta^{1/2}\al^{-\ga}|\cT'|^{-1/2}\dtx^k w\big)\\
&+\big(\tN_0(\ta^{1/2}\al^{-\ga}|\cT'|^{-1/2}\dtx^k w), \dt(\ta^{1/2}\al^{-\ga}|\cT'|^{-1/2})\dtx^k w\big)\\
&+\big(\tN_0(\ta^{1/2}\al^{-\ga}|\cT'|^{-1/2}\dtx^k w), \ta^{1/2}\al^{-\ga}|\cT'|^{-1/2}\dtx^k\dt w\big)\\
:=& I_1+I_2+\cdots+I_5,
\end{split}
\end{equation}
where we take $(\cdot,\cdot)$ for the inner product $(\cdot,\cdot)_{L^2(\RR)}$. 
\bigskip

 {\bf Step 1. Estimates for $I_1+I_5$.}  In fact, one could see that $I_1+I_5$ is the main part in the energy estimate. Plugging equation \eqref{new w eqn} into $I_1$, one has
\[
\begin{split}
I_1+I_5=&-\big(\al^{-\ga}\dtx^k [u_1\dtx (\wt D_t w-\al^\ga a_{w_t})],\,\al^{-\ga}\dtx^k(\wt D_t w-\al^\ga a_{w_t})\big)\\
&-\big(\al^{-\ga}\dtx^k[\ta|\cT'|^{-1}\tN_0 w], \al^{-\ga}\dtx^k(\wt D_t w-\al^\ga a_{w_t})\big)\\
& +\big(\al^{-\ga}\dtx^k\widetilde R, \al^{-\ga}\dtx^k(\wt D_t w-\al^\ga a_{w_t})\big)
+\big(\tN_0(\ta^{1/2}\al^{-\ga}|\cT'|^{-1/2}\dtx^k w), \ta^{1/2}\al^{-\ga}|\cT'|^{-1/2}\dtx^k\dt w\big).
\end{split}
\]
Using commutators to bring some cancellations, one derives
\begin{equation}\label{I2+I6}
\begin{split}
I_1+I_5
=&-\big(\al^{-\ga}\dtx^k [u_1\dtx (\wt D_t w-\al^\ga a_{w_t})],\,\al^{-\ga}\dtx^k(\wt D_t w-\al^\ga a_{w_t})\big)\\
&-\big(\al^{-\ga}[\dtx^k, \ta]|\cT'|^{-1}\tN_0  w, \al^{-\ga}\dtx^k(\wt D_t w-\al^\ga a_{w_t})\big)
\\
& -\big(\al^{-\ga}\ta[\dtx^k, |\cT'|^{-1}]\tN_0w, \al^{-\ga}\dtx^k(\wt D_t w-\al^\ga a_{w_t})\big)+\big(\al^{-\ga}\dtx^k\widetilde R, \al^{-\ga}\dtx^k(\wt D_t w-\al^\ga a_{w_t})\big)
\\
&-\big(\tN_0(\ta^{1/2}\al^{-\ga}|\cT'|^{-1/2}\dtx^k w), \ta^{1/2}\al^{-\ga}|\cT'|^{-1/2}\dtx^k (u_1\dtx w)\big)\\
&+\big(\tN_0(\ta^{1/2}\al^{-\ga}|\cT'|^{-1/2}\dtx^k w), \ta^{1/2}\al^{-\ga}|\cT'|^{-1/2}\dtx^k (\al^{\ga}a_{w_t})\big)\\
& +\big([\tN_0, \ta^{1/2}]\al^{-\ga}|\cT'|^{-1/2}\dtx^k w, \ta^{1/2}\al^{-\ga}|\cT'|^{-1/2}\dtx^k(\wt D_t w-\al^\ga a_{w_t})\big)\\
& +\big(\ta^{1/2}[\tN_0, \al^{-\ga}|\cT'|^{-1/2}] \dtx^k w, \ta^{1/2}\al^{-\ga}|\cT'|^{-1/2}\dtx^k(\wt D_t w-\al^\ga a_{w_t})\big)\\
:=&I_{11}+I_{12}+\cdots+I_{18},
\end{split}
\end{equation}
where careful weighted commutators estimates are needed. We  firstly deal with all these terms except  for $I_{14}$, and the  term $I_{14}$ (which involves $\widetilde R$) is handled later.
\bigskip

\noindent - $I_{11}$ estimate. Direct computations and integration by parts lead to
\[
\begin{split}
I_{11}=&-\big(\al^{-\ga}[\dtx^k,  u_1]\dtx (\wt D_t w-\al^\ga a_{w_t}),\,\al^{-\ga}\dtx^k(\wt D_t w-\al^\ga a_{w_t})\big)\\
&-\big(\al^{-\ga}u_1\dtx\dtx^k (\wt D_t w-\al^\ga a_{w_t}),\,\al^{-\ga}\dtx^k(\wt D_t w-\al^\ga a_{w_t})\big)\\
=&-\big(\al^{-\ga}[\dtx^k,  u_1]\dtx(\wt D_t w-\al^\ga a_{w_t}),\,\al^{-\ga}\dtx^k(\wt D_t w-\al^\ga a_{w_t})\big)\\
&+\f12\big(\dtx(\al^{-2\ga}u_1)\dtx^k (\wt D_t w-\al^\ga a_{w_t}),\, \dtx^k(\wt D_t w-\al^\ga a_{w_t})\big),
\end{split}
\]
where we recall $u_1$ from \eqref{Dt transform}.  Noticing that
\[
\dtx \wt{\underline v}=\dtx \big(v(t, x, \zeta)\circ \cT\big)=\dtx x \big( \wt{\dx v}|_{\G_t}+(\dtx x)^{-1}\dtx \tzeta\,\wt{\p_z v}|_{\G_t}\big)
\]
where \eqref{na u transform} is applied to $\wt{\p v}$.  So we count the derivatives and apply  Lemma \ref{embedding}, Corollary \ref{dx x estimate} and Corollary \ref{estimate of x z}  to obtain
\[
\begin{split}
|I_{11}|&\le C\big( |u_1|_{W^{k,\infty}}+\big|\al^{2\ga}\dtx(\al^{-2\ga}u_1)\big|_{L^\infty}\big)\big|\al^{-\ga}(\wt D_t w-\al^\ga a_{w_t})\big|_{H^k}\\
&\le Q(\cT', w, k)\big(\big| \al^{-\ga-2}\tv-a_v\big|_{H^{k+1}}+\|a_v\|_{L^\infty}+\big|\al^{-(\ga+3)/2}\dt x\big|_{H^{k+1}}\big)\big|\al^{-\ga}\wt D_t w-a_{w_t}\big|^2_{H^k}.
\end{split}
\]
Besides,   using Proposition \ref{dtT estimate} and Proposition \ref{v estimate higher order}, we   arrive at the following estimate:
\[
|I_{11}|\le Q(\cT', \na\Phi, \ttze, w, k)\big(1+\big|\al^{-\ga}\wt D_t w-a_{w_t}\big|_{H^{k-1}}+|a_{\zeta_t, i}|_{L^\infty}\big)\big|\al^{-\ga}\wt D_t w-a_{w_t}\big|^2_{H^k}.
\]

\medskip

\noindent - $I_{12}$ estimate.
We have by \eqref{T prime -1}, Lemma \ref{embedding}, Corollary \ref{dx x estimate} and Corollary \ref{estimate of x z} that
\[
\begin{split}
|I_{12}|&\le \big|\al^{-\ga}[\dtx^k, \ta]|\cT'|^{-1}\tN_0w\big|_{L^2}\big|\al^{-\ga}\dtx^k(\wt D_t w-\al^\ga a_{w_t})\big|_{L^2}\\
&\le Q(\cT', w, 1) \big|\al^{-\ga-2}\ta\big|_{W^{k,\infty}}\big|\al\tN_0w\big|_{H^{k-1}}\big|\al^{-\ga}\wt D_t w-a_{w_t}\big|_{H^k}.
\end{split}
\]

Using Proposition \ref{a estimate}, we obtain 
\[
\begin{split}
|I_{12}|
&\le Q(\cT', w, 1) \big(\big|\al^{-2\ga-3}\ta-\dtz a_p\big|_{H^{k+1}}+| a_{p, i}|_{W^{k+1,\infty}_\tz}\big)\big|\al w\big|_{H^{k}}\big|\al^{-\ga}\wt D_t w-a_{w_t}\big|_{H^k}\\
&\le Q(\cT', \na\Phi, \tP, w, k) \big(1+\big|\al^{-\ga}\wt D_t w-a_{w_t}\big|_{H^{k-1}}+|a_{\zeta_t, i}|^2_{L^\infty}\big)\big|\al^{-\ga}\wt D_t w-a_{w_t}\big|_{H^k}.
\end{split}
\]
\medskip

\noindent - $I_{13}$ estimate.
Similar arguments as for $I_{12}$, we arrive at 
\[
\begin{split}
I_{13}&\le Q(\cT', w, k-1)  \big|\al^{-2\ga-3}\ta\big|_{L^\infty}\big|\al w\big|_{H^{k}}\big|\al^{-\ga}\wt D_t w-a_{w_t}\big|_{H^k}\\
&\le Q(\cT', \na\Phi, \ttze, \tP, w, k) \big(1+\big|\al^{-\ga}\wt D_t w-a_{w_t}\big|_{L^2}+|a_{\zeta_t, i}|^2_{L^\infty}\big)\big|\al^{-\ga}\wt D_t w-a_{w_t}\big|_{H^k}
\end{split}
\]
where Proposition \ref{a estimate} and Proposition \ref{v estimate higher order}  are applied.
\medskip

\noindent - $I_{15}$ estimate.  Similarly as for $I_{11}$, we have
\[
\begin{split}
I_{15}=&-\f12\big(\tN_0(\ta^{1/2}\al^{-\ga}|\cT'|^{-1/2}\dtx^k w), \ta^{1/2}\al^{-\ga}|\cT'|^{-1/2}[\dtx^k, u_1]\dtx w\big)\\
&+\f12\big(\tN_0(\ta^{1/2}\al^{-\ga}|\cT'|^{-1/2}\dtx^k w), u_1\dtx(\ta^{1/2}\al^{-\ga}|\cT'|^{-1/2})\dtx^k w\big)\\
&+\f12\big([u_1\dtx, \tN_0](\ta^{1/2}\al^{-\ga}|\cT'|^{-1/2}\dtx^k w), \ta^{1/2}\al^{-\ga}|\cT'|^{-1/2}\dtx^k w)\big).
\end{split}
\]
Applying Lemma \ref{basic product}, Lemma \ref{basic DN estimate}, Corollary \ref{dx x estimate}, Corollary \ref{estimate of x z}  and Proposition 3.18 \cite{Lannes} (for the standard commutator estimate involving $\tN_0$ in the third term), we obtain
\[
\begin{split}
|I_{15}|
\le & C\big|\ta^{1/2}\al^{-\ga}|\cT'|^{-1/2}\dtx^k w\big|_{H^{1/2}}\big(\big|\ta^{1/2}\al^{-\ga}|\cT'|^{-1/2}[\dtx^k, u_1]\dtx w\big|_{H^{1/2}}\\
&+\big|u_1\dtx(\ta^{1/2}\al^{-\ga}|\cT'|^{-1/2})\dtx^k w\big|_{H^{1/2}}\big)+C |u_1|_{W^{1,\infty}}\big|\ta^{1/2}\al^{-\ga}|\cT'|^{-1/2}\dtx^k w\big|^2_{L^2}\\
 \le & C\big|\al^{-\ga-3/2}\ta^{1/2}|^2_{W^{1,\infty}}\big(1+\big|1-\al^{1/2}|\cT'|^{-1/2}\big|^2_{H^2}\big)|u_1|_{H^{k+1/2}}
\big|\al w\big|^2_{H^{k+1/2}}.
\end{split}
\]
Moreover, direct computations under Assumption \ref{assump on a}  and applying Proposition \ref{a estimate}, Corollary \ref{dx x estimate},  Corollary \ref{estimate of x z} , one has
\[
\big|\al^{-\ga-3/2}\ta^{1/2}\big|_{W^{1,\infty}}\le Q(\cT', \na\Phi, \ttze, \tP, w, 1) \big(1+\big|\al^{-\ga}\wt D_t w-a_{w_t}\big|_{L^2}+|a_{\zeta_t}|^2_{L^\infty}\big)
\]
and
\[
\big|1-\al^{1/2}|\cT'|^{-1/2}\big|_{H^2}\le Q(\cT', w, 1),
\]
which together with Proposition  \ref{dtT estimate}, Proposition \ref{dt Ka estimate} and Proposition \ref{v estimate higher order}  results in
\[
|I_{15}|\le Q(\cT', \na\Phi, \ttze, \tP, w, k+1/2) \big(1+\big|\al^{-\ga}\wt D_t w-a_{w_t}\big|_{H^{k-3/2}}+|a_{\zeta_t, i}|_{L^\infty}\big).
\]

\medskip

\noindent - $I_{16}$ estimate. Similarly as in the previous estimate, we obtain
\[
\begin{split}
|I_{16}|
 & \le C\big|\ta^{1/2}\al^{-\ga}|\cT'|^{-1/2}\dtx^k w\big|_{H^{1/2}}\big|\ta^{1/2}\al^{-\ga}|\cT'|^{-1/2}\dtx^k (\al^{\ga}a_{w_t})\big|_{H^{1/2}}\\
& \le Q(\cT', \na\Phi, \ttze, \tP, w, k+1/2) \big(1\big|\al^{-\ga}\wt D_t w-a_{w_t}\big|_{H^{k-1}}+|a_{\zeta_t, i}|_{L^\infty}\big)|a_{\zeta_t, i}|_{L^\infty}.
\end{split}
\]

\medskip

\noindent - $I_{17}$ estimate. We have
\[
\begin{split}
|I_{17}|&\le C\big|[\ta^{1/2}, \tN_0]\al^{-\ga}|\cT'|^{-1/2}\dtx^kw\big|_{L^2}\big||\cT'|^{-1/2}\ta^{1/2}\big|_{L^\infty}\big|\al^{-\ga}\wt D_t w-a_{w_t}\big|_{H^k}\\
&\le C\big|\al^{-\ga-3/2}\ta^{1/2}\big|_{W^{2,\infty}}\big|\al^{\ga+3/2}\al^{-\ga}|\cT'|^{-1/2}\dtx^kw\big|_{H^{1/2}}
\big|\al^{-1/2}\ta^{1/2}\big|_{L^\infty}\big|\al^{-\ga}\wt D_t w-a_{w_t}\big|_{H^k},
\end{split}
\]
where we use Proposition \ref{DN commutator weighted estimate} (1) for the commutator $[\ta^{1/2}, \tN_0]$ with $\ga_1=-\ga-3/2$,  $\ga_2=\ga+3/2$ and $g=\al^{-\ga}|\cT'|^{-1/2}\dtx^kw$. Therefore, we require also
\[
\ga+3/2=|\ga+3/2|< \min\{\f\pi{2\om}, \f\pi{2\om_r}\}.
\]
Similar arguments as above lead to the following estimate:
\[
|I_{17}| \le Q(\cT', \na\Phi, \ttze, \tP, w, 1) \big(1+\big|\al^{-\ga}\wt D_t w-a_{w_t}\big|_{L^2}+|a_{\zeta_t, i}|_{L^\infty}\big)|\al w|_{H^{k+1/2}}\big|\al^{-\ga}\wt D_t w-a_{w_t}\big|_{H^k}.
\]

\medskip

\noindent - $I_{18}$ estimate.
Applying similar analysis as above again, we derive
\[
\begin{split}
|I_{18}|&\le C \big|\al^{-\ga-3/2}\ta^{1/2}\big|_{L^\infty}\big|\al^{\ga+3/2}[\tN_0, \al^{-\ga}|\cT'|^{-1/2}]\dtx^k w\big|_{L^2}\big|\al^{-1/2}\ta^{1/2}\big|_{L^\infty}\big|\al^{-\ga}\wt D_t w-a_{w_t}\big|_{H^k}\\
&\le   Q(\cT', \na\Phi, \ttze, \tP, w, 1) \big(1+\big|\al^{-\ga}\wt D_t w-a_{w_t}\big|_{L^2}+|a_{\zeta_t, i}|_{L^\infty}\big)|\al w|_{H^{k+1/2}}\big|\al^{-\ga}\wt D_t w-a_{w_t}\big|_{H^k}.
\end{split}
\]

\bigskip

{\bf Step 2. The remainder part: Estimates for $I_{14}$. } Recalling from \eqref{new w eqn} that
\[
\widetilde R=R_1\circ \cT+R_2\circ \cT-\wt D_t(\al^\ga a_{w_t})\qquad\hbox{on}\quad \G_t
\]
where $R_1$, $R_2$ and $\wt D_t(\al^\ga a_{w_t})$ are expressed in \eqref{J eqn 1},\eqref{R2} and  \eqref{Dt al a w}.

Therefore, we split $I_{24}$ into the following three parts:
\[
\begin{split}
I_{14}=&\big(\al^{-\ga}\dtx^k\wt R_1, \al^{-\ga}\dtx^k(\wt D_t w-\al^\ga a_{w_t})\big)+\big(\al^{-\ga}\dtx^k\wt R_2, \al^{-\ga}\dtx^k(\wt D_t w-\al^\ga a_{w_t})\big)\\
&-\big(\al^{-\ga}\dtx^k\wt D_t(\al^\ga a_{w_t}), \al^{-\ga}\dtx^k(\wt D_t w-\al^\ga a_{w_t})\big)\\
:=&B_1+B_2+B_3.
\end{split}
\]

\medskip
\noindent - Term $B_1$. Checking term by term, we know directly from \eqref{J eqn 1} that  the terms  in $R_1$ are always like
\[
(\partial v)^2\partial n_t,  \partial^2 v\partial v,  \partial^2P\partial n_t.
\]
Except $\p^2\tzeta$ parts with $\al^{-1}$, we know that
 $R_1\circ \cT|_{\G_t}$ contains  terms like
 \[
\al^{-3}(\partial \tv)^2, \al^{-3}\partial^2\tv\partial \tv, \al^{-3}\partial^2\tP.
 \]

As a result,  applying Proposition \ref{CM estimate}, Corollary \ref{dx x estimate}, Corollary \ref{estimate of x z} and Corollary \ref{zeta to w estimate} similarly as before and remembering that $\ga>-1$, we obtain for $k\ge 1$  that
\[
\begin{split}
&|B_1|\le C\big|\al^{-\ga}\wt R_1\big|_{H^k}\big|\al^{-\ga}\wt D_t w-a_{w_t}\big|_{H^k}\\
&\le  Q(\cT', w, k)\big|\al^{-1}\tzeta\big|_{H^{k+2}}
\big(\big|\al^{-\ga-3}(\partial \tv)^2\big|_{H^{k}}+\big|\al^{-\ga-3}\partial^2\tv\partial \tv\big|_{H^{k}}+\big|\al^{-\ga-3}\partial^2\tP\big|_{H^{k}}\big)\big|\al^{-\ga}\wt D_t w-a_{w_t}\big|_{H^k}\\
&\le Q(\cT', w, k)\big[\big|\al^{-\ga-2}\tv-a_v\big|^2_{H^{k+2}}+ \big|\al^{-2\ga-4}\tP-a_p\big|_{H^{k+2}}+|\al^{\ga+1}\chi_l|_{H^k}\big(|a_v|_{W^{k+2, \infty}}+|a_{p, i}|_{W^{k+2, \infty}_\tz}\big)\big]\times\\
&\quad \big|\al^{-\ga}\wt D_t w-a_{w_t}\big|_{H^k}.
\end{split}
\]

Using Proposition \ref{tPvv estimate}, Proposition \ref{v estimate higher order} and substituting the estimates above into $B_1$ estimate,  we conclude
\[
\begin{split}
|B_1|&\le Q(\cT', \na\Phi, \ttze, \tP, w, k+1/2)\big(1++\big|\al^{-\ga}\wt D_t w-a_{w_t}\big|_{H^k}+|a_{\zeta_t, i}|_{L^\infty}\big)\big|\al^{-\ga}\wt D_t w-a_{w_t}\big|_{H^k}
\end{split}
\]
\medskip

\noindent - Term $B_2$. Similarly as  above, one has from \eqref{R2}   the following estimates for $k\ge 1$:
\[
\begin{split}
|B_2|&\le C\big|\al^{-\ga}\wt R_2\big|_{H^k} \big|\al^{-\ga}\wt D_t w-a_{w_t}\big|_{H^k}\\
&\le Q(\cT', w, k)\big(\big|\al^{-\ga-2}\tv\big|^2_{H^{k+2}}+ \big|\al^{-2\ga-4}\tP\big|_{H^{k+2}}+\big|\al^{-\ga-1}\dtz\tilde u_p\big|_{H^{k}}
\big)\big|\al^{-\ga}\wt D_t w-a_{w_t}\big|_{H^k}.
\end{split}
\]

Now it remains  to deal with $u_p$ term, which is also a higher-order-derivative term since system \eqref{u1 system} involves $\p^3P$ and $\cT'''$ at $\G_b$.  Therefore, we need a close look at this boundary term.

In fact,  transforming system \eqref{u1 system} directly to an equivalent system in $\cSt$ as system \eqref{tPvv eqn} for $\tP$ and applying Lemma \ref{trace}, Proposition \ref{CM estimate}, Proposition \ref{Weighted elliptic estimate on S} and Corollary \ref{estimate of x z} as before, we have
\begin{equation}\label{up estimate mid}  
\begin{split}
&\big|\al^{-\ga-1}\dtz\tilde u_p\big|_{H^{k}}\le C\big\|\al^{-\ga-1}\tilde u_1\big\|_{H^{k+3/2}}\\
&\le  Q(\cT',  w, k)\big[\big\|\al^{-\ga-2}\p\wt \kappa_\cH\p\tP\big\|_{H^{k-1/2}}+\big\|\al^{-\ga-2}\p\wt\kappa_\cH\p^2\tP\big\|_{H^{k-1/2}}+\big\|\al^{-\ga-2}\wt \kappa_\cH\p^2\tP\big\|_{H^{k-1/2}}\\
&\quad 
+\big\|\wt\kappa_\cH \al^{-\ga-2}\p(\p\tv\p\tv)\big\|_{H^{k-1/2}}+\big\|\al^{-\ga-3}\p\tv\p\tv\big\|_{H^{k+3/2}}
+\big\|\al^{-\ga-3}\tP\big\|_{H^{k+5/2}}+\big|\al^{-\ga-3}\p \tv\p\tv\big|_{H^{k+1}}\\
&\quad +\big|\al^{-\ga-2}\wt\kappa_\cH\p^2 \tP\big|_{H^k(\G_b)}+\big|\al^{-\ga-2}\wt\kappa_\cH\p \tP\big|_{H^k(\G_b)}
+\big(\big|\al^{-1}\cT'-T_{1,c}\big|_{H^{k+2}(\G_b)}+|T_{1,c}|_{L^\infty}\big)\times\\
&\quad\big|\al^{-\ga-3}\tP\big|_{H^{k+3}(\G_b)}
\big]
\end{split}
\end{equation}
where we omit the detailed estimates for $n_t$ and $\cT$ and only keep track of the highest-order term of $\cT$ on $\G_b$. 

Recalling the weights  for $w=\wt \kappa$, $\tv$ and $\tP$ and taking the following estimate of $\tv$ for example, we obtain
\[
\begin{split}
&\big\|\al^{-\ga-3}\dtx\tv\dtx\tv\big\|_{H^{k+3/2}}\\
&\le \big\|\al^{\ga+1}\big(\al^{-\ga-2}\dtx\tv-a_{v, 1}\big)\big(\al^{-\ga-2}\dtx\tv-a_{v, 1}\big)\big\|_{H^{k+3/2}}+C\big\|\al^{\ga+1}\al^{-\ga-2}\dtx\tv  a_{v, 1}\big\|_{H^{k+3/2}}\\
&\quad +C\|\al^{\ga+1}( a_{v, 1})^2\|_{H^{k+3/2}}\\
&\le Q(\cT', \na\tPhi, \ttze, w, k+1/2)\big(1+\big|\al^{-\ga}\wt D_t w-a_{w_t}\big|_{H^k}+|a_{\zeta_t}|^2_{L^\infty}\big)
\end{split}
\]
where Proposition \ref{v estimate higher order} is applied and $a_{v, 1}$ is defined in Lemma \ref{minus norm dx}.

As a result, thanks to Corollary \ref{higher order local}, the estimates for the other terms above in the right side of \eqref{up estimate mid} follow in a similar way. 

\medskip

Consequently, going back to the estimate for $B_2$ and applying Proposition \ref{tPvv estimate}, Proposition \ref{v estimate higher order} and Proposition \ref{dt Ka estimate}, we obtain
\[
\begin{split}
|B_2|\le  Q(\cT', \na\tPhi, \ttze,  \tP, w, k+1/2)\big(1+\big|\al^{-\ga}\wt D_t  w-a_{w_t}\big|^2_{H^k}+|a_{\zeta_t, i}|^2_{L^\infty}\big)\big|\al^{-\ga}\wt D_t  w-a_{w_t}\big|_{H^k}.
\end{split}
\]

\medskip

Besides, for term $B_3$, we have immediately the following estimate by  Proposition \ref{CM estimate} QQ:
\[
\begin{split}
|B_3|&\le C\big(\big|\al^{-1-(\ga+1)/2}\tv\big|_{W^{k,\infty}}+\big|\al^{-1-(\ga+1)/2}\p_t x\big|_{W^{k,\infty}}\big)\big|\al^{(\ga+1)/2}a_{\zeta_t}\big|_{L^2}\big|\al^{-\ga}\wt D_t  w-a_{w_t}\big|_{H^k}\\
&\le Q(\cT', \na\tPhi, \tP, w, k+1/2)\big(1+\big|\al^{-\ga}\wt D_t  w-a_{w_t}\big|^2_{H^k}+|a_{\zeta_t, i}|^2_{L^\infty}\big)\big|\al^{-\ga}\wt D_t  w-a_{w_t}\big|_{H^k}
\end{split}
\]

Summing up these estimates above, we finally conclude that
\[
\begin{split}
|I_{14}|\le& Q(\cT', \na\tPhi, \tP, w, k+1/2)\big(1+\big|\al^{-\ga}\wt D_t w-a_{w_t}\big|^2_{H^{k-1}}+\big|\al^{-\ga-2}\ttze-a_{\zeta_t}\big|^2_{L^2}
+|a_{\zeta_t, i}|^2_{L^\infty}\big)\times\\
& \big|\al^{-\ga}\wt D_t  w-a_{w_t}\big|_{H^k}.
\end{split}
\]

\bigskip

{\bf Step 3. Estimates for $I_2, I_3, I_4$.} We  handle these terms one by one.

To begin with, we have by Proposition \ref{CM estimate}, Proposition \ref{dtT estimate},  Proposition \ref{a estimate} and Proposition \ref{dt Ka estimate} that
\[
\begin{split}
|I_2|&\le \big|\al^{-\ga-1/2}|\cT'|^{-1/2}\ta^{1/2}\dt\big((\al^{-2\ga-3}\ta\big)^{1/2}|\cT'|^{-1/2}\big)\big|_{L^\infty}\big|\al\dtx^k w\big|^2_{L^2}\\
&\le C\big[\big|\al^{-2\ga-3}\dt\ta\big|_{L^\infty}+\big|\al^{-2\ga-3}\ta\big|_{L^\infty}\big|\al^{1/2}|\cT'|^{-5/2}\dt(|\cT'|^2)\big|_{L^\infty}\big]\big|\al w\big|^2_{H^k}\\
&\le Q(\cT', \na\tPhi, \ttze, \tP, w, 1)\big(1+\big|\al^{-\ga}\wt D_t w- a_{w_t}\big|_{H^2}
+|a_{\zeta_t, i}|_{L^\infty}\big)\big|\al w\big|^2_{H^k}
\end{split}
\]
Meanwhile, the estimate for $I_4$ follows similarly as above:
\[
|I_4|\le Q(\cT', \na\tPhi, \ttze, \tP, w, 2)\big(1+\big|\al^{-\ga}\wt D_t w- a_{w_t}\big|_{H^3}
+|a_{\zeta_t, i}|_{L^\infty}\big)\big|\al w\big|^2_{H^{k+1/2}}.
\]

On the other hand, $I_3$ can be handled in a similar way as $I_{11}$ and $I_{16}$:
\[
\begin{split}
I_3=&\big(\ta^{1/2}\al^{-\ga}|\cT'|^{-1/2}\dtx^k(\wt D_t w-\al^\ga a_{w_t}), \ta^{1/2}\al^{-\ga}|\cT'|^{-1/2}\dtx^k w\big)\\
&-\big(\ta^{1/2}\al^{-\ga}|\cT'|^{-1/2}\dtx^k(u_1\dtx w), \ta^{1/2}\al^{-\ga}|\cT'|^{-1/2}\dtx^k w\big)\\
&+\big(\ta^{1/2}\al^{-\ga}|\cT'|^{-1/2}\dtx^k(\al^\ga a_{w_t}), \ta^{1/2}\al^{-\ga}|\cT'|^{-1/2}\dtx^k w\big),
\end{split}
\]
so summing up these estimates above leads to 
\[
\begin{split}
|I_2|+|I_3|+|I_4|\le Q(\cT', \na\tPhi, \ttze,  \tP, w, 2)\big(1+\big|\al^{-\ga}\wt D_t w-a_{w_t}\big|_{H^k}+|a_{\zeta_t, i}|_{L^\infty}
+\big|\al w\big|_{H^{k}}\big)\big|\al w\big|_{H^{k}}
\end{split}
\]
\medskip

{\bf Step 4. The lower-order energy.}  We still need to deal with the lower-order energy $E_0(t)$. 

For the first term of $\cT'$, one has by applying Remark \ref{dt T corner}, Proposition \ref{dtT estimate} and Proposition \ref{dt Ka estimate} that
\[
\f12\f d{dt}\big|\al^{-1}\cT'-T_{1,c}\big|^2_{L^2}=\big(\al^{-1}\p_t\cT', \al^{-1}\cT'-T_{1,c}\big)\le Q(\cT', \ttze,  a_{\zeta_t}, w, 1).
\]

For $\ttze$ term, recalling \eqref{zeta tt eqn} and dealing all the terms as above lead to 
\[
\f12\f d{dt}\big|\al^{-\ga-2}\ttze-a_{\zeta_t}\big|^2_{L^2}=\big(\p_t(\al^{-\ga-2}\ttze), \al^{-\ga-2}\ttze-a_{\zeta_t}\big) \le Q(\cT', \na\tPhi, \ttze, \tP,  w,  1).
\]

For $\na\tPhi$ term, going back directly to  system \eqref{tPhi system} and taking $\dt$, one obtains immediately the following system 
\[
\begin{cases}
\Delta \dt\dtz\tPhi=0\qquad \hbox{in}\quad \cSt,\\
\dt\dtz\tPhi\big|_{\G_t}=-(\dtx\dt x)\wt{\dt\zeta}-(\dtx x)\dt\ttze,\quad \dt\dtz\tPhi\big|_{\G_b}=0.
\end{cases}
\]
where one can resort to \eqref{zeta tt eqn} for the expression of $\dt \ttze$.  Applying Proposition \ref{Dirichlet system estimate} and using similar arguments as above lead to directly 
\[
\big\|\al^{-\ga-3}\dt\dtz\tPhi\big\|_{L^2}\le Q(\cT',  \na\tPhi, \ttze,   \tP, w, 1).
\]

As a result, one has the following estimate
\[
\begin{split}
\f12\f d{dt}\big\|\al^{-\ga-3}\dtz\tPhi-a_{\Phi_z}\big\|^2_{L^2}&=\big(\al^{-\ga-3}\dt\dtz\tPhi-\dt a_{\Phi_z}, \al^{-\ga-3}\dtz\tPhi-a_{\Phi_z}\big)\\
&\le \big\|\al^{-\ga-3}\dt\dtz\tPhi\big\|_{L^2}\big\|\al^{-\ga-3}\dtz\tPhi-a_{\Phi_z}\big\|_{L^2}\\
&\le Q(\cT', \na\tPhi, \ttze,  \tP, w, 1).
\end{split}
\]
where \eqref{time rate} is applied.

Meanwhile, for $\tP$ term, recalling \eqref{time rate} again and using similar arguments as for $\tPhi$ above show that 
\[
\f12\f d{dt}\big\|\al^{-2\ga-4}\tP-a_p\big\|^2_{L^2}\le\big\|\al^{-2\ga-4}\dt\tP\big\|_{L^2}\big\|\al^{-2\ga-4}\tP-a_p\big\|_{L^2}\le Q(\cT',\na\tPhi,  \ttze,  \tP, w, 1).
\]

Consequently, using Remark \ref{dt T corner} and \eqref{time rate} again we  conclude for the lower-order energy part that
\[
\f12\f d{dt}E_0(t)\le Q(\cT', \na\tPhi,\ttze,  \tP, w,  1).
\]

\bigskip

 {\bf Step 5. The end of  the energy estimate.} 
Summing up  all these estimates from Step 1 to Step 4, we arrive at the following energy estimate:
\begin{equation}\label{energy ineq}
\frac 12 \frac d{dt}\cE(t)\le Q_k(t),
\end{equation}
where 
\[
Q_k(t)=Q\big(E_0(t), \big|\al w|^2_{H^{k+1/2}}, \big|\al^{-\ga}\wt D_t w-a_{w_t}\big|^2_{H^k}\big).
\]

To close the energy estimate, we still need to consider the relationship between $\big|\al w\big|_{H^{k+1/2}}$ and the higher-order term of $w$ in $E_{k,\ga}$.  
\begin{lemma}\label{equiv energy} Let integer $k\ge 2$. Then there hold
\[
\big|\al w|^2_{H^{k+1/2}}+\big|\al^{-\ga}\wt D_t w-a_{w_t}\big|^2_{H^k}\le C\cE(t),
\]
and
\[
\cE(t)\le  Q_k(t).
\]
\end{lemma}
\begin{proof}(Proof of the lemma) Comparing all the terms in $Q_k(t)$ with those in $E_{k,\ga}(t)$, we know that the key point here lies in the estimate of $\big|\al w\big|_{H^{k+1/2}}$.

We firstly prove the second inequality. In fact,   applying Lemma \ref{basic DN estimate} shows directly for $l\le k$ that 
\[
\begin{split}
c_1 \big|\ta^{\frac12}\alpha^{-\ga}|\cT'|^{-\frac12}\dtx^l w\big|^2_{H^{1/2}}&\le
\big|\ta^{\frac12}\alpha^{-\ga}|\cT'|^{-\frac12}\dtx^l w\big|^2_{L^2}
+\big(\tN_0(\ta^{\frac12}\alpha^{-\ga}|\cT'|^{-\frac12}\dtx^l w), \ta^{\frac12}\alpha^{-\ga}|\cT'|^{-\frac12}\dtx^lw\big)\\
&\le C_1\big|\ta^{\frac12}\alpha^{-\ga}|\cT'|^{-\frac12}\dtx^l w\big|^2_{H^{1/2}}
\end{split}
\]
where $c_1, C_1$ are positive constants depending on $\cSt$.

Moreover, one has in a similar way as in the estimates before that 
\[
\begin{split}
\big|\ta^{\frac12}\alpha^{-\ga}|\cT'|^{-\frac12}\dtx^l w\big|_{H^{1/2}}&=\big|\al^{-\ga-3/2}\ta^{\frac12}\alpha^{3/2}|\cT'|^{-\frac12}\dtx^l w\big|_{H^{1/2}}\\
& \le \big|\al^{-\ga-3/2}\ta^{\frac12}\alpha^{1/2}|\cT'|^{-\frac12}\big|_{W^{1,\infty}}|\al w|_{H^{l+1/2}}\\
&
\le Q(\cT', \na\tPhi, \ttze, \tP, w, 1)|\al w|_{H^{l+1/2}}.
\end{split}
\]
As  a result, the right side of the desired inequality is  finished.

\medskip

For the first inequality,  noticing that $\al^{-\ga-3/2}\ta^{\frac12}\alpha^{1/2}|\cT'|^{-\frac12}$ is  also bounded below by Assumption \ref{assump on a} for $t=0$, Proposition \ref{CM estimate} and Proposition \ref{a estimate},  we have thanks to Lemma \ref{basic product} (3) that
\[
\begin{split}
|\al\dtx^l w|_{H^{1/2}} &\le C\big(a^{-1}_0, \big|\al^{-\ga-3/2}\ta^{\frac12}\alpha^{1/2}|\cT'|^{-\frac12}\al\big|_{W^{1,\infty}}\big)\big|\al^{-\ga-3/2}\ta^{\frac12}\alpha^{1/2}|\cT'|^{-\frac12}\al\dtx^l w\big|_{H^{1/2}}\\
&\le  C\big|\ta^{\frac12}\alpha^{-\ga}|\cT'|^{-\frac12}\dtx^l w\big|_{H^{1/2}},
\end{split}
\]
where we use the bound in $\Lam(T, k)$ as well as the uniform bound $L$ from  Definition \ref{initial value bound} for the lower-order terms in the last inequality.
Therefore, the proof is finished.

\end{proof}

\bigskip

Now we can finally finish the proof of our theorem.  In fact, applying Lemma \ref{equiv energy} to the right side of \eqref{energy ineq}, we obtain
\[
\f12\f{d}{dt}\cE(t)\le P( \cE(t)).
\]
As a result,  the desired energy estimate is proved.
Moreover,  we see from this inequality  immediately that a small positive $T^*$ exists and depends only on the initial values.
\end{proof}

\begin{remark}\label{homo norm} The weighted norms can be better understood if we pull they back into the norms in he physical domain $\Om$. For the sake of simplicity, we only focus on the left corner and omit the details of computations, and the case for the right corner follows similarly. In fact, looking at the part of domain near the left corner, we know that  $|\cT'|\approx \al\approx r$ where $r$ is the radius with respect to the left corner point. 

Therefore, we have the following settings:

(1) For $\al w=\al \wt\kappa\in H^{k+1/2}$, direct transformation shows immediately that 
\[
r^{1/2}(r\p_r)^l\kappa\in L^2\quad\hbox{near the left corner},\quad l\le k,
\] 
which corresponds to the homogeneous norm defined in \cite{KMR}.

(2) For $\al^{-\ga}\wt D_t w-a_{w_t}\in H^k$, we only look at the $L^2$ norm for simplicity:
\[
r^{-1/2}(r^{-\ga}D_t\kappa-a_{w_t})\in L^2 \quad\hbox{near the left corner}.
\]

As a result, we show by the energy estimate that 
\[
\al^{-\ga-2}\tv-a_{v}\in H^{k+5/2}(\cS_0),\quad \al^{-2\ga-4}\tP-a_p\in H^{k+7/2}(\cS_0),
\] 
which means 
\[
r^{-1}(r^{-\ga-2}v-a_{v})\in L^2,\quad r^{-1}(r^{-2\ga-4}P-a_p)\in L^2 \quad\hbox{near the left corner}.
\]
Moreover, according to our estimates for time derivatives, we can also see that there are more weights in time derivatives, which means faster decay for time derivatives near the corners.
\end{remark}

\section*{Acknowledgements}
This work started from a discussion with Chongchun Zeng in 2019. The author would like to thank Chongchun Zeng, David Lannes and Chao Wang for fruitful discussions and helpful suggestions.  This work is supported by National Natural Science Foundation of China no.12071415 and no.12222116.

\end{document}